\documentclass[reqno]{amsart}
\usepackage{comment,amssymb,latexsym,upref,enumerate,fouridx}
\usepackage{amsmath}
\usepackage{mathrsfs,color}
\usepackage{centernot}
\usepackage[colorlinks,linkcolor=blue,citecolor=blue]{hyperref}
\usepackage{graphicx}

\allowdisplaybreaks

\numberwithin{equation}{section}

\theoremstyle{plain}
\newtheorem{thm}{Theorem}[section]
\newtheorem{lem}[thm]{Lemma}
\newtheorem{prop}[thm]{Proposition}

\theoremstyle{definition}

\theoremstyle{remark}
\newtheorem{rem}[thm]{Remark}

\newcommand{\ep}{\epsilon}

\newcommand{\betah}{\hat{\beta}}

\newcommand{\Jcch}{\check{\mathcal{J}}{}}

\newcommand{\chihu}{\underline{\hat{\chi}}{}}
\newcommand{\ghu}{\underline{\gh}{}}
\newcommand{\whu}{\underline{\wh}{}}
\newcommand{\vhu}{\underline{\vh}{}}

\newcommand{\zhu}{\underline{\zh}{}}
\newcommand{\hhu}{\underline{\hh}{}}
\newcommand{\Bhu}{\underline{\Bh}{}}
\newcommand{\Qhu}{\underline{\Qh}{}}
\newcommand{\Gammahu}{\underline{\hat{\Gamma}}{}}

\newcommand{\chih}{\hat{\chi}{}}

\newcommand{\Kttt}{\tilde{\Ktt}{}}
\newcommand{\gttt}{\tilde{\gtt}{}}

\newcommand{\Hct}{\tilde{\mathcal{H}}}

\newcommand{\Aft}{\tilde{\mathfrak{A}}}
\newcommand{\Mft}{\tilde{\mathfrak{M}}}
\newcommand{\Bft}{\tilde{\mathfrak{B}}}
\newcommand{\Cft}{\tilde{\mathfrak{C}}}
\newcommand{\Hft}{\tilde{\mathfrak{H}}}
\newcommand{\Dft}{\tilde{\mathfrak{D}}}

\newcommand{\Ct}{\tilde{C}}

\newcommand{\Bcv}{\boldsymbol{\Bc}}

\newcommand{\alphah}{\hat{\alpha}}
\newcommand{\Hchat}{\hat{\Hc}}
\newcommand{\Sigmah}{\hat{\Sigma}}

\newcommand{\Fvt}{\tilde{\Fv}}

\newcommand{\Icv}{\boldsymbol{\Ic}}

\newcommand{\gttb}{\bar{\gtt}}
\newcommand{\Kttb}{\bar{\Ktt}}

\input Tdef.def

\begin{document}

\title[Localized Big Bang Stability of Spacetime Dimensions $n\geq4$]{Localized Big Bang Stability of Spacetime Dimensions $n\geq4$}

\author[W.~Zheng]{Weihang Zheng}
\address{School of Mathematical Sciences\\
9 Rainforest Walk\\
Monash University, VIC 3800\\ Australia}
\email{weihang.zheng@monash.edu}

\begin{abstract}
    We prove the past nonlinear stability of the sub-critical Kasner-scalar field solutions to the Einstein-scalar field equations on a truncated cone domain in spacetime dimensions $n\geq4$. Our analysis demonstrates that the perturbed solutions are asymptotically pointwise Kasner, geodesically incomplete in the contracting direction and terminate at quiescent and crushing singularities characterized by the blow-up of curvature invariants. This work generalizes the result of Beyer-Oliynyk-Zheng in \cite{BOZ:2025} to all higher dimensional spacetimes.
\end{abstract}

\maketitle

\section{Introduction}
The study of the cosmological singularities, such as the big bang and black holes, remains an essential challenge in general relativity. According to Penrose's Strong Cosmic Censorship (SCC)\footnote{For a rough version, see \cite[Conjecture 17.1, 17.2]{Ringstrom-CauchyGR:2009}.}, generically the maximal globally hyperbolic development of the initial data is terminated because of the formation of some kind of singularities, e.g. the curvature invariants blowup. Furthermore, due to Hawking's singularity theorem in \cite{Hawking:1967,Hawking:1970}, there is a wide range of initial data for the Einstein's field equations such that the corresponding cosmological solutions eventually exhibit a singularity in the sense of the incompleteness of geodesics, and this result applies to any matter fields that satisfy the strong energy condition. Although the presence of the singularity is predicted by both SCC and Hawking's theorem, the nature of these singularities remains poorly understood, as the latter theorem relies on a proof by contradiction and only implies the boundedness of geodesic length.

In \cite{belinskii1970,Belinskii:1982}, Belinskii, Khalatnikov and Lifshitz (BKL) proposed two types of generic behavior of gravitational fields near cosmological singularities. Specifically, the dynamics of the spacetimes collapsing towards the singularity are either chaotic and oscillatory or quiescent, where we sometimes refer the former case to the \textit{mixmaster universe}, a term introduced in \cite{Misner:1969}. The chaotic and oscillatory case is particularly challenging to analyze, and rigorous results in this context remain limited; see, for example, \cite{Beguin2010,Beguin2023,li2024b,li2024a,Liebscher2013,Marshall:2025,Ringstrom2000,Ringstrom2001,Weaver2000}. In contrast, there are literature dating back to the last century suggesting some heuristics about the quiescent big bang formation. In light of the works in \cite{Barrow:1978,belinskiiKhalatnikov:1972,belinskii1970,Demaret:1985,lifshitz1963}, for certain exact solutions to the Einstein's equations coupled with specific types of matter fields, one may expect big bang formation with asymptotic behavior near the singularity and the key point is that the solutions satisfy the sub-critical condition defined in \cite{Fournodavlos_et_al:2023}.

The understanding of quiescent big bang formation saw its first major breakthrough in the work of Rodnianski and Speck in \cite{RodnianskiSpeck:2018b,RodnianskiSpeck:2018c,Speck:2018}, who demonstrated stable formation for initial data near Friedmann-Lema\^{i}tre-Robertson-Walker (FLRW) solutions to the Einstein-scalar field equations. This result was later extended to moderately spatially anisotropic initial data in \cite{RodnianskiSpeck:2022}. A remarkable end for this series of papers is the work \cite{Fournodavlos_et_al:2023} by Fournodavlos–Rodnianski–Speck (FRS), in which they proved the past nonlinear stability of the Kasner-scalar field metrics of entire the sub-critical regime. Another significant contribution is due to Groeniger–Petersen–Ringstr\"{o}m in \cite{Groeniger_et_al:2023}, where the authors established the formation of the quiescent big bang singularity of a broader range of initial data that does not have to be closed to a certain reference metric and this result applies to the Einstein-scalar field equations with a non-vanishing scalar potential. We would also like to mention that in \cite{AnHeShen:2025}, An-He-Shen generalized the result by FRS to solutions to the more complicated Einstein–Maxwell–scalar field–Vlasov system in the strong sub-critical regime.\footnote{The strong sub-critical regime is defined in \cite[Def. 1.1]{AnHeShen:2025}.} A more comprehensive discussion of other research developments connected to this topic is reserved for Section \ref{related-works}.

Among the aforementioned results concerning the stable big bang formation, the constant mean curvature (CMC) gauge is commonly employed to define the time function. An important feature of this choice is that it can synchronize the big bang singularity and enable us to describe the asymptotic behavior as $t\searrow0$. Also note that in \cite{Ringstrom2025}, Ringstr\"{o}m establishes a local existence theory for a class of CMC gauges for the Einstein-non-linear scalar field equations. However, using the CMC foliation results in infinite propagation speed\footnote{A more precise reason is that there is an elliptic equation for the lapse, see, e.g. \cite[Eq. (2.25)]{Fournodavlos_et_al:2023}.}, implying that the solutions depend non-locally on the initial data. Given the hyperbolic nature of Einstein’s field equations, imposing global proximity to spatially homogeneous models (such as FLRW or Kasner-scalar field metrics) is physically unnatural. This motivates the study of a localized stability theorem, formulated on a truncated cone domain.

This question was first partially answered by Beyer and Oliynyk in \cite{BeyerOliynyk:2024b}, where they proved the localized nonlinear stability of the FLRW solution to the Einstein-scalar field equations in spacetime dimensions $n\geq3$. We stress that in their work, a different time gauge was employed. Other than the CMC, they used the scalar field in the conformal picture as the time function, which can also synchronize the big bang singularity while preserves the hyperbolicity of the Einstein's equations. The method of using the scalar field as a time function also led to the localized stability result for FLRW solutions to the Einstein-Euler-scalar field system in \cite{BeyerOliynyk:2024c}. Later, Beyer-Oliynyk-Zheng (BOZ) extended the result in \cite{BeyerOliynyk:2024b}, proving localized past nonlinear stability of the entire sub-critical family of Kasner solutions in four spacetime dimensions.

A series of foundational work in string theory, see e.g. \cite{CANDELAS1985,GREEN1984,WITTEN1995,Yau1978}, indicates that our universe might have higher dimensions. They predict that there are extra dimensions which are curled up into tiny shapes, such as the Calabi-Yau manifolds. This serves as one of the motivation of the present article and the aim is to generalize the result by BOZ to spacetimes with dimensions $n\geq4$.

\subsection{The Einstein-scalar field equations}
In this article, we analyze the Einstein-scalar field equations on a manifold $(M,\gb)$ given by
\begin{align}
    \Rb_{ab} &= 2\nablab_a\varphi \nablab_b\varphi, \label{ESF.1} \\
    \Box_{\gb}\varphi &=0, \label{ESF.2}
\end{align}
where $\Rb_{ab}$ is the Ricci tensor of the metric $\gb_{ab}$, $\nablab$ is the Levi-Civita connection of $\gb_{ab}$ and $\Box_{\gb}=\gb^{ab}\nablab_a\nablab_b$ is the wave operator. Equation \eqref{ESF.1} can be rewritten in the more standard form $\Gb_{ab}=\Tb_{ab}$, where
\begin{align}
    \Gb_{ab} &= \Rb_{ab} - \frac{1}{2}\Rb \gb_{ab}, \label{Gbab} \\
    \Tb_{ab} &= 2\nablab_a\varphi \nablab_b\varphi - |\nablab\varphi|_{\gb}^2 \gb_{ab}, \label{Tbab}
\end{align}
are the Einstein tensor and the stress-energy tensor, respectively, and $\Rb$ denotes the scalar curvature of $\gb_{ab}$. Substituting \eqref{Gbab} and \eqref{Tbab} into $\Gb_{ab}=\Tb_{ab}$ yields\footnote{This form is useful for deriving the frame formalism for a general stress–energy tensor; see \eqref{cEEa}.}
\begin{equation} \label{ESF.3}
    \Rb_{ab} = \Tb_{ab} - \frac{1}{n-2}\gb^{cd}\Tb_{cd}\gb_{ab} = 2\nablab_a\varphi \nablab_b\varphi.
\end{equation}

\subsection{Kasner-scalar field spacetimes}
In accordance with the discussion in \cite[\S 1.1-1.2]{BeyerOliynyk:2024b}, the
\textit{Kasner-scalar field spacetimes} are determined by the following metric and scalar field
\begin{equation} \label{Kasner-metric}
    \gb_{ab} = t^{\frac{2}{n-2}}\gt_{ab} \quad \text{and} \quad \varphi = \sqrt{\frac{n-1}{2(n-2)}}\ln(\tau),
\end{equation}
respectively, on the spacetime manifold $M=\Rbb_{>0}\times\Tbb^{n-1}$ (see Section \ref{coor-frame} for notations conventions). Here,
\begin{equation} \label{conf-Kasner}
    \gt = -t^{r_0}dt\otimes dt + \sum_{\Lambda=1}^{n-1}t^{r_\Lambda}dx^{\Lambda}\otimes dx^{\Lambda} \quad \text{and} \quad \tau = t,
\end{equation}
where $\gt$ and $\tau$ are the conformal\footnote{The term 'conformal' comes from the fact that in \cite{BeyerOliynyk:2024b}, \{$\gt$,$\tau$\} are actually the solutions to the conformal Einstein-scalar field equations \cite[Eq.(1.16)-(1.17)]{BeyerOliynyk:2024b}.} Kasner metric and the conformal scalar field, that together determine a spatially homogeneous and anisotropic spacetime, cf. \cite[Eq.~(1.18),(1.19)]{BeyerOliynyk:2024b}. In these expressions, the constants $r_0$ and $r_\Lambda$ are called the \textit{conformal Kasner exponents} and are defined by
\begin{equation} \label{Kasner-exps-A}
    r_0 = \frac{1}{P}\sqrt{\frac{2(n-1)}{n-2}} - \frac{2(n-1)}{n-2} \AND r_\Lambda = \frac{1}{P}\sqrt{\frac{2(n-1)}{n-2}}q_\Lambda - \frac{2}{n-2},
\end{equation}
where  $0<P\leq\sqrt{(n-2)/(2(n-1))}$. The parameters $q_\Lambda$, known as the \textit{Kasner exponents}, satisfy the \textit{Kasner relations} 
\begin{equation} \label{Kasner-rels-A}
    \sum_{\Lambda=1}^{n-1}q_\Lambda = 1 \AND \sum_{\Lambda=1}^{n-1}q_\Lambda^2 = 1 - 2P^2. 
\end{equation}
A straightforward calculation using \eqref{Kasner-exps-A} shows that the \textit{Kasner relations} for $q_\Lambda$ are equivalent to the following conditions on $r_0$ and $r_\Lambda$:
\begin{equation} \label{Kasner-rels-B}
    \sum_{\Lambda=1}^{n-1} r_\Lambda = r_0 \AND \sum_{\Lambda=1}^{n-1} r_\Lambda^2 = (r_0+2)^2 - 4.
\end{equation}
From \eqref{Kasner-exps-A} and the bound $0 < P \leq \sqrt{(n-2)/(2(n-1))}$, we observe that 
\begin{equation} \label{r0-bound}
    r_0 \geq 0.
\end{equation}
Furthermore, the second condition in \eqref{Kasner-rels-B} implies
\begin{equation} \label{rA-bound}
    r_\Lambda < r_0 + 2.
\end{equation}

The Kasner spacetimes for which all exponents $q_\Lambda$ are equal are called the FLRW spacetimes. In this case, \eqref{Kasner-rels-A} and \eqref{Kasner-rels-B} imply that $r_0=r_\Lambda=0$, and that the conformal Kasner solution in \eqref{conf-Kasner} reduces to
\begin{equation} \label{FLRW-def}
    \gt = -dt\otimes dt + \sum_{\Lambda=1}^{n-1}dx^{\Lambda}\otimes dx^{\Lambda} \quad \text{and} \quad \tau = t,
\end{equation}
which determines a conformal, spatially homogeneous and isotropic spacetime.

In \cite{Fournodavlos_et_al:2023}, the authors define the sub-critical condition (which is also referred to as the stability condition) to be
\begin{equation*} \label{sub-cond-1}
    \max\limits_{\Omega < \Lambda}\{q_\Omega+q_\Lambda-q_\Gamma\} < 1,
\end{equation*}
for $\Omega,\Lambda,\Gamma=1,2,...,n-1$, cf. \cite[Eq.~(1.8)]{Fournodavlos_et_al:2023}. Using the relations in \eqref{Kasner-exps-A}, thid sub-critical condition can be equivalently expressed in terms of the \textit{conformal Kasner exponents} as
\begin{equation} \label{sub-cond-2}
    \max\limits_{\Omega < \Lambda}\{r_\Omega+r_\Lambda-r_\Gamma\} < r_0 + 2.
\end{equation}
This condition is crucial to our proof using the Fuchsian method in Section~\ref{Fuch-form}, particularly for \eqref{Bc-def}.

\begin{rem}
    It is noted both in \cite{Demaret:1985} and \cite[\S 1.5.1]{Fournodavlos_et_al:2023} that in the vacuum case, i.e. in the extreme case that $P=\sqrt{(n-2)/(2(n-1))}$ in \eqref{Kasner-exps-A}, the set of conformal Kasner exponents satisfying the condition \eqref{sub-cond-2} is non-empty only when the spatial dimension $(n-1)\geq10$, due to the constraints in \eqref{Kasner-rels-A}, while in the case where we have a non-vanishing scalar field, the sub-critical condition can be satisfied when $(n-1)\geq3$.
\end{rem}

\subsection{Related works} \label{related-works}
Beyond the works previously discussed, a substantial body of mathematical results addresses the big bang singularity. We begin by reviewing contributions that rely on symmetry assumptions. In \cite{chruscielStrongCosmicCensorship1990}, the authors proved the Strong Cosmic Censorship conjecture for the Einstein vacuum equations with polarized Gowdy symmetry. Subsequently, \cite{IsenbergMoncreif:1990} demonstrated that solutions under the same symmetry assumptions exhibit AVTD\footnote{See \cite[\S 1.3]{BeyerOliynyk:2024b} for the definition.} behavior. Later, Ringstr\"{o}m extended these results in \cite{ringstrom2008,ringstrom2009a}, proving SCC for Gowdy solutions with spatial topology $\Tbb^3$ without the polarization assumption. Related investigations are documented in \cite{Chrusciel:2003jj,ringstrom2005,Ringstrom2006}.

Other results concern the construction of singular solutions of Einstein’s equations with Gowdy symmetry that exhibit Kasner-like asymptotics near the singularity. The first result for analytic solutions was obtained by Kichenassamy and Rendall in \cite{kichenassamy1998}. This was followed by \cite{rendall2000}, which relaxed the analyticity assumption. Further singular solutions were constructed for spatial topologies $S^2\times S^1$ and $S^3$ in \cite{stahl2002}, under generalized wave gauges in \cite{ames2017}, and for self-gravitating fluid flows in \cite{beyer2017}. The Fuchsian method is central to most of these constructions.

The stability and AVTD behavior of Einstein vacuum solutions under polarized $\Tbb^2$-symmetry assumptions were established in \cite{ABIO:2022_Royal_Soc,ABIO:2022}. A recent contribution in \cite{Dong2026} proves the stability of Kasner singularities for the Einstein vacuum spacetime under polarized $U(1)$-symmetry, again utilizing the Fuchsian method. Parallel to stability analysis, there are also results concerning the construction of singular solutions. Analytic solutions for Einstein vacuum spacetimes under polarized $\Tbb^2$-symmetry were given in \cite{IsenbergKichenassamy:1999}, while the analyticity assumption was dropped for the half-polarized setting in \cite{ames2013a}. The construction of polarized and half-polarized $U(1)$-symmetric vacuum solutions was undertaken in \cite{choquet-bruhat2006,Isenberg_2002}, and topologically general $U(1)$-symmetric solutions were treated in \cite{choquet-bruhat2004}.

For results independent of symmetry assumptions, the earliest work was presented by Andersson and Rendall in \cite{andersson2001}, who specified the asymptotics near the singularity for the Einstein-scalar field and Einstein-stiff fluid systems under a real analyticity assumption. This result was extended to higher dimensions and other matter fields in \cite{damour2002}. The vacuum case was later examined by Klinger in \cite{klinger2015}. A related advance was made in \cite{FournodavlosLuk:2023}, where Fournodavlos and Luk dispensed with the real analyticity requirement; ; a localized version was subsequently established by Athanasiou and Fournodavlos in \cite{AthanasiouFournodavlos:2024}.

We also note the work of Fajman and Urban in \cite{FajmanUrban:2025}, where the nonlinear stability of the FLRW solution for the Einstein-scalar field system with homogeneous scalar field matter and a negative Einstein metric was established both to the future and the past. Stable big bang formation of the Einstein scalar-field Vlasov system under nonlinear perturbations was later proved in \cite{FajmanUrban:2024}. A similar result in $2+1$ dimensions, with the initial hypersurface diffeomorphic to a closed orientable surface of arbitrary genus, was subsequently provided by Urban in \cite{Urban:2024}.

Finally, Ringstr\"{o}m introduced a geometric notion of initial data on the singularity for the Einstein-scalar field equations in \cite{ringstrom2022a,ringstrom2022}, showing that previous notions constituted special cases. Moreover, it is established in \cite{Franco-Grisales2025} that such initial data for the Einstein–nonlinear scalar field equations in four spacetime dimensions admit a unique development, without assuming any symmetry or analyticity and allowing for arbitrary closed spatial topology. A comprehensive framework for the analysis of the Einstein–scalar field equations was also developed by Ringstr\"{o}m in \cite{ringstrom2017,ringstrom2021a,ringstrom2021}.

\subsection{A rough version of the main theorem}
The equivalence between the conformal Einstein-scalar field equations and the physical Einstein-scalar field equations, established in \cite[\S 1.1]{BeyerOliynyk:2024b}, justifies our formulation of the main result within the conformal framework. We consider a reference conformal Kasner-scalar field solution and study the evolution of sufficiently small perturbations of its initial data on a constant time hypersurface. The resulting solution to the Einstein-scalar field equations is then analyzed, with particular emphasis on its behavior as $t\searrow0$.

In contrast to the constant mean curvature (CMC) gauge, we employ the conformal scalar field $\tau$ as the time function, which allows us to establish stability without requiring global smallness conditions on the initial data. This localization is physically natural for Einstein's equations and represents an advancement beyond previous global-in-space stability results. The detailed statement appears in Theorem~\ref{main-thm}.

\begin{rem}
    As established in \cite[Prop. 5.8]{BeyerOliynyk:2024b}, for initial data sufficiently close to the background Kasner metric, the Einstein-scalar field equations admit a solution under a brief evolution in time. One can then identify a level surface on which the scalar field $\tau$ is constant. This procedure demonstrates the synchronizability of this particular choice of time function.
\end{rem}

\begin{thm}[Localized past stability of the sub-critical Kasner metrics in dimensions $n\geq4$]
    The conformal Kasner-scalar field solutions in the sub-critical regime in sapcetime dimensions $n\geq4$ are dynamically stable under small perturbations when restricted to a truncated cone domain.
    
    More precisely, given $\rho_0\in(0,L)$, $\ep_2\in(0,1)$ and $\rho_1>0$ satisfying condition \eqref{rho1-range}, sufficiently small perturbations of the Kasner initial data generates a solution to the Einstein-scalar field equations on the manifold $\Omega_{\Icv}$\footnote{See Section~\ref{Domains} for definitions of domains in $\Rbb^{n-1}$.} with $\Icv=(t_0,0,\rho_0,\rho_1,\ep_2)$. Moreover, the perturbed solutions exhibit the following characteristics: they are are asymptotically pointwise Kasner, display specific asymptotic behaviors, are $C^2$-inextendible and past timelike geodesically incomplete, terminate at a crushing singularity along the spacelike Cauchy hypersurfaces, satisfy the AVTD property and the curvature invariants blow up\footnote{For definitions of the terminologies, see Theorem~\ref{main-thm} for references.}.
\end{thm}

A global-in-space version of this stability result on the manifold $M_{0,t_0}$ can be formulated directly; see Remark~\ref{main-thm-v2} for details.

\begin{rem}
    It is worthwhile to mention a very recent work \cite{ringstrom2026} by Franco-Grisales and Ringstr\"{o}m, where they prove that the solutions constructed in \cite{Groeniger_et_al:2023} induce data on the singularity in the sense of \cite{ringstrom2022a}. In \cite[Thm. 22]{ringstrom2026}, their assumptions are relatively weak regarding gauge choices (for instance, the foliation is not required to be constant mean curvature). However, their result cannot be applied directly to the present work. Specifically, in Theorem~\ref{main-thm}, we establish that the rescaled conformal lapse $\alpha$ decays as (it is straightforward to transfer it back to the physical picture):
    \begin{equation*}
        \alpha = t^{\ep_1+\frac{r_0}{2}+(n-1)\Hchat}\alphah \bigl(1+\Ord_{H^{k-1}(\mathbb{B}_{\rhot_0})}(t^{\zeta})\bigr),
    \end{equation*}
    where $\alphah$ is positive and bounded and $\Hchat$ depends on spatial coordinates. In contrast, \cite[Eq. (36d)]{ringstrom2026} requires the physical lapse to remain near to $1$ for all time. Adapting our framework to fit the hypotheses of \cite[Thm. 22]{ringstrom2026} would therefore require a careful choice of a different gauge. We plan to address this question in future work.
\end{rem}

\subsection{Overview of the proof}
In this article, we present a frame formalism of the Einstein-scalar field equations on spacetimes with dimensions $n\geq4$. This formulation adopts the same gauge choices as those in the \textit{Lagrangian coordinates} formulation in \cite{BeyerOliynyk:2024b}, which allows us to derive a local-in-time existence result that ensures the propagation of the constraints. This frame formalism can be further transformed into the \textit{Fuchsian system}, yielding global-in-time existence, along with energy and decay estimates. These results for the \textit{Fuchsian system} on $\Omega_{\Icv}$ are then applied to establish the localized stable formation of a big bang singularity.

\bigskip

\noindent \underline{\textit{The frame formalism in spacetime dimensions $n\geq4$.}}
We begin by introducing a conformal transformation defined by
\begin{equation*}
    \gb_{\mu\nu} = e^{2\Phi}\gt_{\mu\nu}.
\end{equation*}
Within this conformal framework, we employ a zero-shift gauge, so the conformal metric takes the form
\begin{equation*}
    \gt= - \alphat^2 dt\otimes dt + \gt_{\Sigma \Omega}dx^\Sigma \otimes dx^\Omega,
\end{equation*}
where $\alphat$ denotes the lapse. The orthonormal frame $\{\et_a\}$ is constructed by first setting $\et_0=\frac{1}{\alphat}\del{t}$ and then evolving the remaining spatial frames $\{\et_A\}$ via Fermi-Walker transport:
\begin{equation*}
    \nablat_{\et_0}\et_A = -\frac{\gt(\nablat_{\et_0}\et_0, \et_A)}{\gt(\et_0, \et_0)}\et_0.
\end{equation*}
This choice preserves the orthonormality of the frame. We stress that from now on, all tensor fields are expressed relative to the conformal orthonormal frame $\{\et_a\}$. To formulate the Einstein's equations, we define the following variables originating from the commutators of the frame elements
\begin{align*}
    [\et_0,\et_A] &= \Ut_A \et_0 - (\Hct\delta_{AB} + \Sigmat_A{}^B)\et_B, \\
    [\et_A,\et_B] &= \Ct_A{}^C{}_B \et_C.
\end{align*}

Now, align with the work in \cite[\S 1.1]{BeyerOliynyk:2024b}, the conformal factor is chosen as
\begin{equation*}
    \Phi = \sqrt{\frac{2}{(n-2)(n-1)}}\,\varphi.
\end{equation*}
This leads to the conformal Einstein-scalar field equations:
\begin{align*}
    \Gt_{ab} &= \frac{1}{\tau}\nablat_a\nablat_b\tau, \\
    \Box_{\gt}\tau &= 0,
\end{align*}
where $\Gt_{ab}$ is the Einstein tensor of $\gt$ and $\tau$ is the conformal scalar field. Instead of using the constant mean curvature, we use $\tau$ to define the time function, that is
\begin{equation*}
    \tau = t.
\end{equation*}

In Section~\ref{conf-ESF} and ~\ref{time-gauge}, we rewrite the Einstein-scalar field equations into the following first-order system for the variables $\Wt=(\et_A^\Omega,\alphat,\Ct_{ABC},\Ut_A,\Hct,\Sigmat_{AB})^{\tr}$:
\begin{align}
    \del{t}\et_{A}^{\Omega} &= -\alphat(\Hct\delta_{A}^{B} + \Sigmat_{A}{}^{B})\et_{B}^{\Omega}, \label{EEa.1} \\
    \del{t}\Ct_A{}^C{}_B &= -2\alphat\et_{[A}(\Hct)\delta_{B]}^C - 2\alphat\et_{[A}(\Sigmat_{B]}{}^C) - 2\alphat\Hct \Ut_{[A}\delta_{B]}^C - 2\alphat\Ut_{[A}\Sigmat_{B]}{}^C \notag \\
        &\quad -\alphat\Hct\Ct_A{}^C{}_B + 2\alphat\Sigmat_{[A}{}^D \Ct_{B]}{}^C{}_D + \alphat\Ct_A{}^D{}_B\Sigmat_D{}^C, \label{EEa.2} \\
    \del{t}\Hct &= -\alphat\Hct^{2} + \frac{1}{n-1}\alphat\et_{A}(\Ut^{A}) + \frac{1}{n-1}\alphat\Ut_{A}(\Ut^{A}-\Ct^A{}_B{}^B) - \frac{1}{n-1}\alphat\Sigmat_{AB}\Sigmat^{AB} + \frac{1}{t}\Hct, \label{EEa.3} \\
    \del{t}\Sigmat_{AB} &= -(n-1)\alphat\Hct\Sigmat_{AB} + \alphat\et_{\langle A}(\Ut_{B\rangle}) + \alphat\Ut_{\langle A}\Ut_{B\rangle} - \alphat\et^C(\Ct_{C\expval{AB}}) - \alphat\et_{\langle A}(\Ct_{B\rangle C}{}^C) \notag \\
        &\quad  - \frac{1}{4}\big( 2\alphat\Ct_{CD\langle A}\Ct_{B\rangle}{}^{CD} + 2\alphat\Ct_{CD\langle A}\Ct_{B\rangle}{}^{DC} - \alphat\Ct_{C\langle A}{}^D\Ct_{B\rangle}{}^C{}_D + \alphat\Ct_C{}^D{}_{\langle A}\Ct^C{}_{B\rangle D} \notag \\
        &\quad - \alphat\delta_{DE}\Ct_{C\langle A}{}^D\Ct^E{}_{B\rangle}{}^C \big) - \alphat\Ut^C\Ct_{C\expval{AB}} + \alphat\Ct_{CD}{}^D\Ct^C{}_{\expval{AB}} - \frac{1}{t}\Sigmat_{AB}, \label{EEa.4} \\
    \del{t}\alphat &= (n-1)\Hct \alphat^2, \label{EEa.5} \\
    \del{t}\Ut_A &= (n-1)\alphat\et_A(\Hct) + (n-2)\alphat\Hct\Ut_A - \alphat\Sigmat_A{}^B\Ut_B. \label{EEa.6}
\end{align}
This is accompanied by the following constraint system:
\begin{align}
    \Aft_{AB}^\Omega &:= 2\et_{[A}(\et_{B]}^{\Omega}) - \Ct_A{}^C{}_B \et_C^\Omega = 0, \label{A-cnstr-1} \\
    \Bft_{AB} &:= 2\et_{[A}(\Ut_{B]}) - \Ct_A{}^C{}_B \Ut_C = 0, \label{B-cnstr-1} \\
    \Cft_{ABC}^D &:= \et_C(\Ct_A{}^D{}_B) + \et_A(\Ct_B{}^D{}_C) + \et_B(\Ct_C{}^D{}_A) \notag \\
        &\quad \, + \Ct_A{}^E{}_B\Ct_C{}^D{}_E + \Ct_B{}^E{}_C\Ct_A{}^D{}_E + \Ct_C{}^E{}_A\Ct_B{}^D{}_E = 0, \label{C-cnstr-1} \\
    \Dft_A &:= \et_{A}(\alphat) - \alphat\Ut_{A} = 0, \label{D-cnstr-1} \\
    \Mft_A &:= \et_{B}(\Sigmat_A{}^B) - (n-2)\et_{A}(\Hct) + \Ct_{ABC}\Sigmat^{BC} - \Ct_{BC}{}^C\Sigmat_A{}^B + \frac{1}{\alphat t}\Ut_A = 0, \label{M-cnstr-1}\\
    \Hft &:= 2\et_A(\Ct^A{}_B{}^B) + (n-1)(n-2)\Hct^{2} - \Sigmat_{AB}\Sigmat^{AB} - \Ct_{AB}{}^B\Ct^A{}_C{}^C \nonumber \\
        &\quad - \frac{1}{4}\Ct_{ABC}(\Ct^{ABC}+\Ct^{BAC}+\Ct^{ACB}) + \frac{2(n-1)}{\alphat t}\Hct = 0. \label{H-cnstr-1}
\end{align}
A key feature of this frame formalism is that the elliptic equation for the lapse vanishes, which is essential for achieving finite propagation speed.

\bigskip

\noindent \underline{\textit{The Fuchsian formulation.}}
The frame formulation of the Einstein-scalar field equations is not well-suited for analyzing its behavior near the singularity at $t=0$. To address this, we transform the system \eqref{EEa.1} to \eqref{EEa.6} to the Fuchsian form defined in \cite[\S 3]{BOOS:2021}. This is achieved by introducing the rescaled variables $W = (e_P^\Sigma, \alpha, C_{PQR}, U_P, \Hc, \Sigma_{PQ})^{\tr}$, defined as
\begin{align*}
    \Hc &= t \alphat\Hct - \frac{r_0}{2(n-1)}, \\
    \Sigma_{AB} &= t \alphat\Sigmat_{AB} - \biggl(\frac{1}{2}r_{AB} - \frac{r_0}{2(n-1)}\delta_{AB}\biggr), \\
    \alpha &= t^{\ep_1}\alphat, \\
    e_A^\Omega &= t^{\ep_2}\alphat \et_A^\Omega, \\
    U_A &= t \alphat\Ut_A, \\
    C_{ABC} &= t \alphat\Ct_{ABC},
\end{align*}
where $\ep_1, \ep_2 \in\Rbb$ are constants that satisfy the following inequalities
\begin{equation*}
    \ep_1+\frac{r_0}{2}>0, \quad 0<\ep_2<1 \quad \text{and} \quad \ep_2+\frac{r_0}{2}-\frac{r_A}{2}>0.
\end{equation*}
An important property of this way of defining $W$ is that $\lim_{t\searrow0}W|_{Kasner}=0$, implying that perturbations of $W$ around zero correspond to perturbations of a background Kasner solution; see Remark~\ref{W-Kasner-limit} for details.

The Fuchsian formulation requires two distinct systems. First, we express \eqref{EEa.1} to \eqref{EEa.6} in terms of the rescaled variables $W$ and put them in the matrix form:
\begin{equation} \label{intro-1}
    \del{t}W = \frac{1}{t^{\ep_2}}e_D^\Lambda E^D \del{\Lambda}W + \frac{1}{t}\Bc\Pbb W + \frac{1}{t}F_l.
\end{equation}
Since $\ep_2<1$, we choose an arbitrary positive number $\nu$ such that
\begin{equation*}
    \ep_2 + \nu < 1.
\end{equation*}
The lower order variables $W_\bc$ are defined by
\begin{equation*}
    W_\bc := t^{|\bc|\nu}\partial^{\bc}W = t^{|\bc|\nu}(\partial^{\bc}e_P^\Sigma, \partial^{\bc}\alpha, \partial^{\bc}C_{PQR}, \partial^{\bc}U_P, \partial^{\bc}\Hc, \partial^{\bc}\Sigma_{PQ})^{\tr}, \quad |\bc|<k,
\end{equation*}
where $\bc=(\bc_1,...,\bc_{n-1})$ is a multiindex\footnote{See Section~\ref{Index} for the multiindex definition.} and $k\in\Nbb$ is a positive integer to be determined. Spatially differentiating \eqref{intro-1} yields the evolution equation for the lower order terms:
\begin{equation} \label{intro-2}
    \del{t}W_\bc = \frac{1}{t}(|\bc|\nu\id+\Bc\Pbb)W_\bc + \frac{1}{t^{\ep_2+\nu}}\sum_{\substack{\bc'\leq\bc\\|\ac|=1}}\Pbb W_{\bc-\bc'}*W_{\bc'+\ac} + \frac{1}{t}\sum_{\bc'\leq\bc}\Pbb W_{\bc-\bc'}*(\Pbb W_{\bc'}+\Pbb^{\perp}W_{\bc'}),  \quad |\bc|<k.
\end{equation}
To close the system, we also need an equation for the highest order term, $W_{\bc}$ with $|\bc|=k$. Since adding multiples of constraints to the evolution system preserves physical equivalence, we use this freedom to modify \eqref{intro-1} into what follows
\begin{equation} \label{intro-3}
    \del{t}W + \frac{1}{t^{\ep_2}}e_D^\Lambda A^D \del{\Lambda}W = \frac{1}{t}\Ac W + \frac{1}{t}F_h.
\end{equation}
Introducing a change of variables and an appropriate symmetrizer, equation \eqref{intro-3} can be cast into symmetric hyperbolic form:
\begin{equation} \label{intro-4}
    B^0\del{t}\Wt_\bc + \frac{1}{t^{\ep_2}}e_D^\Lambda B^D \del{\Lambda}\Wt_\bc = \frac{1}{t}\Bsc\Wt_\bc + \frac{1}{t^{\ep_2+\nu}}\sum_{\substack{\bc'\leq\bc\\|\ac|=1}}\Pbb W_{\bc-\bc'}*W_{\bc'+\ac} + \frac{1}{t}\sum_{\bc'\leq\bc}\Pbb W_{\bc-\bc'}*(\Pbb W_{\bc'}+\Pbb^{\perp}W_{\bc'}),
\end{equation}
for $|\bc|=k$; see Section~\ref{low-order} and ~\ref{high-order} for details. Note that the lower order terms satisfy ODEs, while the highest order terms satisfy PDEs. Combining \eqref{intro-2} and \eqref{intro-4} yields the Fuchsian system:
\begin{equation} \label{intro-5}
    \Bv^0\del{t}\Wv + \frac{1}{t^{\ep_2}}\Bv^\Lambda(\Wv) \del{\Lambda}\Wv = \frac{1}{t}\tilde{\Bcv}\Pv\Wv + \frac{1}{t^{\ep_2+\nu}}\Hv(\Wv) + \frac{1}{t}\Fv(\Wv).
\end{equation}

The introduction of additional spatial derivative terms in the symmetrization process ensures a symmetric hyperbolic structure while maintaining the positive definiteness of $\tilde{\Bcv}$. Notice that in \cite{BeyerOliynyk:2024b}, where the authors established the stability around the FLRW solution (which is homogeneous and isotropic), controlling higher order terms is unnecessary, unlike in \cite{BOZ:2025} and the present work. This can be intuitively explained by the fact that greater anisotropy in the spacetime requires controlling higher-order derivatives. This intuition finds further support in \cite{Li:2024}, where the author studied the linearized Einstein-scalar field equations. The optimal scattering results obtained there indicate that it is unavoidable to control higher order spatial derivatives when approaching the boundary of the sub-critical regime.

\bigskip

\noindent \underline{\textit{Global-in-time existence of Fuchsian system.}}
We establish twodistinct stability results: a global-in-space version on $M_{0,t_0}$ and a local-in-space version on $\Omega_{\Icv}$. The global result follows directly from the work in \cite{BOOS:2021}, where the authors derived the global existence theory for Fuchsian systems along with corresponding energy and decay estimates. Specifically, we verify that the coefficients in \eqref{intro-5} satisfy all the assumptions detailed in \cite[\S 3.4]{BOOS:2021}; the complete argument is presented in the proof of Proposition~\ref{Fuch-global}.

Considering the localized Fuchsian stability on $\Omega_{\Icv}$, we proceed by first using the extension operator from Section~\ref{extension} to extend the initial data to the full hypersurface $\{t_0\}\times\Tbb^{n-1}$. An application of Proposition~\ref{Fuch-global} then yields a solution on the global domain $M_{0,t_0}$. The uniqueness of the solution within the truncated cone domain $\Omega_{\Icv}$ is subsequently obtained by demonstrating that its lateral boundary $\Gamma_{\Icv}$ is weakly spacelike with respect to the Fuchsian system \eqref{intro-5}. The detailed justification for this step is provided in the proof of Proposition~\ref{Fuch-local}.

\bigskip

\noindent \underline{\textit{Local-in-time existence via Lagrangian coordinates.}}
In the frame formalism of the Einstein-scalar field equations, we employ the zero-shift gauge and the set $\tau=t$. These gauge choices are identical to those used in the formulation relative to Lagrangian coordinates in \cite[\S 5.3, \S 5.4]{BeyerOliynyk:2024b}; further details are provided in Section~\ref{gauges-veri}. Moreover, the initial data of the Lagrangian system can be used to construct the frame variables $\Wt$ by following the same procedures outlined in Section~\ref{Tetrad-formalism}. Consequently, the local-in-time existence result for the Einstein-scalar field equations in Lagrangian coordinates guarantees the existence of a solution to the system \eqref{EEa.1} to \eqref{EEa.6} and ensures the propagation of the constraints \eqref{A-cnstr-1} to \eqref{H-cnstr-1}.

A key feature of the Lagrangian system is that both the evolution system\footnote{See Section~\ref{Lag-IVP} for definitions of coefficients.}
\begin{equation} \label{intro-6}
    P^0 \del{t}Z + P^\Gamma \del{\Gamma}Z = Y
\end{equation}
and the wave gauge propagation equation
\begin{align}
    \nabla_{\mu}\nabla^{\mu}X^{\nu} + 2\nabla_\mu \ln(\tau)\nabla^{[\mu}X^{\nu]} + \nabla_\mu \ln(\tau)\nabla^{\nu}X^\mu + \bigl(R^{\nu}{}_{\mu} - \nabla^{\nu}\nabla_\mu \ln(\tau) - 2\nabla^\nu\ln(\tau)\nabla_\mu\ln(\tau)\bigr) X^\mu = 0 \label{intro-7}
\end{align}
are hyperbolic and share the same principal parts. Thus, it is not difficult to derive the local-in-time existence, uniqueness and continuation principle for both \eqref{intro-6} and \eqref{intro-7} on the domains $M_{t_1,t_0}$ and $\Omega_{\Icv}$ with $\Icv=(t_0,t_1,\rho_0,\rho_1,\ep_2)$. We also verify that the lateral boundary of $\Omega_{\Icv}$ is weakly spacelike with respect to the Lagrangian system; see the argument preceding Proposition~\ref{local-local}.

Establishing this local-in-time existence is crucial for ensuring that the global-in-time solution obtained from the Fuchsian system indeed corresponds to a solution of the Einstein-scalar field equations. This is necessary because Propositions~\ref{Fuch-global} and~\ref{Fuch-local} do not require the constraints to vanish as the system evolves. In \cite{BOZ:2025}, this issue is addressed by proving that the constraint propagation system is strongly hyperbolic, ensuring that constraints vanishing initially persist. While in the present article, we instead establish this result by exploiting the physical equivalence between the frame system and the Lagrangian system.

\bigskip

\noindent \underline{\textit{Localized big bang stability.}}
Our analysis aims to characterize the physical metric. In Section~\ref{physical-quant}, we derive expressions for key physical quantities, e.g. scalar curvature $\Rb$, Ricci scalar invariant $\Rb_{ab}\Rb^{ab}$, the second fundamental form $\Kttb_{AB}$ and the invariant $\Cb_{A0B0}\Cb^{A0B0}$, originating from the Weyl tensor, relative to the conformal orthonormal frame $\{\et_a\}$.

The localized past nonlinear stability of the entire family of sub-critical Kasner solutions in spacetime dimensions $n\geq4$ is established in Theorem~\ref{main-thm} by synthesizing the results of Proposition~\ref{Fuch-local} and ~\ref{local-local}. The proof proceeds as follows: Proposition~\ref{Fuch-local} guarantees the existence of a unique solution to the Fuchsian system \eqref{intro-5} for initial data that are sufficiently close to the background Kasner metric. This same initial data also generate a local-in-time solution to the frame system \eqref{EEa.1} to \eqref{EEa.6}, which, upon spatial differentiation, yields a solution to the Fuchsian system. A uniqueness argument then shows these two solutions are identical. Furthermore, we establish that the perturbed solutions exhibit the following characteristics near the singularity at $t=0$: they have specific asymptotic behavior, past timelike geodesic incompleteness, terminate at a crushing singularity, possess AVTD property and are asymptotically pointwise Kasner.

\subsection{Acknowledgments}
The author would like to thank Florian Beyer for helpful discussions regarding the selection of appropriate curvature invariants originating from the Weyl tensor to characterize the singularity formation. The author is also grateful to Todd Oliynyk for his valuable comments and suggestions during the revision process, which greatly improved the clarity and rigor of this work.

\section{Preliminaries}

\subsection{Data availability statement}
This article has no associated data.

\subsection{Coordinates and frames conventions} \label{coor-frame}
In this article, we will consider $n$-dimensional spacetime manifolds of the form
\begin{equation*} \label{manifold-def}
    M_{t_1, t_0} = (t_1, t_0] \times \Tbb^{n-1},
\end{equation*}
where $0\leq t_1<t_0$ and $\Tbb^{n-1}$ is $(n-1)$-torus defined as
\begin{equation*} \label{Tbb-def}
    \Tbb^{n-1} = [-L, L]^{n-1}/\sim,
\end{equation*}
where $\sim$ is the equivalence relation obtained from identifying the sides of the box $[-L, L]^{n-1}\in\Rbb^{n-1}$.

On $M_{t_1, t_0}$, we employ the coordinates $(x^\mu)=(x^0, x^\Omega)$ where $(x^\Omega)$ are the periodic spatial coordinates on $\Tbb^{n-1}$ and $(x^0)$ is the time coordinate on the interval $(t_1, t_0]$. Lower case Greek letters, e.g. $\mu,\nu$ will run from $0$ to $n-1$ labelling the spacetime coordinate indices, while upper case Greek letter, e.g. $\Sigma,\Omega$ will run from $1$ to $n-1$ labelling the spatial coordinate indices. Partial derivatives with respect to the coordinates $(x^\mu)$ are denoted by $\del{\mu}=\frac{\partial}{\partial x^\mu}$. We will usually use $t$ to denote the time coordinate $x^0$, i.e. $t=x^0$, and use the $\del{t}$ to denote $\del{0}$, which is the partial derivative with respect to $x^0$.

We will use frames $e_a=e_a^\mu \del{\mu}$ constructed from the coordinate frames in this article. Lower case Lattin letters, e.g. $a,b,c$ will run from $0$ to $n-1$ labelling the spacetime frame indices, while upper case Lattin letters, e.g. $A,B,C$ will run from $1$ to $n-1$ labelling the spatial frame indices.

\subsection{Index operators and multi-indices} \label{Index}
For a spatial tensor on the manifold $M_{t_1, t_0}$, the symmetrization, anti-symmetrization and symmetric trace free operators on any two of its spatial indices are defined respectively by
\begin{equation} \label{index-op}
    T_{(AB)} = \frac{1}{2}(T_{AB}+T_{BA}), \quad T_{[AB]} = \frac{1}{2}(T_{AB}-T_{BA}) \quad \text{and} \quad T_{\langle AB\rangle} = T_{(AB)} - \frac{1}{n-1}\delta_{AB}\delta^{CD}T_{CD}.
\end{equation}
In this article, we will use the $*$-notation to denote the multilinear maps of spatial tensor fields for which it is not necessary to know that exact coefficients. To be more precise, the term $[C*U]$ will stand for the tensor fields of the form
\begin{equation*}
    \delta^{DF}C_{ABC}U_D \quad \text{or} \quad \ell_{AB}^{EFG} C^B{}_{C\langle D}U_{F\rangle},
\end{equation*}
where $\ell_{AB}^{EFG}$ are some constants. The $*$-notation will also obey the distribution law, i.e. we have
\begin{equation*}
    [(C+U)*(\Hc+\Sigma)] = C*\Hc + C*\Sigma + U*\Hc + U*\Sigma.
\end{equation*}

We use the standard multi-index notation. If $\bc=(\bc_1, \bc_2,..., \bc_{n-1})$ is an $(n-1)$-tuple of nonnegative integers $\bc_i$, we call $\bc$ a multi-index. Its order is defined as $|\bc|=\bc_1+\bc_2+\cdot\cdot\cdot+\bc_{n-1}$. We also denote $\bc!=\bc_1!\bc_2!\cdot\cdot\cdot\bc_{n-1}!$. We use $\partial^\bc$ to denote the spatial partial derivative of order $|\bc|$, which is given by
\begin{equation*}
    \partial^\bc = \del{1}^{\bc_1} \del{2}^{\bc_2}\cdot\cdot\cdot \del{n-1}^{\bc_{n-1}}.
\end{equation*}
If $\ac$ and $\bc$ are two multi-indices, we say that $\ac\leq\bc$ provided that $\ac_i\leq\bc_i$ for $1\leq i\leq n-1$. In this case $\bc-\ac$ is also a multi-index and we have that
\begin{equation*}
    \binom{\bc}{\ac} = \frac{\bc!}{\ac!(\bc-\ac)!}.
\end{equation*}
Also, for any $k\in\Nbb$ and a function $W$,
\begin{equation*}
    \partial^k W = \bigl\{\partial^\bc W \, \big| \, |\bc|=k\bigr\}
\end{equation*}
denotes the collection of all the spatial derivatives of $W$ of order $k$.

\subsection{Vectors and matrices}
In this article, the Euclidean norm of a vector $v=(v_1,v_2,...,v_N)^{\tr}$ in $\Rbb^N$ is defined by $|v|=(v^{\tr}v)^{\frac{1}{2}}$. For example, using the Einstein summation convention, we have $|\Sigma|=(\Sigma_{AB}\Sigma^{AB})^{\frac{1}{2}}$.

Let $\mathbb{M}_N$ denote the set of all real $N\times N$ matrices. For $A,B\in\mathbb{M}_N$, we say that $A\leq B$ if
\begin{equation*}
    v^{\tr}Av \leq v^{\tr}Bv,
\end{equation*}
for all $v\in\Rbb^N$.

\subsection{Domains in $\Tbb^{n-1}$} \label{Domains}
A centered ball of radius $\rho$ in $\Tbb^{n-1}$, with $0<\rho<L$, is defined as
\begin{equation*}
    \mathbb{B}_\rho := \bigl\{ x\in[-L, L]^{n-1} \big| \, |x|<\rho \bigr\} \, (\subset\Tbb^{n-1}),
\end{equation*}
where, as the above choice of notation, $|x|=\sqrt{\delta_{\Lambda\Omega}x^\Lambda x^\Omega}$ is the Euclidean norm of $x$ on $\Rbb^{n-1}$. Given constants $t_0>0$, $0\leq t_1<t_0$, $0<\rho_0<L$, $\rho_1>0$ and $0\leq\ep<1$ that satisfy
\begin{equation} \label{rho1-range}
    \rho_0 - \frac{\rho_1 t_0^{1-\ep}}{1-\ep} > 0,
\end{equation}
we label them as a single index $\Icv=(t_0,t_1,\rho_0,\rho_1,\ep)$ and define the truncated cone domain $\Omega_{\Icv}$ as a subset of $M_{0,t_0}$ to be
\begin{equation} \label{OmegaIcv}
    \Omega_{\Icv} = \biggl\{ (t,x)\in(t_1,t_0)\times(-L,L)^{n-1} \, \bigg| \, |x|<\frac{\rho_1(t^{1-\ep}-t_0^{1-\ep})}{1-\ep}+\rho_0 \biggr\}.
\end{equation}
We use $\Gamma_{\Icv}$ to denote the "side" piece of the boundary of $\Omega_{\Icv}$ and we see from \eqref{OmegaIcv} that it is determined by the vanishing of the function
\begin{equation*}
    \omega = |x| - \frac{\rho_1(t^{1-\ep}-t_0^{1-\ep})}{1-\ep} - \rho_0.
\end{equation*}
Differentiating $\omega$ gives us a differential 1-form
\begin{equation} \label{normal}
    n_\mu = -\frac{\rho_1}{t^\ep}\delta_\mu^0 + \frac{1}{|x|}x^\Lambda \delta_{\mu\Lambda},
\end{equation}
which is the outward-pointing normal to $\Gamma_{\Icv}$. The boundary of $\Omega_{\Icv}$ can be decomposed as the following disjoint union
\begin{equation*}
    \partial\Omega_{\Icv} = \bigl(\{t_0\}\times\mathbb{B}_{\rho_0}\bigr) \cup \Gamma_{\Icv} \cup \bigl(\{t_1\}\times\mathbb{B}_{\tilde{\rho}_0}\bigr),
\end{equation*}
where $\tilde{\rho}_0=\frac{\rho_1(t_1^{1-\ep}-t_0^{1-\ep})}{1-\ep}+\rho_0$ and the balls $\{t_0\}\times\mathbb{B}_{\rho_0}$ and $\{t_1\}\times\mathbb{B}_{\tilde{\rho}_0}$ bound the truncated cone domain $\Omega_{\Icv}$ from the top and the bottom respectively.

\subsection{Sobolev spaces and extension operators} \label{extension}
For $k\in\Zbb_{\geq0}$ and $p\geq1$, the $W^{k,p}$ norm of a map $u\in C^{\infty}(U,\Rbb^N)$ with $U\subset\Tbb^{n-1}$ open is defined by
\begin{equation*}
    \norm{u}_{W^{k,p}(U)} = \begin{cases}
        \begin{displaystyle}
            \biggl( \sum_{0\leq|\bc|\leq k}\int_{U}|\partial^{\bc}u|^p \, d^{n-1}x \biggr)^{\frac{1}{p}}
        \end{displaystyle}
        & \text{if $1\leq p<\infty$} \\
        \begin{displaystyle}
            \max_{0\leq|\bc|\leq k} \, \sup_{x\in U} |\partial^{\bc}u(x)|
        \end{displaystyle} 
        & \text{if $p=\infty$}
    \end{cases}
\end{equation*}
whenever the right hand side makes sense and $\bc$ is a multi-index. The Sobolev space $W^{k,p}(U,\Rbb^N)$ is then defined to be the completion of $C^\infty(U,\Rbb^N)$ with respect to the norm $\norm{\cdot}_{W^{k,p}(U)}$. We will also write $W^{k,p}(U)$ instead of $W^{k,p}(U,\Rbb^N)$ if the dimension $N$ is clear from the context. We will employ the standard notation $H^k(U,\Rbb^N)=W^{k,2}(U,\Rbb^N)$ in this article.

For each centered ball $\mathbb{B}_\rho\subset\Tbb^{n-1}$ such that $\rho<L$, there exists a total extension operator defined as
\begin{equation*}
    E_\rho : H^k(\mathbb{B}_\rho, \Rbb^N) \longrightarrow H^k(\Tbb^{n-1}, \Rbb^N), \quad k\in\Zbb_{\geq0}
\end{equation*}
and satisfying
\begin{equation*}
    E_\rho(u) = u \quad \text{a.e. in $\mathbb{B}_\rho$} \quad \text{and} \quad \norm{E_\rho(u)}_{H^k(\Tbb^{n-1})} \leq C\norm{u}_{H^k(\mathbb{B}_\rho)}
\end{equation*}
for some constant $C=C(k,n,\rho)>0$ independent of $u\in H^k(\mathbb{B}_\rho)$. For the existence of such an extension operator, we refer to \cite{AdamsFournier:2003}. Details can be found in Theorem 5.22 and Remark 5.23.

\subsection{Constants and inequalities}
In this article, we will use the notation $f\lesssim g$ to represent the inequalities $f\leq Cg$ when we do not need to know the exact value and the dependence on other quantities of the constant $C$. On the other hand, when the dependence on other quantities is of importance, for example if the constant $C$ depends on the value of $\norm{u}_{H^k(U)}$, we will write $f\leq C(\norm{u}_{H^k(U)})g$. Constants of this type will always be nonnegative, non-decreasing and continuous functions of their dependence.

\subsection{Order notation} \label{order-nota}
For $t_0>0$, $U$ an open set in  $\Tbb^{n-1}$, and $g\in C^0\bigl((0,t_0],\Rbb_{>0}\bigr)$, we say that a time dependent function $f\in C^0\bigl((0,t_0],H^k(U)\bigr)$ is \textit{order $g(t)$} and write $f=\Ord_{H^k(U)}(g(t))$ if
\begin{equation*}
    \norm{f(t)}_{H^k(U)} \lesssim g(t), \quad \forall \, t\in (0,t_0].
\end{equation*}

\begin{lem} \label{lem:asymptotic}
Suppose $t_0>0$, $U$ is an open set in $\Tbb^{n-1}$, $g\in C^0\bigl((0,t_0],\Rbb_{>0}\bigr)$, and
$f\in C^1\bigl((0,t_0],H^k(U)\bigr)$ for $k\in\Zbb_{>\frac{n-1}{2}+1}$, then if $\del{t}f=\Ord_{H^k(U)}(g(t))$ and $\int_0^{t_0} g(s)\,ds<\infty$, there exists a unique $F\in C^0\bigl([0,t_0],H^k(U)\bigr)$ such that $F(t)=f(t)$ for all $t\in (0,t_0]$, and 
$F(t)=F(0)+\Ord_{H^k(U)}(h(t))$ where $h(t)=\int_0^t g(s)\,ds$.
\end{lem}
The proof is identical to that of Lemma 2.1 in \cite{BOZ:2025}.

\section{Frame formalism of the conformal Einstein-scalar field equations} \label{Tetrad-formalism}

The tetrad formulation of Einstein's field equations on a 4-dimensional spacetime was established in \cite{GarfinkleGundlach:2005,Uggla_et_al:2003,vanElstUggla:1997}. In this section, we generalize this formalism to the Einstein-scalar field equations on an $n$-dimensional spacetime for $n\geq4$. Our notation largely follows that of \cite{BOZ:2025}.

\subsection{An orthonormal frame}
We begin by replacing the physical spacetime metric $\gb_{\mu\nu}$ with a conformal metric $\gt_{\mu\nu}$ defined by
\begin{equation} \label{conf-metricA}
    \gb_{\mu\nu} = e^{2\Phi}\gt_{\mu\nu},
\end{equation}
where the conformal scalar $\Phi$ is to be determined later. Then we introduce an orthonormal frame $\et_a = \et_a^\mu \del{\mu}$ for the conformal metric so that relative to the frame, $\gt_{ab}$ is given by
\begin{equation*} 
    \eta_{ab} = \gt(\et_a,\et_b) = \gt_{\mu\nu}\et_a^\mu \et_j^\nu
\end{equation*}
where
\begin{equation} \label{f-metric}
    \eta_{ab} = -\delta_a^0\delta_b^0 + \delta_{AB} \delta_a^A \delta_b^B.
\end{equation}

We employ a zero-shift gauge in the following formulation; that is, relative to the coordinate frame $\del{\mu}$, the conformal metric is given by
\begin{equation*}
    \gt= - \alphat^2 dt\otimes dt + \gt_{\Sigma \Omega}dx^\Sigma \otimes dx^\Omega,
\end{equation*}
where $\alphat$ denotes the lapse function. To fix the frame, we first set
\begin{equation} \label{et0-def}
    \et_0 = \frac{1}{\alphat}\del{t},
\end{equation}
so that clearly
\begin{equation*}
    \gt_{00} = \gt(\et_0, \et_0) = -1.
\end{equation*}
The remaining frame components $\et_A$ are then determined by propagation via Fermi–Walker transport, defined by
\begin{equation} \label{Fermi-tran}
    \nablat_{\et_0}\et_A = -\frac{\gt(\nablat_{\et_0}\et_0, \et_A)}{\gt(\et_0, \et_0)}\et_0,
\end{equation}
where $\nablat$ is the Levi-Civita connection of $\gt$. A brief calculation shows that if initially
\begin{equation*}
    \gt_{0A}=\gt(\et_0, \et_A)=0 \AND \gt_{AB}=\gt(\et_A, \et_B)=\delta_{AB}
\end{equation*}
then, by virtue of \eqref{Fermi-tran},
\begin{equation*}
    \nablat_{\et_0}\gt(\et_0,\et_A)=\nablat_{\et_0}\gt(\et_A,\et_B)=0.
\end{equation*}
Hence the orthonormality of the frame $\et_i$ is preserved. Consequently, the spatial frame can be expressed as
\begin{equation} \label{etA-def}
    \et_A = \et^\Omega_A \del{\Omega}.
\end{equation}

\subsection{Connection and commutator coefficients} \label{conn-comm-coef}
From the above discussion, $\nablat$ denotes the Levi-Civita connection associated with the metric $\gt$. The connection coefficients $\omega_a{}^c{}_b$ are defined by the relation
\begin{equation*}
    \nablat_{\et_a}\et_b = \omega_a{}^c{}_b \et_c.
\end{equation*}
Similarly, the commutator coefficients $c_i{}^k{}_j$ of the orthonormal frame $\et_a$ are introduced via
\begin{equation*}
    [\et_a, \et_b] = c_a{}^c{}_b \et_c.
\end{equation*}
Since $\nablat$ is torsion-free and $\et_i$ is an orthonormal frame, the connection and commutator coefficients have the following relationship
\begin{equation} \label{coeff-relations}
    c_i{}^k{}_j = \omega_i{}^k{}_j - \omega_j{}^k{}_i \AND \omega_{ijk} = \frac{1}{2}(c_{jik}-c_{kji}-c_{ikj}),
\end{equation}
where we use the metric $\eta_{ab}$ to raise and lower indices. Now we express $[\et_0, \et_A]$ and $[\et_A, \et_B]$ as
\begin{align}
    [\et_0,\et_A] &= \Ut_A \et_0 + \St_A{}^B\et_B, \label{comm-decomp-0} \\
    [\et_A,\et_B] &= \Ct_A{}^C{}_B \et_C. \label{comm-decomp-A}
\end{align}
Then decompose $\St_A{}^B$ into its trace and trace-free parts
\begin{equation} \label{St-decomp}
    \St_{AB} = -\Hct\delta_{AB} - \Sigmat_{AB},
\end{equation}
where
\begin{equation*} \label{Hct-Sigmat-def}
    \Hct = -\frac{1}{n-1}\St_A{}^A \AND \Sigmat_{AB} = -\St_{AB} - \Hct\delta_{AB}.
\end{equation*}
Using the index conventions in \eqref{index-op}, we see that
\begin{equation*} \label{index-sym}
    \Ct_{[A}{}^C{}_{B]} = 0 \AND \Sigmat_{\expval{AB}} = \Sigmat_{AB}.
\end{equation*}

\begin{rem}
    If we are on a 4-dimensional spacetime, then according to \cite[Eq.~(7.2.32)]{Wald:1994}, we can further decompose $\Ct_A{}^C{}_B$ into
    \begin{equation*}
        \Ct_A{}^C{}_B = 2\At_{[A}\delta_{B]}^C + \ep_{ABD}\Nt^{DC},
    \end{equation*}
    where
    \begin{equation*}
        \At_A = -\frac{1}{2}\Ct^B{}_{BA}  \AND  \Nt^{AB} = \frac{1}{2}\ep^{ACD}(\Ct^B{}_{CD}-\delta_C^B\Ct^E{}_{ED}),
    \end{equation*}
    which is the same as that in \cite[Eq.~(3.9)]{BOZ:2025}. Since this decomposition can only work on tensors on 3-dimensional vector spaces, we need to use $\Ct_A{}^C{}_B$ to be our variable in the present article. In the following analysis, we will see that using $\Ct_A{}^C{}_B$ generates equations with different structures that require a more complicated treatment.
\end{rem}

From \eqref{coeff-relations} to \eqref{comm-decomp-A}, a direct computation yields the following expressions for $\omega_a{}^c{}_b$ in terms of $\Ut_A$, $\Ct_{ABC}$, $\Hct$ and $\Sigmat_{AB}$
\begin{gather} \label{connect-form}
    \omega_{0A0} = -\omega_{00A} = \Ut_A, \qquad \omega_{A0B} = -\omega_{AB0} = -(\Hct\delta_{AB}+\Sigmat_{AB}), \notag \\
    \omega_{ABC} = \frac{1}{2}(\Ct_{BAC}-\Ct_{CBA}-\Ct_{ACB}), \qquad \omega_{000} = \omega_{A00} = \omega_{0BC}=0. \label{connect-form}
\end{gather}

\subsection{Conformal Einstein-scalar field equations and constraints} \label{conf-ESF}
In this section, we derive the equations of the frame formalism. Some of them come from the geometric setup and the others come from the Einstein equations.

We begin by applying both sides of \eqref{comm-decomp-0} to the coordinate $t$. Using definitions \eqref{et0-def} and \eqref{etA-def}, we obtain
\begin{align*}
    [\et_0, \et_A]t = \et_0\bigl(\et_A(t)\bigr) - \et_A\bigl(\et_0(t)\bigr) &= \frac{1}{\alphat^2}\et_A(\alphat), \\
    (\Ut_A \et_0 + \St_A{}^B\et_B)t &= \frac{1}{\alphat}\Ut_A.
\end{align*}
This leads to the constraint equation
\begin{equation} \label{Dft-cnstr}
    \et_A(\alphat) - \alphat\Ut_A = 0.
\end{equation}
Applying both sides of \eqref{comm-decomp-0} to the coordinate $x^\Omega$ and using \eqref{et0-def}, \eqref{etA-def} and \eqref{St-decomp}, we obtain the evolution equation for $\et_A^\Omega$:
\begin{equation*}
    \del{t}\et_{A}^{\Omega} = -\alphat(\Hct\delta_{A}^{B} + \Sigmat_{A}{}^{B})\et_{B}^{\Omega}.
\end{equation*}
Next, an application of \eqref{comm-decomp-A} to $x^\Omega$ yields an additional constraint:
\begin{equation*}
    2\et_{[A}(\et_{B]}^{\Omega}) - \Ct_A{}^C{}_B \et_C^\Omega = 0.
\end{equation*}
To derive an evolution equation for $\Ct_A{}^C{}_B$, we utilize the Jacobi identity. Consider
\begin{equation} \label{Jacobi-id1}
    \bigl[\et_0, [\et_A,\et_B]\bigr] + \bigl[\et_A, [\et_B,\et_0]\bigr] + \bigl[\et_B, [\et_0,\et_A]\bigr] = 0,
\end{equation}
and apply \eqref{comm-decomp-0} and \eqref{comm-decomp-A}, together with the identity for the Lie brackets,
\begin{equation} \label{Lie-brack}
    [fX, gY] = fg[X, Y] + fX(g)Y - gY(f)X,
\end{equation}
to rewrite \eqref{Jacobi-id1} as
\begin{align*}
    \et_0(\Ct_A{}^C{}_B)\et_C &= - \Ct_A{}^C{}_B(\Ut_C\et_0 + \St_C{}^D\et_D) + 2\Ut_{[A}\St_{B]}{}^C\et_C - 2\St_{[A}{}^C\Ct_{B]}{}^D{}_C\et_D \notag \\
        &\quad + 2\et_{[A}(\Ut_{B]})\et_0 + 2\et_{[A}(\St_{B]}{}^C)\et_C.
\end{align*}
Applying both sides of this identity to the coordinate $x^\Omega$ and using the decomposition for $\St_A{}^B$ given in \eqref{St-decomp}, we obtain the evolution equation for $\Ct_A{}^C{}_B$:
\begin{align*}
    \del{t}\Ct_A{}^C{}_B &= -2\alphat\et_{[A}(\Hct)\delta_{B]}^C - 2\alphat\et_{[A}(\Sigmat_{B]}{}^C) - 2\alphat\Hct \Ut_{[A}\delta_{B]}^C - 2\alphat\Ut_{[A}\Sigmat_{B]}{}^C \notag \\
        &\quad -\alphat\Hct\Ct_A{}^C{}_B + 2\alphat\Sigmat_{[A}{}^D \Ct_{B]}{}^C{}_D + \alphat\Ct_A{}^D{}_B\Sigmat_D{}^C.
\end{align*}
Applying the same procedure to the coordinate $t$ yields an additional constraint:
\begin{equation*}
    2\et_{[A}(\Ut_{B]}) - \Ct_A{}^C{}_B \Ut_C = 0.
\end{equation*}
We now apply the Jacobi identity to three distinct spatial frames, namely
\begin{equation*}
    \bigl[\et_C, [\et_A,\et_B]\bigr] + \bigl[\et_A, [\et_B,\et_C]\bigr] + \bigl[\et_B, [\et_C,\et_A]\bigr] = 0.
\end{equation*}
Using \eqref{comm-decomp-A} and \eqref{Lie-brack} again, we derive the constraint:
\begin{align*}
    \et_C(\Ct_A{}^D{}_B) + \et_A(\Ct_B{}^D{}_C) + \et_B(\Ct_C{}^D{}_A) + \Ct_A{}^E{}_B\Ct_C{}^D{}_E + \Ct_B{}^E{}_C\Ct_A{}^D{}_E + \Ct_C{}^E{}_A\Ct_B{}^D{}_E = 0.
\end{align*}

So far we have not used the Einstein-scalar field equations yet. It is well known that under the conformal transformation $\gb_{ab}=e^{2\Phi}\gt_{ab}$, the Ricci tensor transforms via
\begin{equation*}
    \Rb_{ab} = \Rt_{ab} - \Upsilon_{ab},
\end{equation*}
where $\Rt_{ab}$ is the Ricci curvature tensor of the conformal metric $\gt$ and
\begin{equation} \label{Upsilon-def}
    \Upsilon_{ab} := (n-2)\nablat_a\psi_b - (n-2)\psi_a\psi_b + \eta_{ab}(\nablat_c \psi^c + (n-2)\psi_c \psi^c)
\end{equation}
with
\begin{equation*} \label{psia-def}
    \psi_a = \nablat_a\Phi = \et_a(\Phi).
\end{equation*}
Also the Einstein-scalar field equations \eqref{ESF.3} becomes
\begin{equation} \label{cEEa}
    \Rt_{ab} - \Upsilon_{ab} = \Tb_{ab} - \frac{1}{n-2}\Tb_{c}{}^{c}\eta_{ab}
\end{equation}
where $\Tb_c{}^c=\eta^{bc}\Tb_{bc}$ and
\begin{align}
    \Tb_{ab} &= 2\et_a(\varphi)\et_b(\varphi) - \eta^{cd}\et_c(\varphi)\et_d(\varphi)\eta_{ab}. \label{Tbab-def}
\end{align}
For use below, we note that using \eqref{connect-form}, the components of $\Upsilon_{ab}$ can be expressed as:
\begin{align*}
    \Upsilon_{00} &= (n-1)\et_0(\psi_0) + (n-1)\Hct\psi_0 - \bigl(\et_A(\psi^A) + (n-2)\psi_A\psi^A + (n-1)\Ut^A \psi_A + \Ct^A{}_{AB}\psi^B\bigr), \\ 
    \Upsilon_{0A} &= (n-2)\et_0(\psi_A) - (n-2)\Ut_A\psi_0 - (n-2)\psi_0\psi_A, \\ 
    \Upsilon_{\expval{AB}} &= -(n-2)\Sigmat_{AB}\psi_0 + (n-2)\et_{\langle A}\psi_{B\rangle} - (n-2)\psi_{\langle A}\psi_{B\rangle} + \frac{1}{2}(\Ct_{BAC}-\Ct_{CBA}-\Ct_{ACB})\psi^C \notag \\
        &\quad - \delta_{AB}\Ct^C{}_{CD}\psi^D, \\ 
    \Upsilon_{00} + \Upsilon_A{}^A &= -(n-1)(n-2)(2\Hct+\psi_0)\psi_0 + (2n-4)\et_A(\psi^A) + (n-2)(n-3)\psi_A\psi^A \notag \\
        &\quad + (2n-4)\Ct^A{}_{AB}\psi^B. 
\end{align*}
Equation \eqref{cEEa} indicates that the components of $\Rt_{ab}$ can be expressed as
\begin{equation} \label{Ricci}
    \begin{gathered}
        \Rt_{00} = \frac{n-3}{n-2}\Tb_{00} + \frac{1}{n-2}\Tb_A{}^A + \Upsilon_{00}, \quad \Rt_{0A} = \Tb_{0A} + \Upsilon_{0A}, \\
        \Rt_A{}^A = \frac{n-1}{n-2}\Tb_{00} - \frac{1}{n-2}\Tb_A{}^A + \Upsilon_A{}^A, \quad \Rt_{\langle AB \rangle} = \Tb_{\langle AB \rangle} + \Upsilon_{\langle AB \rangle}.
    \end{gathered}
\end{equation}
Recall that the Riemann curvature tensor $\Rt_{abcd}$ can be calculated using the connection coefficients $\omega_{abc}$ via
\begin{equation} \label{Riemann}
    \Rt_{abcd} = \et_a(\omega_{bcd}) - \et_b(\omega_{acd}) + \eta^{ef}(\omega_{bfc}\omega_{aed} + \omega_{bfa}\omega_{ecd} - \omega_{afc}\omega_{bed} - \omega_{afb}\omega_{ecd}).
\end{equation}
Hence, the Ricci tensor may alternatively be computed via $\Rt_{ac}=\eta^{bd}\Rt_{abcd}$ where $\Rt_{abcd}$ is provided in \eqref{Riemann}. Equating the resulting expressions with those in \eqref{Ricci}, we are able to derive the evolution equations for $\Hct$ and $\Sigmat_{AB}$:
\begin{align}
    \del{t}\Hct &= -\alphat\Hct^{2} + \frac{1}{n-1}\alphat\et_{A}(\Ut^{A}) + \frac{1}{n-1}\alphat\Ut_{A}(\Ut^{A}-\Ct^A{}_B{}^B) - \frac{1}{n-1}\alphat\Sigmat_{AB}\Sigmat^{AB} \notag \\
        &\quad - \frac{(n-3)\alphat}{(n-1)(n-2)}\Tb_{00} - \frac{\alphat}{(n-1)(n-2)}\Tb_A{}^A - \frac{\alphat}{n-1}\Upsilon_{00}, \label{Hct-ev} \\
    \del{t}\Sigmat_{AB} &= -(n-1)\alphat\Hct\Sigmat_{AB} + \alphat\et_{\langle A}(\Ut_{B\rangle}) + \alphat\Ut_{\langle A}\Ut_{B\rangle} - \alphat\et^C(\Ct_{C\expval{AB}}) - \alphat\et_{\langle A}(\Ct_{B\rangle C}{}^C) \notag \\
        &\quad  - \frac{1}{4}\alphat\big( 2\Ct_{CD\langle A}\Ct_{B\rangle}{}^{CD} + 2\Ct_{CD\langle A}\Ct_{B\rangle}{}^{DC} - \Ct_{C\langle A}{}^D\Ct_{B\rangle}{}^C{}_D + \Ct_C{}^D{}_{\langle A}\Ct^C{}_{B\rangle D} \notag \\
        &\quad - \delta_{DE}\Ct_{C\langle A}{}^D\Ct^E{}_{B\rangle}{}^C \big) - \alphat\Ut^C\Ct_{C\expval{AB}} + \alphat\Ct_{CB}{}^B\Ct^C{}_{\expval{AB}} + \alphat(\Tb_{\expval{AB}} + \Upsilon_{\expval{AB}}), \label{Sigmat-ev}
\end{align}
along with two additional constraints:
\begin{align}
    \et_{B}(\Sigmat_A{}^B) - (n-2)\et_{A}(\Hct) + \Ct_{ABC}\Sigmat^{BC} - \Ct_{BC}{}^C\Sigmat_A{}^B - \Tb_{0A} - \Upsilon_{0A} &= 0, \label{momen-cnstr} \\
    2\et_A(\Ct^A{}_B{}^B) + (n-1)(n-2)\Hct^{2} - \Sigmat_{AB}\Sigmat^{AB} - \Ct_{AB}{}^B\Ct^A{}_C{}^C \qquad\qquad& \notag \\
        \quad - \frac{1}{4}\Ct_{ABC}(\Ct^{ABC}+\Ct^{BAC}+\Ct^{ACB}) - 2\Tb_{00} - (\Upsilon_{00}+\Upsilon_A{}^A) &= 0, \label{Hamil-cnstr}
\end{align}
where we note that they represent the momentum and the Hamiltonian constraints, respectively.

\subsection{Time slicing gauge} \label{time-gauge}
To complete the formulation of a closed system, it remains necessary to derive evolution equations for the lapse function $\alphat$ and $\Ut_A$. We begin by specifying the conformal factor $\Phi$ and the time function $t$. In accordance with the discussion in \cite[\S 1.1]{BeyerOliynyk:2024b}, we define
\begin{equation*}
    \Phi = \sqrt{\frac{2}{(n-2)(n-1)}}\,\varphi,
\end{equation*}
which leads to the conformal Einstein-scalar field equations:
\begin{align}
    \Gt_{ab} &= \frac{1}{\tau}\nablat_a\nablat_b\tau, \label{conf-ESF.1} \\
    \Box_{\gt}\tau &= 0, \label{conf-ESF.2}
\end{align}
where $\Gt_{ab}$ denotes the Einstein tensor of $\gt$ and the conformal scalar field $\tau$ is determined via
\begin{equation*}
    \varphi = \sqrt{\frac{n-1}{2(n-2)}}\ln(\tau),
\end{equation*}
which further implies that
\begin{equation} \label{Phi-fix}
    \Phi = \frac{1}{n-2}\ln(\tau).
\end{equation}
We use the conformal scalar field $\tau$ as the time function $t$, that is
\begin{equation} \label{time-fix}
    \tau = t.
\end{equation}
This choice of time-slicing gauge allows for the synchronization of the big bang singularity, as elaborated in \cite[\S 5.7]{BeyerOliynyk:2024b}. From the scalar field equation \eqref{ESF.2} together with \eqref{time-fix}, it follows that
\begin{equation*}
    \Box_{\gt}t = \Box_{\gt}\tau = 0.
\end{equation*}
Consequently, we obtain
\begin{equation*}
    \eta^{ab}\Bigl( \et_a\bigl(\et_b(t)\bigr) - \omega_a{}^c{}_b\et_c(t) \Bigr) = 0.
\end{equation*}
Then, using \eqref{connect-form}, the evolution equation for $\alphat$ becomes
\begin{equation} \label{alphat-ev}
    \del{t}\alphat = (n-1)\Hct\alphat^2.
\end{equation}
Dividing both sides of \eqref{alphat-ev} by $\alphat^2$ and applying $\et_A$ to the resulting equation yields
\begin{align*}
    \et_A\bigl(\et_0(\ln\alphat)\bigr) &= [\et_A,\et_0](\ln\alphat) + \et_0\bigl(\et_A(\ln\alphat)\bigr) \\
        &= -\frac{1}{\alphat}\bigl(\Ut_A\et_0(\alphat) - (\Hct\delta_A^B+\Sigmat_A{}^B)\et_B(\alphat)\bigr) + \et_0\Biggl(\frac{1}{\alphat}\et_A(\alphat)\Biggr) \\
        &= -(n-1)\Hct\Ut_A + (\Hct\delta_A^B+\Sigmat_A{}^B)\Ut_B + \et_0(\Ut_A) \\
        &= (n-1)\et_A(\Hct),
\end{align*}
where we have used \eqref{comm-decomp-0}, \eqref{St-decomp}, \eqref{Dft-cnstr} and \eqref{alphat-ev}. Rearranging the terms leads to the evolution equation for $\Ut_A$:
\begin{equation*}
    \del{t}\Ut_A = (n-1)\alphat\et_A(\Hct) + (n-2)\alphat\Hct\Ut_A - \alphat\Sigmat_A{}^B\Ut_B.
\end{equation*}

\subsection{A complete system and constraints}
Based on the definitions of $\Upsilon_{ab}$ and $\Tb_{ab}$ provided in \eqref{Upsilon-def} and \eqref{Tbab-def}, respectively, and given our gauge choices for the conformal factor and time slicing specified in \eqref{Phi-fix} and \eqref{time-fix}, we find that the components of $\Tb_{ab}$ and $\Upsilon_{ab}$ are expressed as follows:
\begin{equation} \label{Tb-Upsilon-cmpts}
    \begin{gathered} 
        \Tb_{00} = \frac{n-1}{2(n-2)\alphat^2 t^2}, \quad \Tb_{0A} = 0, \quad \Tb_A{}^A = \frac{(n-1)^2}{2(n-2)\alphat^2 t^2}, \quad \Tb_{\expval{AB}} = 0, \quad \Tb_{AB} = \frac{n-1}{2(n-2)\alphat^2 t^2}\delta_{AB}, \\
        \Upsilon_{00} = -\frac{(n-1)\Hct}{\alphat t} - \frac{n-1}{(n-2)\alphat^2 t^2}, \quad \Upsilon_{0A} = -\frac{1}{\alphat t}\Ut_A, \quad \Upsilon_A{}^A = -\frac{(n-1)\Hct}{\alphat t}, \quad \Upsilon_{\expval{AB}} = -\frac{1}{\alphat t}\Sigmat_{AB}.
    \end{gathered}
\end{equation}
Substituting the expressions in \eqref{Tb-Upsilon-cmpts} into \eqref{Hct-ev} to \eqref{Hamil-cnstr}, and incorporating the remaining evolution equations and constraints derived in Sections ~\ref{conf-ESF} and ~\ref{time-gauge}, we obtain the following tetrad formulation of the Einstein-scalar field equations:

\subsubsection*{Evolution equations}
\begin{align}
    \del{t}\et_{A}^{\Omega} &= -\alphat(\Hct\delta_{A}^{B} + \Sigmat_{A}{}^{B})\et_{B}^{\Omega}, \label{EEb.1} \\
    \del{t}\Ct_A{}^C{}_B &= -2\alphat\et_{[A}(\Hct)\delta_{B]}^C - 2\alphat\et_{[A}(\Sigmat_{B]}{}^C) - 2\alphat\Hct \Ut_{[A}\delta_{B]}^C - 2\alphat\Ut_{[A}\Sigmat_{B]}{}^C \notag \\
        &\quad -\alphat\Hct\Ct_A{}^C{}_B + 2\alphat\Sigmat_{[A}{}^D \Ct_{B]}{}^C{}_D + \alphat\Ct_A{}^D{}_B\Sigmat_D{}^C, \label{EEb.2} \\
    \del{t}\Hct &= -\alphat\Hct^{2} + \frac{1}{n-1}\alphat\et_{A}(\Ut^{A}) + \frac{1}{n-1}\alphat\Ut_{A}(\Ut^{A}-\Ct^A{}_B{}^B) - \frac{1}{n-1}\alphat\Sigmat_{AB}\Sigmat^{AB} + \frac{1}{t}\Hct, \label{EEb.3} \\
    \del{t}\Sigmat_{AB} &= -(n-1)\alphat\Hct\Sigmat_{AB} + \alphat\et_{\langle A}(\Ut_{B\rangle}) + \alphat\Ut_{\langle A}\Ut_{B\rangle} - \alphat\et^C(\Ct_{C\expval{AB}}) - \alphat\et_{\langle A}(\Ct_{B\rangle C}{}^C) \notag \\
        &\quad  - \frac{1}{4}\big( 2\alphat\Ct_{CD\langle A}\Ct_{B\rangle}{}^{CD} + 2\alphat\Ct_{CD\langle A}\Ct_{B\rangle}{}^{DC} - \alphat\Ct_{C\langle A}{}^D\Ct_{B\rangle}{}^C{}_D + \alphat\Ct_C{}^D{}_{\langle A}\Ct^C{}_{B\rangle D} \notag \\
        &\quad - \alphat\delta_{DE}\Ct_{C\langle A}{}^D\Ct^E{}_{B\rangle}{}^C \big) - \alphat\Ut^C\Ct_{C\expval{AB}} + \alphat\Ct_{CD}{}^D\Ct^C{}_{\expval{AB}} - \frac{1}{t}\Sigmat_{AB}, \label{EEb.4} \\
    \del{t}\alphat &= (n-1)\Hct \alphat^2, \label{EEb.5} \\
    \del{t}\Ut_A &= (n-1)\alphat\et_A(\Hct) + (n-2)\alphat\Hct\Ut_A - \alphat\Sigmat_A{}^B\Ut_B. \label{EEb.6}
\end{align}

\subsubsection*{Constraint equations}
\begin{align}
    \Aft_{AB}^\Omega &:= 2\et_{[A}(\et_{B]}^{\Omega}) - \Ct_A{}^C{}_B \et_C^\Omega = 0, \label{A-cnstr-2} \\
    \Bft_{AB} &:= 2\et_{[A}(\Ut_{B]}) - \Ct_A{}^C{}_B \Ut_C = 0, \label{B-cnstr-2} \\
    \Cft_{ABC}^D &:= \et_C(\Ct_A{}^D{}_B) + \et_A(\Ct_B{}^D{}_C) + \et_B(\Ct_C{}^D{}_A) \notag \\
        &\quad \, + \Ct_A{}^E{}_B\Ct_C{}^D{}_E + \Ct_B{}^E{}_C\Ct_A{}^D{}_E + \Ct_C{}^E{}_A\Ct_B{}^D{}_E = 0, \label{C-cnstr-2} \\
    \Dft_A &:= \et_{A}(\alphat) - \alphat\Ut_{A} = 0, \label{D-cnstr-2} \\
    \Mft_A &:= \et_{B}(\Sigmat_A{}^B) - (n-2)\et_{A}(\Hct) + \Ct_{ABC}\Sigmat^{BC} - \Ct_{BC}{}^C\Sigmat_A{}^B + \frac{1}{\alphat t}\Ut_A = 0, \label{M-cnstr-2}\\
    \Hft &:= 2\et_A(\Ct^A{}_B{}^B) + (n-1)(n-2)\Hct^{2} - \Sigmat_{AB}\Sigmat^{AB} - \Ct_{AB}{}^B\Ct^A{}_C{}^C \nonumber \\
        &\quad - \frac{1}{4}\Ct_{ABC}(\Ct^{ABC}+\Ct^{BAC}+\Ct^{ACB}) + \frac{2(n-1)}{\alphat t}\Hct = 0. \label{H-cnstr-2}
\end{align}

\section{Fuchsian Formulation} \label{Fuch-form}
In this section, we rewrite the frame formalism of the Einstein-scalar field equations as shown in \eqref{EEb.1} to \eqref{EEb.6} into a suitable form so that we can analyze and expect global existence. The strategy is similar to that in \cite[\S 4.1]{BOZ:2025}, which involves adding multiples of constraint equations and identifying an appropriate symmetrizer. It is worth mentioning here that it is also important to ensure that the solutions are indeed the solutions to the Einstein-scalar field equations, namely, we need the constraint equations to propagate, which we will address later in this article.

\subsection{A system in terms of the rescaled variables}
To recast the system into a suitable form, we introduce a set of rescaled variables designed to vanish on Kasner solutions. First, we observe that, in terms of the tetrad variables $\{\et_A^\Omega, \Ct_{ABC}, \Hct, \Sigmat_{AB}, \alphat, \Ut_A\}$, the conformal Kasner-scalar field solutions \eqref{conf-Kasner} are given by
\begin{equation} \label{Kasner-solns-B}
    \begin{gathered}
        \mathring{\et}_A^\Omega = t^{-r_A/2}\delta^\Omega_A, \quad  \Ct_{ABC} = 0, \quad \mathring{\Ut}_A = 0, \quad \mathring{\alphat} = t^{r_0/2}, \\
        \mathring{\Hct} = \frac{r_0}{2(n-1)}t^{-r_0/2-1}, \quad \mathring{\Sigmat}_{AB} = \Biggl(\frac{1}{2}r_{AB} - \frac{r_0}{2(n-1)}\delta_{AB}\Biggr)t^{-r_0/2-1},
    \end{gathered}
\end{equation}
where
\begin{equation} \label{rAB-def-1}
    r_{AB} = \begin{cases}
                 r_A, &\text{ if } A=B \\ 
                  0,  &\text{ if } A \neq B
             \end{cases}.
\end{equation}
As previously noted, spatial indices are raised and lowered using $\delta_{AB}$, i.e.
\begin{equation} \label{rAB-def-2}
    r_A^B = \delta^{BC}r_{AC}.
\end{equation}
The new variables are then defined as follows:
\begin{align}
    \Hc &= t \alphat\Hct - \frac{r_0}{2(n-1)}, \label{Hc-def} \\
    \Sigma_{AB} &= t \alphat\Sigmat_{AB} - \biggl(\frac{1}{2}r_{AB} - \frac{r_0}{2(n-1)}\delta_{AB}\biggr), \label{Sigma-def} \\
    \alpha &= t^{\ep_1}\alphat, \label{alpha-def} \\
    e_A^\Omega &= t^{\ep_2}\alphat \et_A^\Omega, \label{e-def} \\
    U_A &= t \alphat\Ut_A, \label{U-def} \\
    C_{ABC} &= t \alphat\Ct_{ABC}, \label{C-def}
\end{align}
where $\ep_1, \ep_2 \in\Rbb$ are constants that satisfy the following inequalities
\begin{equation} \label{ep-con-1}
    \ep_1+\frac{r_0}{2}>0, \quad 0<\ep_2<1 \quad \text{and} \quad \ep_2+\frac{r_0}{2}-\frac{r_A}{2}>0.
\end{equation}

\begin{rem} \label{W-Kasner-limit}
    If we set $W=(\Hc,\Sigma_{AB},\alpha,e_A^\Omega,U_A,C_{ABC})^{\tr}$, then from \eqref{Kasner-solns-B}, we see that
    \begin{equation*}
        \lim_{t\searrow0} W\big|_{Kasner} = \lim_{t\searrow0}\,(0, 0, t^{\ep_1+\frac{r_0}{2}}, t^{\ep_2+\frac{r_0}{2}-\frac{r_A}{2}}\delta_A^\Omega, 0, 0)^{\tr} = 0.
    \end{equation*}
    The importance of the above fact is that although we do not rescaled $\alphat$ and $\et_A^\Omega$ in a way that $\alpha$ and $e_A^\Omega$ will vanish on exact Kasner solutions, we can still ensure that if we make small enough perturbation around an exact Kasner solution at a small enough time, this perturbation can be treated as a small perturbation around the zero solution, and there is no loss of generality assuming the starting time is small enough. Indeed, by Cauchy stability, any perturbation of $W$ at a finite time $t_0>0$ can be used to get a perturbation of $W$ at any smaller time $t_1\in(0,t_0)$, which can also be guaranteed to be small enough provided the initial perturbation is small enough.
\end{rem}

A direct computation reveals that the evolution equations \eqref{EEb.1} to \eqref{EEb.6} and the constraint equations \eqref{A-cnstr-2} to \eqref{H-cnstr-2} can be expressed in terms of the rescaled variables defined in \eqref{Hc-def} to \eqref{C-def} as

\subsubsection*{Evolution equations}
\begin{align}
    \del{t}e_A^\Omega &= \frac{1}{t}\biggl(\ep_2 + \frac{r_0}{2} - \frac{r_A}{2}\biggr)e_A^\Omega + \frac{1}{t}\Bigl((n-2)\Hc\delta_A^B-\Sigma_A{}^B\Bigr)e_B^\Omega, \label{EEc.1} \\
    \del{t}C_{ABC} &= -2t^{-\ep_2}e_{[A}(\Hc)\delta_{C]B} - 2t^{-\ep_2}e_{[A}(\Sigma_{C]B}) + \frac{1}{t}\biggl(1+\frac{r_0}{2}-\frac{1}{2}(r_A+r_C-r_B)\biggr)C_{ABC} \notag \\
        &\quad + \frac{1}{t}[(\Hc+\Sigma)*C]_{ABC}\footnotemark, \label{EEc.2} \\
    \del{t}\Hc &= -\frac{2}{n-1}t^{-\ep_2}e_A(C^A{}_B{}^B) + \frac{1}{n-1}t^{-\ep_2}e_A(U^A) + \frac{1}{t}[(U+C)*(U+C)], \label{EEc.3} \\
    \del{t}\Sigma_{AB} &= t^{-\ep_2}e_{\langle A}(U_{B\rangle}) - t^{-\ep_2}e^C(C_{C\expval{AB}}) - t^{-\ep_2}e_{\langle A}(C_{B\rangle C}{}^C) + \frac{1}{t}[(U+C)*(U+C)]_{AB}, \label{EEc.4} \\
    \del{t}\alpha &= \frac{1}{t}\biggl(\ep_1 + \frac{r_0}{2}\biggr)\alpha + \frac{n-1}{t}\alpha \Hc, \label{EEc.5} \\
    \del{t}U_A &= (n-1)t^{-\ep_2}e_A(\Hc) + \frac{1}{t}\biggl(1+\frac{r_0}{2}-\frac{r_A}{2}\biggr)U_A + \frac{1}{t}[(\Hc+\Sigma)*U]_A. \label{EEc.6}
\end{align}
\footnotetext{For $*$ notation, see section~\ref{Index}.}

\subsubsection*{Constraints}
\begin{align}
    \Af_{AB}^\Omega &:= t^{1+\ep_2}\alphat^2\Aft_{AB}^\Omega = 2t^{1-\ep_2}e_{[A}(e_{B]}^\Omega) - 2U_{[A}e_{B]}^\Omega - C_A{}^C{}_B e_C^\Omega = 0, \label{A-cnstr-3} \\
    \Bf_{AB} &:= t^{2}\alphat^2\Bft_{AB} = 2t^{1-\ep_2}e_{[A}(U_{B]}) - 2U_{[A}U_{B]} - C_A{}^C{}_B U_C = 0, \label{B-cnstr-3} \\
    \Cf_{ABC}^D &:= t^2\alphat^2\Cft_{ABC}^D = t^{1-\ep_2}e_C(C_A{}^D{}_B) + t^{1-\ep_2}e_A(C_B{}^D{}_C) + t^{1-\ep_2}e_B(C_C{}^D{}_A) \notag \\
        &\qquad\qquad\qquad\quad - U_C C_A{}^D{}_B - U_A C_B{}^D{}_C - U_B C_C{}^D{}_A \notag \\
        &\qquad\qquad\qquad\quad + C_A{}^E{}_B C_C{}^D{}_E + C_B{}^E{}_C C_A{}^D{}_E + C_C{}^E{}_A C_B{}^D{}_E = 0, \label{C-cnstr-3} \\
    \Df_A &:= t^{1+\ep_1}\alphat^2\Dft_A = \alpha U_A - t^{1-\ep_2}e_A(\alpha) = 0, \label{D-cnstr-3} \\
    \Mf_A &:= t^2\alphat^2\Mft_A = t^{1-\ep_2}e_B(\Sigma_A{}^B) - (n-2)t^{1-\ep_2}e_A(\Hc) + \biggl(\frac{r_0}{2}+1-\frac{r_A}{2}\biggr)U_A \notag \\
        &\quad + \frac{1}{2}r^{BC}C_{ABC} - \frac{1}{2}r_A^B C_{BC}{}^C - \Sigma_A{}^B U_B + (n-2)\Hc U_A + C_{ABC}\Sigma^{BC} - C_{BC}{}^C\Sigma_A{}^B = 0, \label{M-cnstr-3} \\
    \Hf &:= t^2\alphat^2\Hft = 2t^{1-\ep_2}e_A(C^A{}_B{}^B) + (n-1)(n-2)\Hc^2 + (n-2)r_0\Hc + 2(n-1)\Hc - r_{AB}\Sigma^{AB} \notag \\
        &\quad - \Sigma_{AB}\Sigma^{AB} - 2U_A C^A{}_B{}^B - C_{AB}{}^B C^A{}_C{}^C - \frac{1}{4}C_{ABC}(C^{ABC}+C^{BAC}+C^{ACB}) = 0. \label{H-cnstr-3}
\end{align}

\begin{rem}
    To derive \eqref{EEc.3}, we first express \eqref{EEb.3} in terms of the rescaled variables which yields an equation of the form $\del{t}\Hc=Q$. We then add the term $-\frac{1}{(n-1)t}\Hf$ to the right hand side, resulting in $\del{t}\Hc=Q-\frac{1}{(n-1)t}\Hf$, to arrive at \eqref{EEc.3}. Since the Hamiltonian constraint $\Hf$ vanishes for initial data satisfying the constraint equations, the systems \eqref{EEb.3} and \eqref{EEc.3} are physically equivalent. For the same reason, we employ \eqref{D-cnstr-2} to eliminate $\et_A(\alphat)$. Another reason for the modification of $\del{t}\Hc$ is to remove problematic terms, like $\Hc^2$, in the Fuchsian perspective - that is, the product of the non-decaying terms. For further discussions, we refer to \cite[\S 3.4, \S 3.5]{BOOS:2021}.
\end{rem}

\subsection{Equations for lower order terms} \label{low-order}
To successfully derive the Fuchsian system, two distinct versions of the evolution equations \eqref{EEc.1} to \eqref{EEc.6} are required, following an approach analogous to that in \cite[\S 5]{BOZ:2025}. We now express the system of equations \eqref{EEc.1} to \eqref{EEc.6} in matrix form as
\begin{equation} \label{Fuch-1}
    \del{t}W = \frac{1}{t^{\ep_2}}e_D^\Lambda E^D \del{\Lambda}W + \frac{1}{t}\Bc\Pbb W + \frac{1}{t}F_l,
\end{equation}
where
\begin{equation} \label{W-def}
    W = (e_P^\Sigma, \alpha, C_{PQR}, U_P, \Hc, \Sigma_{PQ})^{\tr},
\end{equation}
\begin{equation} \label{Pbb-def}
    \Pbb = \diag(\delta_A^P\delta_\Sigma^\Omega, 1, \delta_A^P\delta_B^Q\delta_C^R, \delta_A^P, 0, 0),
\end{equation}
\begin{equation} \label{Bc-def}
    \Bc = \diag\Bigl(\bigl(\kappa_2\delta_A^P-\frac{1}{2}r_A^P\bigr)\delta_\Sigma^\Omega, \kappa_1, \kappa_0\delta_A^P\delta_B^Q\delta_C^R-\frac{1}{2}(r_A^P\delta_B^Q\delta_C^R+r_C^R\delta_A^P\delta_B^Q-r_B^Q\delta_A^P\delta_C^Q), \kappa_0\delta_A^P-\frac{1}{2}r_A^P, \kappa_0, \kappa_0\delta_A^P\delta_B^Q\Bigr),
\end{equation}
with
\begin{equation*} \label{kappa-def}
    \kappa_0 = 1+\frac{r_0}{2}, \quad \kappa_1 = \ep_1+\frac{r_0}{2}, \quad \kappa_2 = \ep_2+\frac{r_0}{2},
\end{equation*}
and
\begin{equation*}
    \renewcommand{\arraystretch}{1.2}
    E^D = \begin{bmatrix}
        0& 0& 0& 0& 0& 0& \\
        0& 0& 0& 0& 0& 0& \\
        0& 0& 0& 0& -2\delta_{[A}^D\delta_{C]B}& -2\delta_{[A}^D\delta_{C]}^{\langle P}\delta_B^{Q\rangle}& \\
        0& 0& 0& 0& (n-1)\delta_A^D& 0& \\
        0& 0& -\frac{2}{n-1}\delta^{D[P}\delta^{R]Q}& \frac{1}{n-1}\delta^{DP}& 0& 0& \\
        0& 0& -\delta^{D[P}\delta_{\langle A}^{R]}\delta_{B\rangle}^Q-\delta_{\langle A}^D\delta_{B\rangle}^{[P}\delta^{R]Q}& \delta_{\langle A}^D\delta_{B\rangle}^P& 0& 0
    \end{bmatrix},
\end{equation*}
\begin{equation} \label{Fl-def}
    F_l = \begin{bmatrix}
        [(\Hc+\Sigma)*e]& \\
        (n-1)\alpha\Hc& \\
        [(\Hc+\Sigma)*C]& \\
        [(\Hc+\Sigma)*U]& \\
        [(U+C)*(U+C)]& \\
        [(U+C)*(U+C)]&
    \end{bmatrix}.
\end{equation}
The conditions \eqref{r0-bound}, \eqref{rA-bound}, \eqref{sub-cond-2}, and \eqref{ep-con-1} ensure that the entries of the matrix $\Bc$ are strictly positive, which is important for our future analysis since it allows for the identification of the decaying terms in the system. For use below, we define the complementary projection operator of $\Pbb$ in \eqref{Pbb-def} by
\begin{equation*} \label{Pbbp-def}
    \Pbb^\perp = \id - \Pbb.
\end{equation*}
A key observation regarding the nonlinear terms $\Fb$ is that they satisfy the decomposition
\begin{equation} \label{Fb-decomp}
    \Pbb F = \Pbb W*\Pbb^{\perp}W \AND \Pbb^{\perp}F = \Pbb W*\Pbb W.
\end{equation}

Given that $\ep_2<1$ from \eqref{ep-con-1}, we choose an arbitrary positive number $\nu$ such that
\begin{equation} \label{nu-con}
    \ep_2 + \nu < 1,
\end{equation}
and define the lower order variables as the rescaled spatial derivatives of $W$ via
\begin{equation} \label{Wbc-def}
    W_\bc := t^{|\bc|\nu}\partial^{\bc}W = t^{|\bc|\nu}(\partial^{\bc}e_P^\Sigma, \partial^{\bc}\alpha, \partial^{\bc}C_{PQR}, \partial^{\bc}U_P, \partial^{\bc}\Hc, \partial^{\bc}\Sigma_{PQ})^{\tr}, \quad |\bc|<k,
\end{equation}
where $\bc=(\bc_1,...,\bc_{n-1})$ is a multiindex, cf. Section~\ref{Index} and $k\in\Nbb$ is a positive integer to be determined. We set $W_{(0,0,...,0)}=W$. Using \eqref{Fuch-1} and the fact that $B^D$, $\Bc$ and $\Pbb$ are constant matrices, the evolution equation for $W_\bc$ is given by
\begin{equation*}
    \del{t}W_\bc = \frac{1}{t}(|\bc|\nu\id+\Bc\Pbb)W_\bc + t^{|\bc|\nu-\ep_2}B^D \biggl(\sum_{\bc'\leq\bc}\binom{\bc}{\bc'}\partial^{\bc-\bc'}e_D^\Lambda\cdot\partial^{\bc'}\del{\Lambda}W\biggr) + t^{|\bc|\nu-1}\partial^\bc F, \quad |\bc|<k.
\end{equation*}
Since $F$ is quadratic in $W$ and satisfies \eqref{Fb-decomp}, the above equation can be written as
\begin{equation} \label{Fuch-2}
    \del{t}W_\bc = \frac{1}{t}(|\bc|\nu\id+\Bc\Pbb)W_\bc + \frac{1}{t^{\ep_2+\nu}}\sum_{\substack{\bc'\leq\bc\\|\ac|=1}}\Pbb W_{\bc-\bc'}*W_{\bc'+\ac} + \frac{1}{t}\sum_{\bc'\leq\bc}\Pbb W_{\bc-\bc'}*(\Pbb W_{\bc'}+\Pbb^{\perp}W_{\bc'}),  \quad |\bc|<k.
\end{equation}
We can see from \eqref{Fuch-2} that we are actually using ODEs to govern the lower order terms $W_\bc$ with $|\bc|<k$. From the definition of $\Bc$ in \eqref{Bc-def}, we know that the matrix $|\bc|\nu\id+\Bc\Pbb$ is a diagonal matrix. In the Fuchsian sense, see \cite[\S 3.4, \S 3.5]{BOOS:2021}, the components of $W_\bc$ corresponding to the zero diagonal entries in $|\bc|\nu\id+\Bc\Pbb$ are those variables that will decay to zero as $t\searrow0$, while the components corresponding to the strictly positive diagonal entries of $|\bc|\nu\id+\Bc\Pbb$ are the variables that converge to functions that do not vanish everywhere, and in the present context, the converging variables are $\Hc$ and $\Sigma_{AB}$.

\subsection{Equations for the highest order terms} \label{high-order}
We see in \eqref{Fuch-2} that there are terms with derivative up to order $k$. Then to obtain a closed system, in this section, we derive a symmetric hyperbolic formulation for the highest order term
\begin{equation*}
    W_\bc := t^{|\bc|\nu}\partial^{\bc}W = t^{|\bc|\nu}(\partial^{\bc}e_P^\Sigma, \partial^{\bc}\alpha, \partial^{\bc}C_{PQR}, \partial^{\bc}U_P, \partial^{\bc}\Hc, \partial^{\bc}\Sigma_{PQ})^{\tr}, \quad |\bc|=k.
\end{equation*}
To achieve this, we begin with a modified system. Specifically, we exploit the freedom to add multiples of the constraints. Rather than considering \eqref{EEc.2} and \eqref{EEc.6}, namely $\del{t}C_{ABC}=P_{ABC}$ and $\del{t}U_A=R_A$, we instead adopt the expressions $\del{t}C_{ABC}=P_{ABC}+\frac{\mu}{t}\Mf_{[A}\delta_{C]B}$ and $\del{t}U_A=R_A+\frac{\gamma}{t}\Mf_A$. This leads to the following modified system:
\begin{align}
    \del{t}e_A^\Omega &= \frac{1}{t}\biggl(\ep_2 + \frac{r_0}{2} - \frac{r_A}{2}\biggr)e_A^\Omega + \frac{1}{t}[(\Hc+\Sigma)*e]_A^\Omega, \label{EEd.1} \\
    \del{t}C_{ABC} &= -2t^{-\ep_2}e_{[A}(\Sigma_{C]B}) + \mu t^{-\ep_2}e_D(\Sigma_{[A}{}^D)\delta_{C]B} - (\mu n-2\mu+2)t^{-\ep_2}e_{[A}(\Hc)\delta_{C]B} \notag \\
        &\quad + \frac{1}{t}\biggl(1+\frac{r_0}{2}-\frac{1}{2}(r_A+r_C-r_B)\biggr)C_{ABC} - \frac{\mu}{2t}r^{DE}C_{ED[A}\delta_{C]B} - \frac{\mu}{2t}r_{[A}^D\delta_{C]B}C_{DE}{}^E \notag \\
        &\quad + \frac{\mu}{t}\Bigl(\frac{r_0}{2}+1-\frac{r_A}{2}\Bigr)U_{[A}\delta_{C]B} + \frac{1}{t}[(\Hc+\Sigma)*(U+C)]_{ABC}, \label{EEd.2} \\
    \del{t}\Hc &= -\frac{2}{n-1}t^{-\ep_2}e_A(C^A{}_B{}^B) + \frac{1}{n-1}t^{-\ep_2}e_A(U^A) + \frac{1}{t}[(U+C)*(U+C)], \label{EEd.3} \\
    \del{t}\Sigma_{AB} &= t^{-\ep_2}e_{\langle A}(U_{B\rangle}) - t^{-\ep_2}e^C(C_{C\expval{AB}}) - t^{-\ep_2}e_{\langle A}(C_{B\rangle C}{}^C) + \frac{1}{t}[(U+C)*(U+C)]_{AB}, \label{EEd.4} \\
    \del{t}\alpha &= \frac{1}{t}\biggl(\ep_1 + \frac{r_0}{2}\biggr)\alpha + \frac{n-1}{t}\alpha \Hc, \label{EEd.5} \\
    \del{t}U_A &= \bigl((n-1)-\gamma(n-2)\bigr)t^{-\ep_2}e_A(\Hc) + \gamma t^{-\ep_2}e_B(\Sigma_A{}^B) + \frac{\gamma+1}{t}\Big(1+\frac{r_0}{2}-\frac{r_A}{2}\Big)U_A \notag \\
        &\quad + \frac{\gamma}{2t}r^{BC}C_{ABC} - \frac{\gamma}{2t}r_A^B C_{BC}{}^C + \frac{1}{t}[(\Hc+\Sigma)*U]_A + \frac{1}{t}[\Sigma*C]_A. \label{EEd.6}
\end{align}
Collecting \eqref{EEd.1} to \eqref{EEd.6}, we express the evolution equations in matrix form as
\begin{equation} \label{Fuch-3}
    \del{t}W + \frac{1}{t^{\ep_2}}e_D^\Lambda A^D \del{\Lambda}W = \frac{1}{t}\Ac W + \frac{1}{t}F_h,
\end{equation}
where $W$ is defined in \eqref{W-def} and
\begin{equation*}
    A^D = - {\small\renewcommand{\arraystretch}{1.2}
    \begin{bmatrix}
        0& 0& 0& 0& 0& 0& \\
        0& 0& 0& 0& 0& 0& \\
        0& 0& 0& 0& -(\mu n-2\mu+2)\delta_{[A}^D\delta_{C]B}& -2\delta_{[A}^D\delta_{C]}^{\langle P}\delta_B^{Q\rangle}+\mu \delta^{D\langle P}\delta_{[A}^{Q\rangle}\delta_{C]B}& \\
        0& 0& 0& 0& \big((n-1)-\gamma(n-2)\big)\delta_A^D& \gamma \delta^{D\langle P}\delta_A^{Q\rangle}& \\
        0& 0& -\frac{2}{n-1}\delta^{D[P}\delta^{R]Q}& \frac{1}{n-1}\delta^{DP}& 0& 0& \\
        0& 0& -\delta^{D[P}\delta_{\langle A}^{R]}\delta_{B\rangle}^Q-\delta_{\langle A}^D\delta_{B\rangle}^{[P}\delta^{R]Q}& \delta_{\langle A}^D\delta_{B\rangle}^P& 0& 0
    \end{bmatrix}},
\end{equation*}
\begin{equation*}
    \Ac = \begin{bmatrix}
        \Bigl(\kappa_2\delta_A^P-\frac{1}{2}r_A^P\Bigr)\delta_\Sigma^\Omega& 0& 0& 0& 0& 0& \\
        0& \kappa_1& 0& 0& 0& 0& \\
        0& 0& \Ac_{33}& \Ac_{34}& 0& 0& \\
        0& 0& \Ac_{43}& \Ac_{44}& 0& 0& \\
        0& 0& 0& 0& 0& 0& \\
        0& 0& 0& 0& 0& 0&
    \end{bmatrix},
\end{equation*}
with
\begin{gather*}
    \Ac_{33} = \kappa_0\delta_A^P\delta_B^Q\delta_C^R-\frac{1}{2}(r_A^P\delta_B^Q\delta_C^R+r_C^R\delta_A^P\delta_B^Q-r_B^Q\delta_A^P\delta_C^Q) - \frac{\mu}{2}r^{PQ}\delta_{[A}^R\delta_{C]B} - \frac{\mu}{2}r_{[A}^P\delta_{C]B}\delta^{QR}, \\
    \Ac_{34} = \mu\Bigl(\kappa_0\delta_{[A}^P\delta_{C]B} - \frac{1}{2}r_{[A}^P\delta_{C]B}\Bigr), \\
    \Ac_{43} = \frac{\gamma}{2}r^{QR}\delta_A^P - \frac{\gamma}{2}r_A^P \delta^{QR}, \\
    \Ac_{44} = (\gamma+1)\Bigl(\kappa_0\delta_A^P - \frac{1}{2}r_A^P\Bigr),
\end{gather*}
and
\begin{equation*}
    F_h = \begin{bmatrix}
        [(\Hc+\Sigma)*e] \\
        (n-1)\alpha\Hc \\
        [(\Hc+\Sigma)*(U+C)] \\
        [(\Hc+\Sigma)*U] + [\Sigma*C] \\
        [(U+C)*(U+C)] \\
        [(U+C)*(U+C)]
    \end{bmatrix}.
\end{equation*}
Spatially differentiating \eqref{Fuch-3} yields
\begin{equation} \label{Fuch-4}
    \del{t}W_\bc + \frac{1}{t^{\ep_2}}e_D^\Lambda A^D \del{\Lambda}W_\bc = \frac{1}{t}(k\nu\id+\Ac)W_\bc + \frac{1}{t^{\ep_2+\nu}}\sum_{\substack{\bc'\leq\bc\\|\ac|=1}}\Pbb W_{\bc-\bc'}*W_{\bc'+\ac} + \frac{1}{t}\sum_{\bc'\leq\bc}\Pbb W_{\bc-\bc'}*(\Pbb W_{\bc'}+\Pbb^{\perp}W_{\bc'}),
\end{equation}
for $|\bc|=k$.

\subsection{Symmetrization} \label{Symmetrization}
This section draws upon the matrix identities \eqref{dd-1} to \eqref{dd-pi2} at several stages of the calculation. To transform \eqref{Fuch-4} into a symmetric hyperbolic system, we introduce a change of variables via
\begin{equation} \label{V-def}
    V = \begin{bmatrix}
        \delta_A^P\delta_\Omega^\Sigma& 0& 0& 0& 0& 0& \\
        0& 1& 0& 0& 0& 0& \\
        0& 0& a\delta_A^P\delta_B^Q\delta_C^R& b\delta_{[A}^P\delta_{C]B}& 0& 0& \\
        0& 0& c\delta_A^{[P}\delta^{R]Q}& d\delta_A^P& 0& 0& \\
        0& 0& 0& 0& 1& 0& \\
        0& 0& 0& 0& 0& \delta_A^P\delta_B^Q
    \end{bmatrix}.
\end{equation}
Following the discussion on the Schur complement in Chapter 0.8.5 of \cite{HornJohnson:2013}, to ensure the invertibility of $V$, the invertibility of $V$ requires the matrices $d\delta_A^P$ and
\begin{equation} \label{M*-def}
    M_* := a\delta_A^P\delta_B^Q\delta_C^R -bd^{-1}c\delta_{[A}^E\delta_{C]B}\delta_E^{[P}\delta^{R]Q}
\end{equation}
to be invertible. Under these conditions, the inverse of $V$ is given by
\begin{equation} \label{Vinvs-def}
    V^{-1} = \begin{bmatrix}
        \delta_A^P\delta_\Omega^\Sigma& 0& 0& 0& 0& 0& \\
        0& 1& 0& 0& 0& 0& \\
        0& 0& M_*^{-1}& -bd^{-1}M_*^{-1}\delta_{[A}^P\delta_{C]B}& 0& 0& \\
        0& 0& -cd^{-1}\delta_A^{[P}\delta^{R]Q}M_*^{-1}& d^{-1}\delta_A^P+bcd^{-2}\delta_A^{[D}\delta^{F]E}M_*^{-1}\delta_{[A}^P\delta_{C]B}& 0& 0& \\
        0& 0& 0& 0& 1& 0& \\
        0& 0& 0& 0& 0& \delta_A^P\delta_B^Q&
    \end{bmatrix}.
\end{equation}
Defining
\begin{equation} \label{Whbc-def}
    \Wt_\bc = V^{-1}W_\bc,
\end{equation}
and introducing the matrix
\begin{equation} \label{Sc-def}
    \Sc = \begin{bmatrix}
        \delta_A^P\delta_\Omega^\Sigma& 0& 0& 0& 0& 0& \\
        0& 1& 0& 0& 0& 0& \\
        0& 0& p\delta_A^P\delta_B^Q\delta_C^R& q\delta_{[A}^P\delta_{C]B}& 0& 0& \\
        0& 0& s\delta_A^{[P}\delta^{R]Q}& u\delta_A^P& 0& 0& \\
        0& 0& 0& 0& h& 0& \\
        0& 0& 0& 0& 0& l\delta_A^P\delta_B^Q&
    \end{bmatrix}
\end{equation}
with $h, l>0$, we use \eqref{Whbc-def} to rewrite \eqref{Fuch-4} as
\begin{equation} \label{Fuch-5}
    B^0\del{t}\Wt_\bc + \frac{1}{t^{\ep_2}}e_D^\Lambda B^D \del{\Lambda}\Wt_\bc = \frac{1}{t}\Bsc\Wt_\bc + \frac{1}{t^{\ep_2+\nu}}\sum_{\substack{\bc'\leq\bc\\|\ac|=1}}\Pbb W_{\bc-\bc'}*W_{\bc'+\ac} + \frac{1}{t}\sum_{\bc'\leq\bc}\Pbb W_{\bc-\bc'}*(\Pbb W_{\bc'}+\Pbb^{\perp}W_{\bc'}),
\end{equation}
for $|\bc|=k$, where
\begin{equation} \label{Bsc-def}
    \Bsc = k\nu B^0 + \Sc\Ac V,
\end{equation}
\begin{equation} \label{B0-def}
    B^0 := \Sc V = \begin{bmatrix}
       \delta_A^P\delta_\Omega^\Sigma& 0& 0& 0& 0& \\
        0& 1& 0& 0& 0& \\
        0& 0& \Btt_0& 0& 0& \\
        0& 0& 0& h& 0& \\
        0& 0& 0& 0& l\delta_A^P\delta_B^Q&
    \end{bmatrix}
\end{equation}
with
\begin{equation} \label{Btt0-def}
    \Btt_0 = \begin{pmatrix}
    ap\delta_A^P\delta_B^Q\delta_C^R+cq\delta_{[A}^E\delta_{C]B}\delta_E^{[P}\delta^{R]Q}& (bp+dq)\delta_{[A}^P\delta_{C]B}& \\
    (as+cu)\delta_A^{[P}\delta^{R]Q}& \Bigl(\frac{bs(n-2)}{2}+du\Bigr)\delta_A^P&
    \end{pmatrix},
\end{equation}
and
\begin{equation} \label{BD-def}
    B^D := \Sc A^D V = - \begin{bmatrix}
        0& 0& 0& \\
        0& 0& \Btt_1& \\
        0& \Btt_2& 0&
    \end{bmatrix}
\end{equation}
with
\begin{equation} \label{Btt1-def}
    \Btt_1 = \begin{pmatrix}
        \Bigl(-p(\mu n-2\mu+2)+q\bigl((n-1)-\gamma(n-2)\bigr)\Bigr)\delta_{[A}^D\delta_{C]B}& -2p M_3+\bigl(\mu p+q\gamma\bigr)\delta^{D\langle P}\delta_{[A}^{Q\rangle}\delta_{C]B}& \\
        \Bigl(-\frac{n-2}{2}(\mu n-2\mu+2)s+u\bigl((n-1)-\gamma(n-2)\bigr)\Bigr)\delta_A^D& \Bigl(s+\frac{n-2}{2}\mu s+u\gamma\Bigr)\delta^{D\langle P}\delta_A^{Q\rangle}&
    \end{pmatrix}
\end{equation}
and
\begin{equation} \label{Btt2-def}
    \Btt_2 = \begin{pmatrix}
        \Bigl(-\frac{2ah}{n-1}+\frac{ch}{n-1}\Bigr)\delta^{D[P}\delta^{R]Q},& \Bigl(-\frac{bh(n-2)}{n-1}+\frac{dh}{n-1}\Bigr)\delta^{DP}& \\
        -al M_3^{\tr}+(c-a)l \delta_{\langle A}^D\delta_{B\rangle}^{[P}\delta^{R]Q}& \Bigl(-\frac{bl(n-3)}{2}+dl\Bigr)\delta_{\langle A}^D\delta_{B\rangle}^P&
    \end{pmatrix}.
\end{equation}
We will examine $\Bsc$ after we assign proper values for the parameters. In \eqref{BD-def}, the requirement that $B^D$ be symmetric is equivalent to the condition that $\Btt_2^{\tr}=\Btt_1$. It is evident that the matrices $\delta_{[A}^D\delta_{C]}^{\langle P}\delta_B^{Q\rangle}$ and $\delta^{D\langle P}\delta_{[A}^{Q\rangle}\delta_{C]B}$ are linearly independent. Therefore, from \eqref{Btt1-def}, \eqref{Btt2-def} and the adjoint relations for $\delta_{[A}^D\delta_{C]}^{\langle P}\delta_B^{Q\rangle}$ and $\delta^{D\langle P}\delta_{[A}^{Q\rangle}\delta_{C]B}$ given in \eqref{M3-adj} and \eqref{M4-adj}, we have
\begin{align*}
    -p(\mu n-2\mu+2)+q\bigl((n-1)-\gamma(n-2)\bigr) &= -\frac{2ah}{n-1}+\frac{ch}{n-1}, \\ 
    2p &= al, \\ 
    \mu p+q\gamma &= (c-a)l, \\ 
    -\frac{n-2}{2}(\mu n-2\mu+2)s+u\bigl((n-1)-\gamma(n-2)\bigr) &= -\frac{bh(n-2)}{n-1}+\frac{dh}{n-1}, \\ 
    s+\frac{n-2}{2}\mu s+u\gamma &= -\frac{bl(n-3)}{2}+dl. 
\end{align*}
Solving the above equations for $p$, $q$, $s$, $u$ and $a$, we obtain that
\begin{align}
    p &= \frac{cl\Bigl(l(n-1)\bigl(1+n(\gamma-1)-2\gamma\bigr)+h\gamma\Bigr)}{4h\gamma+l(n-1)\bigl(2-6\gamma+n(2\gamma-\mu-2)+\mu\bigr)}, \label{p-def} \\
    q &= \frac{cl\Bigl(h(\mu-2)+l(n-1)\bigl(2+(n-2)\mu\bigr)\Bigr)}{4h\gamma+l(n-1)\bigl(2-6\gamma+n(2\gamma-\mu-2)+\mu\bigr)}, \label{q-def} \\
    s &= \frac{1}{(n-1)^2\bigl(2+(n-2)\mu\bigr)}\biggl(-2dl(n-1)\bigl(1+n(\gamma-1)-2\gamma\bigr) \notag \\
        &\quad - 2dh\gamma+b\Bigl(l(n^2-4n+3)\bigl(1+n(\gamma-1)-2\gamma\bigr)+2h(n-2)\gamma\Bigr)\biggr), \label{s-def} \\
    u &= \frac{-b(n-2)\bigl(2h+l(n^2-4n+3)\bigr)+2d\bigl(h+l(n^2-3n+2)\bigr)}{2(n-1)^2}, \label{u-def} \\
    a &= \frac{2c\Bigl(l(n-1)\bigl(1+n(\gamma-1)-2\gamma\bigr)+h\gamma\Bigr)}{4h\gamma+l(n-1)\bigl(2-6\gamma+n(2\gamma-\mu-2)+\mu\bigr)}. \label{a-def}
\end{align}
To simplify the analysis and streamline subsequent calculations, it is convenient to impose the following choices in advance:
\begin{equation} \label{hl-fix}
    h = 1 \AND l = \frac{1}{n-1}.
\end{equation}
Substituting \eqref{hl-fix} into \eqref{p-def} to \eqref{a-def} yields
\begin{gather}
    p = \frac{c(\gamma-1)}{(n-1)(2\gamma-\mu-2)}, \quad q = -\frac{c\mu}{(n-1)(2\gamma-\mu-2)}, \notag \\
    s = \frac{b\bigl(3+n(\gamma-1)-2\gamma\bigr)-2d(\gamma-1)}{(n-1)\bigl(2+(n-2)\mu\bigr)}, \quad u = \frac{2d-b(n-2)}{2(n-1)}, \quad a = \frac{2c(\gamma-1)}{2\gamma-\mu-2}. \label{pqsua}
\end{gather}
The symmetry of $B^0$, as implied by \eqref{B0-def} and \eqref{Btt0-def}, requires that
\begin{equation} \label{B0-sym1}
    bp+dq = as+cu.
\end{equation}
Inserting the expressions from \eqref{pqsua} into \eqref{B0-sym1} shows that the symmetry condition for $B^0$ is satisfied provided
\begin{equation} \label{B0-sym2}
    \frac{c\bigl(-4d(\gamma-1)+b(n-2)(2\gamma-\mu-2)\bigr)\bigl(4-2\gamma+(n-2)\mu\bigr)}{2(n-1)(2\gamma-\mu-2)\bigl(2+(n-2)\mu\bigr)} = 0.
\end{equation}
Solving \eqref{B0-sym2} for $\gamma$ gives
\begin{equation} \label{gamma-mu}
    \gamma = \frac{n-2}{2}\mu+2.
\end{equation}
It follows directly from \eqref{B0-def} that $B^0$ is positive definite as long as $\Btt_0$ is positive definite positive. Moreover, according to the discussion in \cite[Eq.~(7.7.5)]{HornJohnson:2013}, the positive definiteness of $\Btt_0$ requires
\begin{gather}
    \Biggl(\frac{bs(n-2)}{2}+du\Biggr)\delta_A^P > 0, \label{Btt0-pos-1a} \\
    ap\delta_A^P\delta_B^Q\delta_C^R+\biggl(cq - \Bigl(bp+dq\Bigr) \Bigl(\frac{bs(n-2)}{2}+du\Bigr)^{-1} \Bigl(as+cu\Bigr) \biggr)\delta_{[A}^E\delta_{C]B}\delta_E^{[P}\delta^{R]Q} > 0. \label{Btt0-pos-2a}
\end{gather}
The coefficient of the term $\delta_{[A}^E\delta_{C]B}\delta_E^{[P}\delta^{R]Q}$ in \eqref{Btt0-pos-2a} is rather lengthy and again to simplify this, we again make a priori choices using \eqref{gamma-mu} and set
\begin{equation} \label{gamma-mu-fix}
    \mu = 0 \AND \gamma = 2.
\end{equation}
With the values from \eqref{pqsua} and \eqref{gamma-mu-fix}, the first positive definiteness condition \eqref{Btt0-pos-1a} for $\Btt_0$ becomes equivalent to
\begin{equation} \label{Btt0-pos-1b}
    \frac{4d^2-4bd(n-2)+b^2(n^2-3n+2)}{4(n-1)} > 0.
\end{equation}
This inequality holds provided $2d\neq b(n-2)$, as can be seen from the lower bound
\begin{equation*}
    \frac{4d^2-4bd(n-2)+b^2(n^2-3n+2)}{4(n-1)} \geq \frac{\bigl(2d-b(n-2)\bigr)^2}{4(n-1)} \geq 0.
\end{equation*}
Under the same parameter choices, the coefficient of $\delta_{[A}^E\delta_{C]B}\delta_E^{[P}\delta^{R]Q}$ in \eqref{Btt0-pos-2a} simplifies to
\begin{equation*} \label{Btt0-pos-2a*}
    -\frac{b^2c^2}{(n-1)\Bigl(4d^2-4bd(n-2)+b^2(n-1)(n-2)\Bigr)}.
\end{equation*}
Then by Lemma~\ref{Mc-PD}, condition \eqref{Btt0-pos-2a} is satisfied provided that
\begin{equation} \label{Btt0-pos-2b}
    \frac{c^2}{n-1} > \Biggl| \frac{b^2c^2}{(n-1)\Bigl(4d^2-4bd(n-2)+b^2(n-1)(n-2)\Bigr)} \Biggr|.
\end{equation}
Finally, applying Lemma~\ref{Mc-PD} once more to \eqref{M*-def}, the invertibility of $M_*$ is guaranteed when
\begin{equation} \label{M*-invs}
    c > \frac{|bd^{-1}c|}{2}.
\end{equation}
In summary, with the assumptions $h=1$, $l=\frac{1}{n-1}$, $\mu=0$ and $\gamma=2$, the matrices $B^0$ and $B^D$ are symmetric as long as $p$, $q$, $s$, $u$ and $a$ are given by \eqref{pqsua}. Moreover, $B^0$ is positive definite as long as $b$, $c$ and $d$ satisfy the inequalities \eqref{Btt0-pos-1b}, \eqref{Btt0-pos-2b} and \eqref{M*-invs}.

An appropriate choice of values for the parameters is given by 
\begin{gather} \label{param-fix}
    a=n-1, \quad b=2, \quad c=n-1, \quad d=\frac{3}{2}, \notag \\
    p=\frac{1}{2}, \quad q=0, \quad s=\frac{2n-5}{2n-2}, \quad u=\frac{7-2n}{2n-2}, \notag \\
    h=1, \quad l=\frac{1}{n-1}, \quad \mu=0, \quad \gamma=2.
\end{gather}

\begin{rem}
    We mention here that it is permissible fixing the values of $h,l,\mu$ and $\gamma$ in advance, because the final choice of the symmetrization matrices actually relies on many factors. Indeed, our choice of the structures of the matrices $V$ and $\Sc$ is just one promising way to proceed the symmetrization and there can be many other choices. Furthermore, the equations for the parameters $\{a,b,c,d,p,q,s,u,h,l,\mu,\gamma\}$, resulting from \eqref{Btt1-def} and \eqref{Btt2-def}, form a system whose solution space has multiple branches and different branches can lead to different conditions on the parameters. Since we are just looking for one possible way of symmetrization, we are justified to fix some values in advance to simplify the process.
\end{rem}

Substituting \eqref{param-fix} into \eqref{V-def} and \eqref{Sc-def} and using \eqref{B0-def} to \eqref{Btt2-def} yields that
\begin{equation} \label{B0-fix}
    B^0 = \begin{bmatrix}
        \delta_A^P\delta_\Sigma^\Omega& 0& 0& 0& 0& 0& \\
        0& 1& 0& 0& 0& 0& \\
        0& 0& \frac{n-1}{2}\delta_A^P\delta_B^Q\delta_C^R& \delta_{[A}^P\delta_{C]B}& 0& 0& \\
        0& 0& \delta_A^{[P}\delta^{R]Q}& \frac{4n^2-24n+41}{4n-4}\delta_A^P& 0& 0& \\
        0& 0& 0& 0& 1& 0& \\
        0& 0& 0& 0& 0& \frac{1}{n-1}\delta_A^P\delta_B^Q&
    \end{bmatrix}
\end{equation}
and
\begin{equation} \label{BD-def}
    B^D = \Sc A^D V = \renewcommand{\arraystretch}{1.2}
    \begin{bmatrix}
        0& 0& 0& 0& 0& 0& \\
        0& 0& 0& 0& 0& 0& \\
        0& 0& 0& 0& -\delta_{[A}^D\delta_{C]B}& -\delta_{[A}^D\delta_{C]}^{\langle P}\delta_B^{Q\rangle}& \\
        0& 0& 0& 0& \frac{11-4n}{2n-2}\delta_A^D& \frac{9-2n}{2n-2}\delta^{D\langle P}\delta_A^{Q\rangle}& \\
        0& 0& -\delta^{D[P}\delta^{R]Q}& \frac{11-4n}{2n-2}\delta^{DP}& 0& 0& \\
        0& 0& -\delta^{D[P}\delta_{\langle A}^{R]}\delta_{B\rangle}^Q& \frac{9-2n}{2n-2}\delta_{\langle A}^D\delta_{B\rangle}^P& 0& 0&
    \end{bmatrix},
\end{equation}
From \eqref{Bsc-def}, the matrix $\Bsc$ satisfies
\begin{equation*} \label{knu-fix}
    \frac{1}{2}(\Bsc^{\tr}+\Bsc) = k\nu B^0 + \frac{1}{2}\bigl((\Sc\Ac V)^{\tr}+\Sc\Ac V\bigr).
\end{equation*}
Since $\nu>0$ and $B^0$ is positive definite by construction, there exists a $K\in\Zbb_{\geq0}$, such that for every $k\geq K$, $\frac{1}{2}(\Bsc^{\tr}+\Bsc)$ is positive definite. The parameter $k$ determines the highest order of spatial derivative of the fields that must be controlled. As we can see from the above discussion of the symmetrization process, the lower bound for $k$ depends on our way of symmetrizing the system \eqref{Fuch-4} and the choice for all the parameters that we set. Therefore, we do not derive an explicit lower bound for $k$ here, but simply indicating that there exists such a $k$ to ensure the positive definiteness of $\frac{1}{2}(\Bsc^{\tr}+\Bsc)$. It can also be verified that when the conformal Kasner exponents are zero in the matrix $\Ac$, i.e. when we are in the case of an FLRW metric, $k$ can be chosen to be 0. This means we do not need to control terms with higher order spatial derivatives and this is consistent with the results in \cite{BeyerOliynyk:2024b}.

\subsection{A bound on $B^0$}
For subsequent applications, we derive a rough estimate for the matrix $B^0$. Given any vector
\begin{equation*}
    w = (e_A^\Omega, \alpha, C_{ABC}, U_A, \Hc, \Sigma_{AB})^{\tr},
\end{equation*}
a direct calculation yields that
\begin{align*}
    w^{\tr}\Biggl(B^0-\frac{1}{2n^2}\id\Biggr)w &= \Biggl(1-\frac{1}{2n^2}\Biggr)|e|^2 + \Biggl(1-\frac{1}{2n^2}\Biggr)|\alpha|^2 + \Biggl(1-\frac{1}{2n^2}\Biggr)|\Hc|^2 + \Biggl(\frac{1}{n-1}-\frac{1}{2n^2}\Biggr)|\Sigma|^2 \\
        &\quad + \Biggl(\frac{n-1}{2}-\frac{1}{2n^2}\Biggr)|C|^2 + \Biggl(\frac{4n^2-24n+41}{4n-4}-\frac{1}{2n^2}\Biggr)|U|^2 + 2C^{ABC}U_{[A}\delta_{C]B}.
\end{align*}
We now examine the last three terms. Let
\begin{align*}
    I &:= \Biggl(\frac{n-1}{2}-\frac{1}{2n^2}\Biggr)|C|^2 + \Biggl(\frac{4n^2-24n+41}{4n-4}-\frac{1}{2n^2}\Biggr)|U|^2 + 2C^{ABC}U_{[A}\delta_{C]B} \\
        &\,\,= \Biggl(\frac{n-1}{2}-\frac{1}{2n^2}\Biggr)C_{ABC}C^{ABC} + \Biggl(\frac{4n^2-24n+41}{4n-4}-\frac{1}{2n^2}\Biggr)\frac{2}{n-2}\delta_{[A}^P\delta_{C]B}U_P\delta_Q^{[A}\delta^{C]B}U^Q \\
        &\,\,\quad + 2C^{ABC}\delta_{[A}^P\delta_{C]B}U_P,
\end{align*}
where we have used the identity \eqref{dd-1}. By the Cauchy-Schwarz inequality and the fact that
\begin{equation*}
    \Biggl(\frac{n-1}{2}-\frac{1}{2n^2}\Biggr) \Biggl(\frac{4n^2-24n+41}{4n-4}-\frac{1}{2n^2}\Biggr)\frac{2}{n-2} > 1 \quad \text{for} \quad n\geq4,
\end{equation*}
it follows that $I\geq0$. Consequently,
\begin{equation*}
    w^{\tr}\biggl(B^0-\frac{1}{2n^2}\id\biggr)w \geq 0,
\end{equation*}
and thus $B^0\geq\frac{1}{2n^2}\id$.

Since $\Btt_0$ is positive definite by \eqref{Btt0-pos-1a} and \eqref{Btt0-pos-2a}, we also have
\begin{equation*}
    \frac{n-1}{2}|C|^2 + \frac{4n^2-24n+41}{4n-4}|U|^2 + 2C^{ABC}U_{[A}\delta_{C]B} \geq 0.
\end{equation*}
Therefore,
\begin{align*}
    w^{\tr}(2n\id-B^0)w &\geq (2n-1)|e|^2 + (2n-1)|\alpha|^2 + (2n-1)|\Hc|^2 + \Biggl(2n-\frac{1}{n-1}\Biggr)|\Sigma|^2 \\
        &\quad + \frac{n-1}{2}|C|^2 + \frac{4n^2-24n+41}{4n-4}|U|^2 + 2C^{ABC}U_{[A}\delta_{C]B} \\
        &\geq 0,
\end{align*}
where the first inequality uses the elementary bounds
\begin{equation*}
    2n - \frac{n-1}{2} \geq \frac{n-1}{2} \AND 2n - \frac{4n^2-24n+41}{4n-4} \geq \frac{4n^2-24n+41}{4n-4}, \quad \text{for} \quad n\geq4.
\end{equation*}
To conclude, we obtain the following bounds for $B^0$
\begin{equation} \label{B0-bound}
    \frac{1}{2n^2}\id \leq B^0 \leq 2n\id.
\end{equation}

\subsection{The Fuchsian system}
Collecting \eqref{Fuch-2} and \eqref{Fuch-5}, the Fuchsian system of the frame formulation of the Einstein-scalar field equations is given by
\begin{equation} \label{Fuch-final*}
    \Bv^0\del{t}\Wv + \frac{1}{t^{\ep_2}}\Bv^\Lambda(\Wv) \del{\Lambda}\Wv = \frac{1}{t}\tilde{\Bcv}\Pv\Wv + \frac{1}{t^{\ep_2+\nu}}\Hv(\Wv) + \frac{1}{t}\Fv(\Wv),
\end{equation}
where
\begin{align}
    \Wv &= \Bigl( (W_\bc)_{|\bc|=0}, (W_\bc)_{|\bc|=1},..., (W_\bc)_{|\bc|=k-1}, (\Wt_\bc)_{|\bc|=k} \Bigr)^{\tr}, \label{Wv-def} \\
    \Bv^0 &= \diag(\id, \id,..., \id, B^0), \label{Bv0-def} \\
    \Bv^\Lambda(\Wv) &= \diag(0, 0,..., 0, e_D^\Lambda B^D), \label{BvLamb-def} \\
    \tilde{\Bcv} &= \diag(\Bc, \nu\id+\Bc\Pbb, 2\nu\id+\Bc\Pbb,..., (k-1)\nu\id+\Bc\Pbb, \Bsc), \label{Bcvt-def} \\
    \Pv &= \diag(\Pbb, \id, \id,..., \id), \label{Pv-def}
\end{align}
and $\Hv(\Wv)$ and $\Fv(\Wv)$ are quadratic in $\Wv$. Moreover, from the quadratic terms in \eqref{Fuch-2} and \eqref{Fuch-5}, $\Hv(\Wv)$ and $\Fv(\Wv)$ satisfy
\begin{equation*}
    \Hv(0) = 0,
\end{equation*}
and
\begin{equation} \label{Fv-decomp}
    \Fv(\Wv) = \Pv\Fv(\Wv) + \Pv^{\perp}\Fv(\Wv),
\end{equation}
where $\Pv\Fv(\Wv)$ can be further written as
\begin{equation} \label{PvFv}
    \Pv\Fv(\Wv) = \tilde{\Fv}(\Wv)\Pv\Wv,
\end{equation}
for some $\Fvt(\Wv)$ which is a matrix-valued map that depends linearly on $\Wv$ and satisfies
\begin{equation} \label{Fvt-Pv}
    \Fvt(0) = 0 \AND [\Fvt(\Wv), \Pv] = 0.
\end{equation}
This property of $\Fv(\Wv)$ is essential for us to get the global-in-time existence, see the proof of Proposition \ref{Fuch-global}.

\section{Fuchsian Stability}
In this section, we establish two distinct versions of nonlinear stability results of the trivial solution $\Wv=0$ to the system \eqref{Fuch-final*}, i.e. the global-in-space one and the local-in-space one. We begin by considering the global-in-space stability. For the purpose of being suitable for our subsequent analysis, we reformulate \eqref{Fuch-final*} as follows, supplemented by an initial condition:
\begin{align}
    \Bv^0\del{t}\Wv + \frac{1}{t^{\ep_2}}\Bv^\Lambda(\Wv) \del{\Lambda}\Wv &= \frac{1}{t}\Bcv(\Wv) \Pv\Wv + \frac{1}{t^{\ep_2+\nu}}\Hv(\Wv) + \frac{1}{t}\Pv^{\perp}\Fv(\Wv) \hspace{0.6cm} \text{in $M_{0,t_0}=(0,t_0]\times \Tbb^{n-1}$}, \label{Fuchsian-1} \\
    \Wv &= \Wv_0 \hspace{6.4cm} \text{in $\Sigma_{t_0}=\{t_0\}\times \Tbb^{n-1}$,} \label{Fuchsian-2}
\end{align}
where
\begin{equation} \label{Bcv-def}
    \Bcv(\Wv) = \tilde{\Bcv} + \Fvt(\Wv).
\end{equation}
It is clear that \eqref{Fuchsian-1} is just a direct rewriting of \eqref{Fuch-final*} using the properties given in \eqref{Fv-decomp} and \eqref{PvFv}. Under the assumption that the initial data $\Wv_0$ satisfy a suitable smallness condition, we state the existence and uniqueness of solutions to the above Fuchsian global initial value problem (GIVP). For use below, we also define
\begin{equation} \label{Pvperp-def}
    \Pv^{\perp} := \id - \Pv = \diag(\Pbb^{\perp}, 0, 0,..., 0).
\end{equation}

\begin{prop} \label{Fuch-global}
    Suppose $T_0>0$, $k_1\in\Zbb_{>\frac{n-1}{2}+1}$\footnote{Here, we do not require the regularity to be $k_1\in\Zbb_{>\frac{n-1}{2}+3}$ as in \cite[Thm.~3.8]{BOOS:2021}, since the power of $t$ in the spatial derivative matrix is $\ep_1<1$, see \eqref{ep-con-1} and \eqref{Fuchsian-1}, which allows us to avoid the necessity of \cite[Lemma 3.5]{BOOS:2021} and reduce the regularity.}, the conformal Kasner exponents $r_0$ and $r_A$, where $A=1,...,n-1$, satisfy the Kasner relations \eqref{Kasner-rels-B} and the subcritical condition \eqref{sub-cond-2}, $\ep_1,\ep_2\in\Rbb$ satisfy \eqref{ep-con-1}, $\nu\in\Rbb_{>0}$ satisfies \eqref{nu-con} and $k\in\Zbb_{\geq0}$ is chosen large enough such that for $\Bsc$ in \eqref{Bcvt-def}, $\frac{1}{2}(\Bsc^{\tr}+\Bsc)$ is positive definite, then there exists a $\delta_0>0$ such that for every $t_0\in(0,T_0]$ and every $\Wv_0\in H^{k_1}(\Tbb^{n-1})$ satisfying
    \begin{equation} \label{smallness-1}
        \norm{\Wv_0}_{H^{k_1}(\Tbb^{n-1})} < \delta_0,
    \end{equation}
    the GIVP \eqref{Fuchsian-1}-\eqref{Fuchsian-2} admits a unique solution
    \begin{equation*}
        \Wv \in C_b^0\bigl((0,t_0], H^{k_1}(\Tbb^{n-1})\bigr) \cap C^1\bigl((0,t_0], H^{k_1-1}(\Tbb^{n-1})\bigr)
    \end{equation*}
    such that $\lim_{t\searrow0}\Pv^{\perp}\Wv(t)$ exists in $H^{k_1-1}(\Tbb^{n-1})$ and is denoted by $\Pv^{\perp}\Wv(0)$. Moreover, there exists a $\zeta>0$ such that the solution $\Wv$ satisfies the energy estimate
    \begin{equation} \label{energy-1}
        \norm{\Wv(t)}_{H^{k_1}(\Tbb^{n-1})}^2 + \int_t^{t_0} \frac{1}{s}\norm{\Pv\Wv(s)}_{H^{k_1}(\Tbb^{n-1})}^2 \, ds \lesssim \norm{\Wv_0}_{H^{k_1}(\Tbb^{n-1})}^2
    \end{equation}
    and the decay estimate
    \begin{align}
        \norm{\Pv\Wv(t)}_{H^{k_1-1}(\Tbb^{n-1})} + \norm{\Pv^{\perp}\Wv(t)-\Pv^{\perp}\Wv(0)}_{H^{k_1-1}(\Tbb^{n-1})} \lesssim t^\zeta \label{decay-1}
    \end{align}
    for all $t\in(0,t_0]$. The implicit constants in the energy and decay estimates are independent of the the choice of $t_0\in(0,T_0]$ and $\Wv_0$ satisfying \eqref{smallness-1}.
\end{prop}

\begin{proof}
    We first verify that the coefficients in \eqref{Fuchsian-1} satisfy the conditions in \cite[\S 3.4]{BOOS:2021} and then the conclusion of the Proposition will be a direct consequence of \cite[Thm.~3.8]{BOOS:2021}. We begin by fixing the constant $p$ to be
    \begin{equation} \label{cnst-p}
        p = \min\{1-\ep_2, 1-(\ep_2+\nu)\}.
    \end{equation}
    We notice that the matrix $\Pv$ defined in \eqref{Pv-def} satisfies
    \begin{equation*}
        \Pv^2 = \Pv, \quad \Pv^{\tr} = \Pv, \quad \del{t}\Pv = 0 \quad \text{and} \quad \del{\Lambda}\Pv = 0,
    \end{equation*}
    so \cite[\S 3.4 (i*)]{BOOS:2021} is satisfied.
    
    It is clear from \eqref{B0-bound} and \eqref{Bv0-def} that $\Bv^0$ is symmetric and positive definite. From the choice for the constant $k$ and the definition in \eqref{Bcvt-def}, $\tilde{\Bcv}$ is also positive definite. Based on the fact that $\Fvt(\Wv)$ depends smoothly on $\Wv$ with $\Fvt(0)=0$, we see from \eqref{Bcv-def} that $\Bcv(0)=\tilde{\Bcv}$ and consequently
    \begin{equation*}
        \Bcv(\Wv) = \tilde{\Bcv} + \Ord(\Wv).
    \end{equation*}
    Thus for all $\Wv\in \mathbb{B}_R(\Vv)$ such that $R>0$ is chosen small enough and $\Vv$ is the vector space that $\Wv$ lies in, there exist constants $\tilde{\kappa}, \tilde{\gamma}$ so that $\tilde{\kappa}\id \leq\Bcv(\Wv)\leq \tilde{\gamma}\id$. This implies that
    \begin{equation*}
        \frac{1}{2n^2}\id \leq \Bv^0 \leq \frac{1}{\kappa}\Bcv(\Wv) \leq \gamma\id,
    \end{equation*}
    for some constants $\kappa,\gamma>0$. We also note from \eqref{Pbb-def}, \eqref{Bc-def}, \eqref{Bv0-def}, \eqref{Bcvt-def}, \eqref{Pv-def}, \eqref{Fvt-Pv} and \eqref{Pvperp-def} that
    \begin{equation*}
        [\Pv, \Bcv(\Wv)] = [\Pv, \tilde{\Bcv}+\Fvt(\Wv)] = 0
    \end{equation*}
    and
    \begin{equation*}
        \Pv\Bv^0\Pv^{\perp} = \Pv^{\perp}\Bv^0\Pv = 0.
    \end{equation*}
    Together with the fact that $\Bv^0$ and $\tilde{\Bcv}$ are both constant matrices, \cite[\S 3.4 (ii*)]{BOOS:2021} is satisfied.
    
    The source term in \eqref{Fuchsian-1}, $t^{-(\ep_2+\nu)}\Hv(\Wv)+t^{-1}\Pv^{\perp}\Fv(\Wv)$, which corresponds to the $F$ in \cite[\S 3.4]{BOOS:2021}, can be expanded as follows
    \begin{equation*}
        \frac{1}{t^{\ep_2+\nu}}\Hv(\Wv) + \frac{1}{t}\Pv^{\perp}\Fv(\Wv) = t^{-(1-p)}\tilde{\Hv} + t^{-(1-p)}\Hv_0 + t^{-(1-\frac{p}{2})}\Hv_1 + t^{-1}\Hv_2,
    \end{equation*}
    where $\tilde{\Hv}=\Hv_1=0$, $\Hv_0=t^{1-(\ep_2+\nu)-p}\Hv(\Wv)$ and $\Hv_2=\Pv^{\perp}\Fv(\Wv)$. By \eqref{cnst-p} and the fact that $\Hv(0)=0$, we see that $\Hv_0(\Wv)=\Ord(\Wv)$. By \eqref{Pv-def} and \eqref{Pvperp-def}, $\Pv\Hv_2=\Pv\Pv^{\perp}\Fv(\Wv)=0$. Finally it can be seen from \eqref{Fl-def} that $\Pv^{\perp}\Fv(\Wv)=\Ord(|\Pv\Wv|^2)$. Hence, there exists a constant $\lambda_3$ and $\lambda_1=\lambda_2=0$ so that \cite[\S 3.4 (iii*)]{BOOS:2021} is satisfied.

    The coefficient matrix of the spatial derivative $t^{-\ep_2}\Bv^\Lambda$, which corresponds to the $B$ in \cite[\S 3.4]{BOOS:2021}, can be expanded as follows
    \begin{equation*}
        \frac{1}{t^{\ep_2}}\Bv^\Lambda = t^{-(1-p)}\Bv_0^\Lambda + t^{-(1-\frac{p}{2})}\Bv_1^\Lambda + t^{-1}\Bv_2^\Lambda,
    \end{equation*}
    where $\Bv_1^\Lambda=\Bv_2^\Lambda=0$ and $\Bv_0^\Lambda=t^{1-\ep_2-p}\Bv^\Lambda$. Then we can set the constant $\alpha$ to be 0 and \cite[\S 3.4 (iv*)]{BOOS:2021} is satisfied.

    We now calculate $\Div\!\Bv$ using the definition in \cite[\S 3.4 (v*)]{BOOS:2021}, \eqref{Bv0-def} and \eqref{BvLamb-def}, to be
    \begin{equation*}
        \Div\!\Bv = \del{t}\Bv^0 + \del{\Lambda}\biggl(\frac{1}{t^{\ep_2}}\Bv^\Lambda(\Wv)\biggr) = \frac{1}{t^{\ep_2}}\Cv(\Wv, \Uv)\Big|_{\Uv=\del{\Lambda}\Wv}.
    \end{equation*}
    It is clear that $\Cv(0,0)=0$, which indicates that there exists a constant $\theta>0$ such that
    \begin{equation*}
        \Bigl| \frac{1}{t^{-(1-\ep_2-p)}}\Cv(\Wv, \Uv) \Bigr| \leq \theta
    \end{equation*}
    for all $t\in[0, T_0]$ and $\Wv,\Uv \in \mathbb{B}_R(\Vv)$ provided that $R$ is chosen small enough. For the rest of the constants, we can simply set $\beta_j$, $j=0,1,...,7$ to be 0 and \cite[\S 3.4 (v*)]{BOOS:2021} is satisfied.

    For a final step, we may restrict the value of $\lambda_3$ so that $2n^2\lambda_3<\kappa$ and the condition in \cite[Eq. (3.58)]{BOOS:2021} is satisfied. This can be achieved by choosing $R$ small enough when verifying \cite[\S 3.4 (iii*)]{BOOS:2021}.

    We conclude that the coefficients in \eqref{Fuchsian-1} satisfy all the assumptions in \cite[\S 3.4]{BOOS:2021} and as a consequence, there exists a unique solution
    \begin{equation*}
        \Wv \in C_b^0\bigl((0,t_0], H^{k_1}(\Tbb^{n-1})\bigr) \cap C^1\bigl((0,t_0], H^{k_1-1}(\Tbb^{n-1})\bigr)
    \end{equation*}
    of the GIVP \eqref{Fuchsian-1} to \eqref{Fuchsian-2} such that the limit $\lim_{t\searrow0}\Pv^{\perp}\Wv(t)$, denoted by $\Pv^{\perp}\Wv(0)$, exists in $H^{k_1-1}(\Tbb^{n-1})$. And the energy estimate \eqref{energy-1} and the decay estimate \eqref{decay-1} are direct results of \cite[Thm.~3.8]{BOOS:2021}. It is also important for us to notice that our choice of the constant $\delta_0$ does not depend on the value of $t_0\in(0,T_0]$. Indeed, for any $t_0\in(0,T_0]$, we can see from the above discussion that we may choose the same constants and all the coefficients assumptions will still be satisfied. More precisely, our choices for the above constants determine the value of $\delta_0$, which guarantees the existence of the solution on the interval $(0,T_0]$, and clearly also yields solution on any smaller interval $(0,t_0]$. 
\end{proof}

Having established the past global-in-space stability, we now turn to the past local-in-space stability result. Specifically, we examine the stability of the trivial solution $\Wv=0$ to the following Fuchsian initial value problem on a truncated cone domain
\begin{align}
    \Bv^0\del{t}\Wv + \frac{1}{t^{\ep_2}}\Bv^\Lambda(\Wv) \del{\Lambda}\Wv &= \frac{1}{t}\Bcv(\Wv) \Pv\Wv + \frac{1}{t^{\ep_2+\nu}}\Hv(\Wv) + \frac{1}{t}\Pv^{\perp}\Fv(\Wv) \hspace{0.6cm} \text{in $\Omega_{\Icv}$}, \label{Fuchsian-3} \\
    \Wv &= \Wv_0 \hspace{6.4cm} \text{in $\{t_0\}\times \mathbb{B}_{\rho_0}$}, \label{Fuchsian-4}
\end{align}
where the index is given by $\Icv=(t_0,0,\rho_0,\rho_1,\ep_2)$, see Section~\ref{Domains} for the definition of a truncated cone domain.

\begin{prop} \label{Fuch-local}
    Suppose $T_0>0$, $k_1\in\Zbb_{>\frac{n-1}{2}+1}$\footnote{The reason for the reduction of the regularity is the same as in Proposition~\ref{Fuch-global}.}, the conformal Kasner exponents $r_0$ and $r_A$, where $A=1,...,n-1$, satisfy the Kasner relations \eqref{Kasner-rels-B} and the subcritical condition \eqref{sub-cond-2}, $\ep_1,\ep_2\in\Rbb$ satisfy \eqref{ep-con-1}, $\nu\in\Rbb_{>0}$ satisfies \eqref{nu-con}, $k\in\Zbb_{\geq0}$ is chosen large enough such that for $\Bsc$ in \eqref{Bcvt-def}, $\frac{1}{2}(\Bsc^{\tr}+\Bsc)$ is positive definite and $0<\rho_0<L$, then there exists a $\delta_0>0$ such that for every $t_0\in(0,T_0]$ and every $\Wv_0\in H^{k_1}(\mathbb{B}_{\rho_0})$ satisfying
    \begin{equation} \label{smallness-2}
        \norm{\Wv_0}_{H^{k_1}(\mathbb{B}_{\rho_0})} < \delta_0,
    \end{equation}
    the GIVP \eqref{Fuchsian-3}-\eqref{Fuchsian-4} admits a unique solution $\Wv$ such that
    \begin{equation*}
        \Wv(t) \in H^{k_1}(\mathbb{B}_{\rho(t)}) \AND \del{t}\Wv(t) \in H^{k_1-1}(\mathbb{B}_{\rho(t)}),
    \end{equation*}
    where $\rho(t)=\frac{\rho_1(t^{1-\ep_2}-t_0^{1-\ep_2})}{1-\ep_2}+\rho_0$, for all $t\in(0,t_0]$. Moreover, the limit $\lim_{t\searrow0}\Pv^{\perp}\Wv(t)$, denoted by $\Pv^{\perp}\Wv(0)$, exists in $H^{k_1-1}(\mathbb{B}_{\tilde{\rho}_0})$, with $\tilde{\rho}_0=\rho_0-\frac{\rho_1 t_0^{1-\ep_2}}{1-\ep_2}$, and there exists a $\zeta>0$ such that the solution $\Wv$ satisfies the energy estimate
    \begin{equation*}
        \norm{\Wv(t)}_{H^{k_1}(\mathbb{B}_{\rho(t)})}^2 + \int_t^{t_0} \frac{1}{s}\norm{\Pv\Wv(s)}_{H^{k_1}(\mathbb{B}_{\rho(s)})}^2 \, ds \lesssim \norm{\Wv_0}_{H^{k_1}(\mathbb{B}_{\rho_0})}^2
    \end{equation*}
    and the decay estimate
    \begin{equation*}
        \norm{\Pv\Wv(t)}_{H^{k_1-1}(\mathbb{B}_{\rhot_0})} + \norm{\Pv^{\perp}\Wv(t)-\Pv^{\perp}\Wv(0)}_{H^{k_1-1}(\mathbb{B}_{\rhot_0})} \lesssim t^{\zeta}
    \end{equation*}
    for all $t\in(0,t_0]$. And the very first component of $\Wv$, i.e. $e=(e_P^\Sigma)$, see \eqref{W-def}, \eqref{Wbc-def} and \eqref{Wv-def}, satisfies the following inequality
    \begin{equation*}
        \sup_{(t,x)\in \Omega_{\Icv}} |e(t,x)|\footnote{Here $e(t,x)$ is to highlight the dependence of the norm $|e|$ on $t$ and $x$.} \leq \frac{\rho_1}{6n^3}.
    \end{equation*}
    The implicit constants in the energy and decay estimates are independent of the the choice of $t_0\in(0,T_0]$ and $\Wv_0$ satisfying \eqref{smallness-2}.
\end{prop}

\begin{proof}
Observe first that the coefficients of \eqref{Fuchsian-3} are identical to those in \eqref{Fuchsian-1}, so then the existence of a solution can be established by invoking Proposition~\ref{Fuch-global}. For any given initial data $\Wv_0\in H^{k_1}(\mathbb{B}_{\rho_0})$, due to the argument in Section~\ref{extension}, there exists an extension operator $E_{\rho_0}$ such that
\begin{equation*}
    \tilde{\Wv}_0 := E_{\rho_0}(\Wv_0) \in H^{k_1}(\Tbb^{n-1})
\end{equation*}
and
\begin{equation} \label{Wv0}
    \tilde{\Wv}_0 = \Wv_0 \quad \text{a.e. in $\mathbb{B}_{\rho_0}$} \AND \norm{\tilde{\Wv}_0}_{H^{k_1}(\Tbb^{n-1})} \leq C_0 \norm{\Wv_0}_{H^{k_1}(\mathbb{B}_{\rho_0})}
\end{equation}
for some constant $C_0>0$. We take $\tilde{\Wv}_0$ to be the initial data of the system \eqref{Fuchsian-1} to \eqref{Fuchsian-2}. According to proposition~\ref{Fuch-global}, there exists a $\tilde{\delta}_0>0$ such that for any $t_0\in(0,T_0]$, as long as
\begin{equation*}
    \norm{\Wv_0}_{H^{k_1}(\mathbb{B}_{\rho_0})} < \frac{\tilde{\delta}_0}{C_0},
\end{equation*}
then $\norm{\tilde{\Wv}_0}_{H^{k_1}(\Tbb^{n-1})}<\tilde{\delta}_0$ holds, and there exists a unique solution
\begin{equation} \label{Wvt-regu}
    \tilde{\Wv} \in C_b^0\bigl((0,t_0], H^{k_1}(\Tbb^{n-1})\bigr) \cap C^1\bigl((0,t_0], H^{k_1-1}(\Tbb^{n-1})\bigr).
\end{equation}
Moreover, the limit $\lim_{t\searrow0}\Pv^{\perp}\tilde{\Wv}(t)$, denoted by $\Pv^{\perp}\tilde{\Wv}(0)$, exists in $H^{k_1-1}(\Tbb^{n-1})$, and there exists a $\zeta>0$ such that the solution $\tilde{\Wv}$ satisfies the energy estimate
\begin{equation} \label{Wvt-energy}
    \norm{\tilde{\Wv}(t)}_{H^{k_1}(\Tbb^{n-1})}^2 + \int_t^{t_0} \frac{1}{s}\norm{\Pv\tilde{\Wv}(s)}_{H^{k_1}(\Tbb^{n-1})}^2 \lesssim \norm{\tilde{\Wv}_0}_{H^{k_1}(\Tbb^{n-1})}^2
\end{equation}
and the decay estimate
\begin{equation} \label{Wvt-decay}
    \norm{\Pv\tilde{\Wv}(t)}_{H^{k_1-1}(\Tbb^{n-1})} + \norm{\Pv^{\perp}\tilde{\Wv}(t)-\Pv^{\perp}\tilde{\Wv}(0)}_{H^{k_1-1}(\Tbb^{n-1})} \lesssim t^{\zeta}
\end{equation}
for all $t\in(0,t_0]$.

Consequently, we have established a solution on $M_{0,t_0}$ and the restriction
\begin{equation} \label{Wv-local}
    \Wv := \tilde{\Wv}|_{\Omega_{\Icv}}
\end{equation}
yields a solution to the GIVP \eqref{Fuchsian-3} to \eqref{Fuchsian-4}. From \eqref{Wvt-regu} and \eqref{Wv-local}, it is clear that
\begin{equation*}
    \Wv(t) \in H^{k_1}(\mathbb{B}_{\rho(t)}) \AND \del{t}\Wv(t) \in H^{k_1-1}(\mathbb{B}_{\rho(t)}).
\end{equation*}
By the existence of $\Pv^{\perp}\tilde{\Wv}(0)$ in $H^{k_1-1}(\Tbb^{n-1})$, the limit $\lim_{t\searrow0}\Pv^{\perp}\Wv(t)$ also exists in $H^{k_1-1}(\mathbb{B}_{\tilde{\rho}_0})$ and is denoted by $\Pv^{\perp}\Wv(0)$. Using the fact that $\Pv$ is a constant matrix, we derive the following two inequalities
\begin{align*}
    \norm{\Wv(t)}_{H^{k_1}(\mathbb{B}_{\rho(t)})}^2 + \int_t^{t_0} \frac{1}{s}\norm{\Pv\Wv(s)}_{H^{k_1}(\mathbb{B}_{\rho(s)})}^2 \, ds &\leq \norm{\tilde{\Wv}(t)}_{H^{k_1}(\Tbb^{n-1})}^2 + \int_t^{t_0} \frac{1}{s}\norm{\Pv\tilde{\Wv}(s)}_{H^{k_1}(\Tbb^{n-1})}^2, \\
    \norm{\Pv\Wv(t)}_{H^{k_1-1}(\mathbb{B}_{\rhot_0})} + \norm{\Pv^{\perp}\Wv(t)&-\Pv^{\perp}\Wv(0)}_{H^{k_1-1}(\mathbb{B}_{\rhot_0})} \\
    \leq \norm{\Pv\tilde{\Wv}(t)}_{H^{k_1-1}(\Tbb^{n-1})} &+ \norm{\Pv^{\perp}\tilde{\Wv}(t)-\Pv^{\perp}\tilde{\Wv}(0)}_{H^{k_1-1}(\Tbb^{n-1})}.
\end{align*}
Then from \eqref{Wv0}, \eqref{Wvt-energy} and \eqref{Wvt-decay}, the function $\Wv(t)$ satisfies the corresponding energy estimate
\begin{equation} \label{Wv-energy}
    \norm{\Wv(t)}_{H^{k_1}(\mathbb{B}_{\rho(t)})}^2 + \int_t^{t_0} \frac{1}{s}\norm{\Pv\Wv(s)}_{H^{k_1}(\mathbb{B}_{\rho(s)})}^2 \, ds \lesssim \norm{\Wv_0}_{H^{k_1}(\mathbb{B}_{\rho_0})}^2
\end{equation}
and the decay estimate
\begin{equation} \label{Wv-decay}
    \norm{\Pv\Wv(t)}_{H^{k_1-1}(\mathbb{B}_{\rhot_0})} + \norm{\Pv^{\perp}\Wv(t)-\Pv^{\perp}\Wv(0)}_{H^{k_1-1}(\mathbb{B}_{\rhot_0})} \lesssim t^{\zeta}
\end{equation}
for all $t\in(0,t_0]$.

Finally, we establish the uniqueness of the solution. Let $n_\mu$ denote a normal vector to the side piece of the domain $\Omega_{\Icv}$ and define
\begin{equation} \label{Bvbar-def}
    \bar{\Bv} := n_0\Bv^0 + \frac{1}{t^{\ep_2}}n_\Lambda \Bv^{\Lambda}(\Wv).
\end{equation}
It follows from \eqref{normal} that $n_\mu$ is given by
\begin{equation} \label{normal-fix}
    n_\mu = -\frac{\rho_1}{t^{\ep_2}}\delta_\mu^0 + \frac{1}{|x|}x^\Lambda \delta_{\mu\Lambda}.
\end{equation}
Then employing \eqref{Bv0-def}, \eqref{BvLamb-def}, \eqref{Bvbar-def} and \eqref{normal-fix}, we obtain that
\begin{equation} \label{Bvbar-fix}
    \bar{\Bv} = \diag\Biggl(  -\frac{\rho_1}{t^{\ep_2}}\id, -\frac{\rho_1}{t^{\ep_2}}\id,..., -\frac{\rho_1}{t^{\ep_2}}\id, -\frac{\rho_1}{t^{\ep_2}}B^0+\frac{1}{t^{\ep_2}}\frac{x_\Lambda}{|x|}e_D^\Lambda B^D \Biggr).
\end{equation}
We focus first on the last component of $\bar{\Bv}$. From \eqref{B0-fix} and \eqref{BD-def}, for any vector of the form\footnote{Here, we do not assume any symmetry or trace-free property of $C_{ABC}$ and $\Sigma_{AB}$ in $v$.}
\begin{equation*}
    v = (e_A^\Omega, \alpha, C_{ABC}, U_A, \Hc, \Sigma_{AB})^{\tr},
\end{equation*}
a simple calculation yields
\begin{align*}
    v^{\tr}B^0 v &= |e|^2 + |\alpha|^2 + \frac{n-1}{2}|C|^2 + \frac{4n^2-24n+41}{4n-4}|U|^2 + |\Hc|^2 + \frac{1}{n-1}|\Sigma|^2 + 2C^{ABC}U_{[A}\delta_{C]B}, \\
    v^{\tr}B^D v &= -2C^{ABC}\delta_{[A}^D\delta_{C]B}\Hc - 2C^{ABC}\delta_{[A}^D\delta_{C]}^{\langle P}\delta_B^{Q\rangle}\Sigma_{PQ} + \frac{11-4n}{n-1}U^A\delta_A^D\Hc + \frac{9-2n}{n-1}U^A\delta^{D\langle P}\delta_A^{Q\rangle}\Sigma_{PQ}.
\end{align*}
Due to the bound of $B^0$ given in \eqref{B0-bound}, we derive the estimate
\begin{equation} \label{uniq-est-1}
    -\frac{\rho_1}{t^{\ep_2}} v^{\tr}B^0 v < -\frac{\rho_1}{t^{\ep_2}} \Biggl( |e|^2 + |\alpha|^2 + \frac{1}{2n^2}|C|^2 + \frac{1}{2n^2}|U|^2 + |\Hc|^2 + \frac{1}{n-1}|\Sigma|^2 \Biggr).
\end{equation}
Proceeding, we set
\begin{equation} \label{fD-def}
    f_D = \frac{x_\Lambda}{|x|} e_D^\Lambda,
\end{equation}
and substitute \eqref{fD-def} into the expression for $v^{\tr}B^D v$, which leads to
\begin{equation*}
    \frac{x_\Lambda}{|x|}e_D^\Lambda v^{\tr}B^D v = -2C^{ABC}f_{[A}\delta_{C]B}\Hc - 2C^{ABC}f_{[A}\delta_{C]}^{\langle P}\delta_B^{Q\rangle}\Sigma_{PQ} + \frac{11-4n}{n-1}U^A f_{A}\Hc + \frac{9-2n}{n-1}U^A f^{\langle P}\delta_A^{Q\rangle}\Sigma_{PQ}.
\end{equation*}
Each term is now controlled using the Cauchy-Schwarz inequality:
\begin{align}
    -2C^{ABC}f_{[A}\delta_{C]B}\Hc &= -\Hc C^{ABC}f_A\delta_{CB} + \Hc C^{ABC}f_C\delta_{AB} \leq 2\sqrt{(n-1)}|\Hc||C||f|, \label{uniq-est-2} \\
    -2C^{ABC}f_{[A}\delta_{C]}^{\langle P}\delta_B^{Q\rangle}\Sigma_{PQ} &= -\frac{1}{2}C^{ABC}f_A\Sigma_{CB} - \frac{1}{2}C^{ABC}f_A\Sigma_{BC} + \frac{1}{n-1}C^{ABC}f_A\delta_{BC}\Sigma_P{}^P \notag \\
        &\quad + \frac{1}{2}C^{ABC}f_C\Sigma_{AB} + \frac{1}{2}C^{ABC}f_C\Sigma_{BA} - \frac{1}{n-1}C^{ABC}f_C\delta_{AB}\Sigma_P{}^P \notag \\
        &\leq 4|C||f||\Sigma|, \label{uniq-est-3} \\
    \frac{11-4n}{n-1}U^A f_{A}\Hc &\leq \frac{4n-11}{n-1}|\Hc||U||f|, \label{uniq-est-4} \\
    \frac{9-2n}{n-1}U^A f^{\langle P}\delta_A^{Q\rangle}\Sigma_{PQ} &= \frac{9-2n}{n-1}\Biggl( \frac{1}{2}f^PU^Q\Sigma_{PQ} + \frac{1}{2}f^QU^P\Sigma_{PQ} - \frac{1}{n-1}f_AU^A\Sigma_P{}^P \Biggr) \notag \\
        &\leq \Biggl| \frac{9-2n}{n-1}\Biggl(\frac{1}{\sqrt{n-1}}+1\Biggr) \Biggr| |f||U||\Sigma|, \label{uniq-est-5}
\end{align}
where we have used the following identity and estimate
\begin{align*}
    \bigl(f_A\delta_{CB}f^A\delta^{CB}\bigr)^{\frac{1}{2}} = \sqrt{n-1}\,|f| \AND |\Sigma_P{}^P| = |\delta^{PQ}\Sigma_{PQ}| \leq \sqrt{n-1}\,|\Sigma|.
\end{align*}
An estimate for $|f|$ is also required. Applying inequality \eqref{CS-ineq}, we see that
\begin{align*}
    |f|^2 &= \sum_{D=1}^{n-1}\Biggl(\frac{x_\Lambda}{|x|} e_D^\Lambda\Biggr)^2 \leq \sum_{D=1}^{n-1} (n-1)\frac{|x_\Lambda|^2}{|x|^2}|e_D^\Lambda|^2 \leq \sum_{D=1}^{n-1}\sum_{\Lambda=1}^{n-1} (n-1)|e_D^\Lambda|^2,
\end{align*}
and thus,
\begin{equation} \label{uniq-est-6}
    |f| \leq \sqrt{n-1}\,|e|.
\end{equation}
Gathering the estimates \eqref{uniq-est-2} to \eqref{uniq-est-5} with \eqref{uniq-est-6} and applying Young's inequality, $ab\leq a^2/2+b^2/2$, we arrive at the following estimate:
\begin{align}
     \frac{x_\Lambda}{|x|}e_D^\Lambda v^{\tr}B^D v &\leq \Biggl((n-1)+2\sqrt{n-1}\Biggr)|e||C|^2 + \Biggl(\frac{4n-11}{2\sqrt{n-1}}+\frac{|2n-9|(1+\sqrt{n-1})}{2(n-1)}\Biggr)|e||U|^2 \notag \\
        &\quad + \Biggl((n-1)+\frac{4n-11}{2\sqrt{n-1}}\Biggr)|e||\Hc|^2 + \Biggl(2\sqrt{n-1}+\frac{|2n-9|(1+\sqrt{n-1})}{2(n-1)}\Biggr)|e||\Sigma|^2. \label{uniq-est-7}
\end{align}
Comparing the coefficients of $|C|^2,|U|^2,|\Hc|^2$ and $|\Sigma|^2$ in \eqref{uniq-est-1} and \eqref{uniq-est-7}, we find that, provided
\begin{equation*}
    \sup_{(t,x)\in \Gamma_{\Icv}}|e(t,x)| \leq \frac{\rho_1}{6n^3},
\end{equation*}
it follows that
\begin{equation*}
    -\frac{\rho_1}{t^{\ep_2}} v^{\tr}B^0 v + \frac{1}{t^{\ep_2}}\frac{x_\Lambda}{|x|}e_D^\Lambda v^{\tr}B^D v < 0.
\end{equation*}
Since $\rho_1>0$, it is clear from \eqref{Bvbar-fix} that
\begin{equation*}
    \bar{\Bv}|_{\Gamma_{\Icv}} \leq 0.
\end{equation*}
This indicates that the hypersurface $\Gamma_{\Icv}$ is weakly spacelike with respect to the system \eqref{Fuchsian-3} to \eqref{Fuchsian-4}, in accordance with the definition provided preceding \cite[Thm.~4.5]{Lax:2006}. Therefore, by \cite[Thm.~4.5]{Lax:2006}, the solution constructed in \eqref{Wv-local} is the unique solution satisfying the initial condition $\Wv|_{\{t_0\}\times\mathbb{B}_{\rho_0}}=\Wv_0$.

So far, we have not yet specified the choice of $\delta_0$. Noticing that $e=(e_P^\Sigma)$ is a component of the vector $\Wv$, we apply the Sobolev inequality \eqref{Sobolev} and the energy estimate \eqref{Wv-energy} to obtain
\begin{equation*}
    \sup_{(t,x)\in\Omega_{\Icv}}|e(t,x)| \lesssim \sup_{0<t\leq t_0} \norm{\Wv(t)}_{L^\infty(\mathbb{B}_{\rho(t)})} \lesssim \sup_{0<t\leq t_0} \norm{\Wv(t)}_{H^{k_1}(\mathbb{B}_{\rho(t)})} \lesssim \norm{\Wv_0}_{H^{k_1}(\mathbb{B}_{\rho_0})}.
\end{equation*}
Consequently, there exists a constant $C_1>0$ such that $\sup_{(t,x)\in\Omega_{\Icv}}|e(t,x)| \leq C_1\norm{\Wv_0}_{H^{k_1}(\mathbb{B}_{\rho_0})}$. We may therefore choose
\begin{equation*}
    \delta_0 = \min\Biggl\{ \frac{\tilde{\delta}_0}{C_0}, \frac{\rho_1}{6C_1 n^3} \Biggr\},
\end{equation*}
and this completes the proof.

\end{proof}

\section{Local Existence in Lagrangian Coordinates}
As noted at the beginning of Section~\ref{Fuch-form}, we need to make sure that the global-in-time solutions in Proposition~\ref{Fuch-global} and ~\ref{Fuch-local} are indeed the solutions to the Einstein-scalar field equations. A crucial step in this verification is to ensure the propagation of the constraint equations \eqref{A-cnstr-2} to \eqref{H-cnstr-2}. While the authors of \cite{BOZ:2025} derive a strongly hyperbolic propagation system for these constraints directly in Section 4.2, we adopt an alternative approach in this section. Specifically, we establish the local-in-time existence of solutions to the Einstein-scalar field equations on both $M_{0,t_0}$ and on $\Omega_{\Icv}$ using the framework developed in \cite[\S 5]{BeyerOliynyk:2024b}, which is formulated relative to \textit{Lagrangian coordinates}. The uniqueness property of the solutions then allows us to achieve our goal.

The precise construction of the \textit{Lagrangian coordinates} is discussed in \cite[\S 5.3, \S 5.4]{BeyerOliynyk:2024b}.\footnote{The complete setup is provided in these two sections, and we will not reiterate it subsequently.} In \cite[\S 5]{BeyerOliynyk:2024b}, the authors use symbols with a "hat" (e.g. $\hat{\cdot}$\,) to denote the scalars and tensors in the conformal picture relative to the standard coordinates; symbols without any document to denote the geometric pull-back by the Lagrangian map $l$; and symbols with an "underline" (e.g. $\underline{\cdot}$\,) to denote the pull-back of scalars by $l$. For instance, the $g$ throughout this section corresponds to the conformal metric $\gt$ in the present article. We adhere to these conventions in this section and will provide clarification should any potential for confusion arise.

\subsection{Gauges} \label{gauges-veri}
Before presenting the Einstein-scalar field equations in Lagrangian coordinates, we must clarify that the gauges we use are the same as that in \cite[\S 5]{BeyerOliynyk:2024b}. This alignment is essential because, although different gauge choices do not imply physical difference, the specific variables used in this article, such as $\Aft_{AB}^\Omega$ and $\Bft_{AB}$, are not gauge-invariant and in the system, there are constraints beyond the standard Hamiltonian and momentum constraints.

In particular, the Lagrangian coordinates in \cite{BeyerOliynyk:2024b} is constructed via the vector field
\begin{equation*}
    \hat{\chi}^\mu = \frac{1}{|\vh|_{\gh}^2}\vh^\mu,
\end{equation*}
where $\vh^\mu=\hat{\partial}_{\mu}\hat{\tau}$ and we assume that $\chih^\mu$ is timelike with respect to $\gh$. According to \cite[Eq.~(5.109),(5.110)]{BeyerOliynyk:2024b}, the geometric pull-back of $\hat{\chi}^\mu$ is given by
\begin{equation} \label{chimu-def}
    \chi^\mu = \frac{\nabla^\mu \tau}{|\nabla\tau|_g^2} = \delta_0^\mu.
\end{equation}
As shown \cite[Eq.~(5.112)]{BeyerOliynyk:2024b}, this implies
\begin{equation*}
    \tau = t - t_0 + \mathring{\tau}.
\end{equation*}
Furthermore, \cite[Prop.~5.8]{BeyerOliynyk:2024b} establishes that after a short evolution, $\tau=t$, which is consistent with our choice of the time function specified in \eqref{time-fix}. This consistency further ensures that the evolution equation for the lapse, derived from the identity $\Box_{\gt}t=0$, remains unchanged. To see that the shift vanishes in the Lagrangian coordinates, we observe that \eqref{chimu-def} implies that the vector $\del{t}$ is parallel to the vector $\nabla^\mu \tau$, which is orthogonal to the $\tau=constant$ hypersurface. Since $\tau=t$, it follows that $\del{t}$ is orthogonal to the $t=constant$ hypersurface, so then the shift is zero. In summary, we have verified that the choices of gauges in \cite[\S 5]{BeyerOliynyk:2024b} are the same as those in the present article.

\subsection{Initial value problem in Lagrangian coordinates} \label{Lag-IVP}
The Einstein-scalar field equations in Lagrangian coordinates are expressed as follows, see also \cite[Eq.~(5.48)]{BeyerOliynyk:2024b} to \cite[Eq.~(5.55)]{BeyerOliynyk:2024b},
\begin{align}
    \Bhu^{\lambda\beta\alpha} \Jcch_\alpha^\gamma \del{\gamma}\hhu_{\beta\mu\nu} &= \Jc_0^\lambda \Bigr(\Qhu_{\mu\nu}+\frac{2}{\tau}\bigl( \whu_{(\mu \nu)} - \Gammahu^\gamma_{\mu\nu} \vhu_\gamma\bigl) \Bigr), \label{Lag-1} \\
    \Bhu^{\lambda\beta\alpha} \Jcch_\alpha^\gamma \del{\gamma}\whu_{\beta\mu} &= - \Jc^\lambda_0\ghu^{\alpha\sigma}\ghu^{\beta\delta}\hhu_{\mu\sigma\delta}\whu_{\alpha\beta}, \label{Lag-2} \\
    \Bhu^{\lambda\beta\alpha} \Jcch_\alpha^\gamma \del{\gamma}\zhu_{\beta} &=0, \label{Lag-3} \\
    \del{0}\ghu_{\mu\nu} &= \Jc_0^\alpha \hhu_{\alpha\mu\nu}, \label{Lag-4} \\
    \del{0}\vhu_{\mu} &= \Jc_0^\alpha \whu_{\alpha\mu}, \label{Lag-5} \\
    \del{0}\tau &= \Jc^\alpha_0\zhu_{\alpha}, \label{Lag-6} \\
    \del{0}\Jc^\mu_\nu &=  \Jc^\lambda_\nu \Jsc^\mu_\lambda, \label{Lag-7} \\
    \del{0}l^\mu &= \chihu^\mu, \label{Lag-8}
\end{align}
where we view $\{\hhu_{\beta\mu\nu}, \whu_{\beta\mu}, \zhu_{\beta}, \ghu_{\mu\nu}, \vhu_{\mu}, \tau, \Jc^\mu_\nu, l^\mu\}$ as the first order variables and
\begin{align*}
    \Bh^{\lambda\beta\alpha} &= -\chih^\lambda \gh^{\beta\alpha} - \chih^\beta \gh^{\lambda\alpha} + \gh^{\lambda\beta} \chih^\alpha, \\
    (\Jcch_\nu^\mu) &:= (\Jc_\nu^\mu)^{-1},
\end{align*}
with $\Jc_\nu^\mu=\del{\nu}l^\mu$. The initial value problem for the system \eqref{Lag-1} to \eqref{Lag-8} is given by
\begin{align}
    P^0 \del{t}Z + P^\Gamma \del{\Gamma}Z &= Y \hspace{0.6cm} \text{in $M_{t_1,t_0}=(t_1,t_0]\times \Tbb^{n-1}$}, \label{Lag-IVP-1} \\
    Z &= Z_0 \hspace{0.5cm} \text{in $\Sigma_{t_0}=\{t_0\}\times \Tbb^{n-1}$}, \label{Lag-IVP-2}
\end{align}
where
\begin{equation} \label{Z-def}
    Z = (\hhu_{\beta\mu\nu}, \whu_{\beta\mu}, \zhu_{\beta}, \ghu_{\mu\nu}, \vhu_{\mu}, \tau, \Jc^\mu_\nu, l^\mu)^{\tr},
\end{equation}
\begin{equation} \label{P0-def}
    P^0 = \diag(\Bhu^{\lambda\beta\alpha} \Jcch_\alpha^0, \Bhu^{\lambda\beta\alpha} \Jcch_\alpha^0, \Bhu^{\lambda\beta\alpha} \Jcch_\alpha^0, \delta_\alpha^\mu \delta_\beta^\nu, \delta_\alpha^\mu, 1, \delta_\alpha^\nu \delta_\mu^\beta, \delta_\alpha^\mu),
\end{equation}
\begin{equation} \label{PGamma-def}
    P^\Gamma = \diag(\Bhu^{\lambda\beta\alpha} \Jcch_\alpha^\Gamma, \Bhu^{\lambda\beta\alpha} \Jcch_\alpha^\Gamma, \Bhu^{\lambda\beta\alpha} \Jcch_\alpha^\Gamma, 0, 0, 0, 0, 0),
\end{equation}
and
\begin{equation*}
    Y = \biggl( \Jc_0^\lambda \Bigr(\Qhu_{\mu\nu}+\frac{2}{\tau}\bigl( \whu_{(\mu \nu)} - \Gammahu^\gamma_{\mu\nu} \vhu_\gamma\bigl) \Bigr), - \Jc^\lambda_0\ghu^{\alpha\sigma}\ghu^{\beta\delta}\hhu_{\mu\sigma\delta}\whu_{\alpha\beta}, 0, \Jc_0^\alpha \hhu_{\alpha\mu\nu}, \Jc_0^\alpha \whu_{\alpha\mu}, \Jc^\alpha_0\zhu_{\alpha}, \Jc^\lambda_\nu \Jsc^\mu_\lambda, \chihu^\mu \biggr)^{\tr},
\end{equation*}
with $t_1\in(0,t_0]$ and $Z_0$ is chosen to be same as that in \cite[Eq.~(5.56)]{BeyerOliynyk:2024b} to \cite[Eq.~(5.71)]{BeyerOliynyk:2024b}. Following the argument in \cite[Eq.~(5.88)]{BeyerOliynyk:2024b} to \cite[Eq.~(5.91)]{BeyerOliynyk:2024b}, the system \eqref{Lag-IVP-1} to \eqref{Lag-IVP-2} is a symmetric hyperbolic one. Also we use the Lagrangian map $l$ to pull back the equation \cite[Eq.~(5.20)]{BeyerOliynyk:2024b}, and the corresponding initial value problem for the wave gauge vector field $X^\mu$ is given by
\begin{align}
    \nabla_{\mu}\nabla^{\mu}X^{\nu} + 2\nabla_\mu \ln(\tau)\nabla^{[\mu}X^{\nu]} + \nabla_\mu \ln(\tau)\nabla^{\nu}X^\mu + \bigl(R^{\nu}{}_{\mu} - \nabla^{\nu}\nabla_\mu \ln(\tau) - 2\nabla^\nu\ln(\tau)\nabla_\mu\ln(\tau)\bigr) X^\mu = 0, \label{Lag-IVP-3}
\end{align}
in $M_{t_1,t_0}=(t_1,t_0]\times \Tbb^{n-1}$ with the initial conditions
\begin{align}
    X^\mu &= 0 \hspace{0.4cm} \text{in $\Sigma_{t_0}=\{t_0\}\times \Tbb^{n-1}$}, \label{Lag-IVP-4} \\
    \del{t}X^\mu &= 0 \hspace{0.4cm} \text{in $\Sigma_{t_0}=\{t_0\}\times \Tbb^{n-1}$}. \label{Lag-IVP-5}
\end{align}
In the following section, we also deal with IVPs that are identical to the systems \eqref{Lag-IVP-1} to \eqref{Lag-IVP-2} and \eqref{Lag-IVP-3} to \eqref{Lag-IVP-5}, but with the domains changed to $\Omega_{\Icv}$ and the initial hypersurface changed to $\{t_0\}\times\mathbb{B}_{\rho_0}$.

\subsection{Local-in-time existence on $M_{t_1,t_0}$ and $\Omega_{\Icv}$.}
In this section, we prove the local well-posedness of the system \eqref{EEb.1} to \eqref{EEb.6} and the propagation of the constraints system \eqref{A-cnstr-2} to \eqref{H-cnstr-2}, both on $M_{t_1,t_0}$ and $\Omega_{\Icv}$ with $\Icv=(t_0,t_1,\rho_0,\rho_1,\ep_2)$. To be clear, the initial value problems are
\begin{align}
    \del{t}\Wt + \Bt^\Omega \del{\Omega}\Wt &= \Tt \hspace{0.7cm} \text{in $M_{t_1,t_0}=(t_1,t_0]\times \Tbb^{n-1}$}, \label{te-local-1} \\
    \Wt &= \Wt_0 \hspace{0.4cm} \text{in $\Sigma_{t_0}=\{t_0\}\times \Tbb^{n-1}$}, \label{te-local-2}
\end{align}
and
\begin{align}
    \del{t}\Wt + \Bt^\Omega \del{\Omega}\Wt &= \Tt \hspace{0.7cm} \text{in $\Omega_{\Icv}$}, \label{te-local-3} \\
    \Wt &= \Wt_0 \hspace{0.4cm} \text{in $\{t_0\}\times \mathbb{B}_{\rho_0}$}, \label{te-local-4}
\end{align}
where $\Wt=(\et_A^\Omega,\Ct_A{}^C{}_B,\Hct,\Sigmat_{AB},\alphat,\Ut_A)$, $\Bt^\Omega$ is the coefficient matrix of the spatial derivative, $\Tt$ consists of the source terms and $\Wt_0$ is the initial value to be determined later. Since in the end we do not need to analyze these IVPs directly, we choose not to present the expressions for these matrices explicitly.

From standard theories of symmetric hyperbolic systems and wave equations, for instance \cite[Thm.~10.1]{BenzoniSerre:2007} and \cite[Prop.~9.12]{Ringstrom-CauchyGR:2009}, we have the local-in-time existence, uniqueness and continuation principle of the IVPs \eqref{Lag-IVP-1} to \eqref{Lag-IVP-2} and \eqref{Lag-IVP-3} to \eqref{Lag-IVP-5}. In Section~\ref{gauges-veri}, we have verified that the choices of gauges in the Lagrangian coordinates align with the frame formalism in the present article, thus using the initial data prescribed in \cite[Eq.~(5.56) to (5.71)]{BeyerOliynyk:2024b} which satisfies the gravitational and wave gauge constraints, we can construct an orthonormal frame $\{\mathring{\et}_a\}$ on the initial hypersurface $\{t_0\}\times\Tbb^{n-1}$. Then apply the Fermi-Walker transport in \eqref{Fermi-tran} to propagate $\{\mathring{\et}_a\}$ and get an orthonormal frame $\{\et_a\}$ on $M_{t_1,t_0}$. Following the procedure detailed in Section~\ref{conn-comm-coef}, we construct the fields $\{\et_A^\Omega, \Ct_A{}^C{}_B, \Hct, \Sigmat_{AB}, \alphat, \Ut_A\}$ on $M_{t_1,t_0}$, such that the constraints \eqref{A-cnstr-2} to \eqref{H-cnstr-2} also vanish on $M_{t_1,t_0}$. Collectively, these arguments establish the following proposition:
\begin{prop} \label{global-local}
    Suppose $t_0>0$, $k_0\in\Zbb_{>\frac{n-1}{2}+1}$ and $\Wt_0\in H^{k_0}(\Tbb^{n-1})$, then there exists a $t_1\in(0,t_0)$ and a unique classical solution $\Wt\in C^1(M_{t_1,t_0})$ to the IVP \eqref{te-local-1} to \eqref{te-local-2} such that $\Wt\in\cap_{i=0}^{k_0}C^i\bigl((t_1,t_0],H^{k_0-i}(\Tbb^{n-1})\bigr)$ and satisfies the continuation principle, i.e. if $\sup_{t_1<t<t_0}\norm{\Wt}_{W^{1,\infty}(\Tbb^{n-1})}<\infty$, then there exists a time $\tilde{t}_1\in[0,t_1)$, such that $\Wt$ can be uniquely continued to a classical solution on $M_{\tilde{t}_1,t_0}$.
    
    Moreover, if the initial data $\Wt_0$ is chosen so that the constraints \eqref{A-cnstr-2} to \eqref{H-cnstr-2} vanish on $\Sigma_{t_0}$, then the constraints vanish on $M_{t_1,t_0}$.
\end{prop}

To establish a result for the IVP \eqref{te-local-3} to \eqref{te-local-4} analogous to that for \eqref{te-local-1} to \eqref{te-local-2}, it is necessary to demonstrate that the boundary $\Omega_{\Icv}$, i.e. $\Gamma_{\Icv}$, is weakly spacelike with respect to the equations \eqref{Lag-IVP-1} and \eqref{Lag-IVP-3}. We mention here that the principal parts of \eqref{Lag-IVP-1} and \eqref{Lag-IVP-3} are essentially the same since both of them originate from the wave operator $\Box_g$. Thus, if $\Gamma_{\Icv}$ is weakly spacelike with respect to either one of the equations, then it will also be weakly spacelike with respect to the other one. Let $n_{\mu}$ denote the normal vector to $\Gamma_{\Icv}$ and define the matrix
\begin{equation*} \label{Pb-def}
    \Pb := n_0 P^0 + n_{\Gamma}P^{\Gamma}.
\end{equation*}
Using equations \eqref{normal}, \eqref{P0-def} and \eqref{PGamma-def}, we obtain
\begin{align}
    \Pb &= \diag\Biggl(-\frac{\rho_1}{t^{\ep_2}}\Bhu^{\lambda\beta\alpha} \Jcch_\alpha^0, -\frac{\rho_1}{t^{\ep_2}}\Bhu^{\lambda\beta\alpha} \Jcch_\alpha^0, -\frac{\rho_1}{t^{\ep_2}}\Bhu^{\lambda\beta\alpha} \Jcch_\alpha^0, -\frac{\rho_1}{t^{\ep_2}}\delta_\alpha^\mu \delta_\beta^\nu, -\frac{\rho_1}{t^{\ep_2}}\delta_\alpha^\mu, -\frac{\rho_1}{t^{\ep_2}}, -\frac{\rho_1}{t^{\ep_2}}\delta_\alpha^\nu \delta_\mu^\beta, -\frac{\rho_1}{t^{\ep_2}}\delta_\alpha^\mu\Biggr) \notag \\
        &\quad + \diag\Biggl(\frac{x_{\Gamma}}{|x|}\Bhu^{\lambda\beta\alpha} \Jcch_\alpha^\Gamma, \frac{x_{\Gamma}}{|x|}\Bhu^{\lambda\beta\alpha} \Jcch_\alpha^\Gamma, \frac{x_{\Gamma}}{|x|}\Bhu^{\lambda\beta\alpha} \Jcch_\alpha^\Gamma, 0, 0, 0, 0, 0\Biggr), \label{Pb-fix}
\end{align}
where $x_{\Gamma}=x^\Lambda\delta_{\Lambda\Gamma}$. From \cite[Eq.~(5.56)]{BeyerOliynyk:2024b}, the following identity holds:
\begin{equation} \label{Bhu-1}
    \Bhu^{\lambda\beta\alpha}\Jcch^0_\alpha = -\frac{\chihu^\sigma \Jcch^0_\sigma}{|\chihu|_{\ghu}^2}\chihu^\beta\chihu^\lambda - \chihu^\lambda \pi^{\beta\sigma}\Jcch^0_\sigma - \chihu^\beta \pi^{\lambda\sigma}\Jcch^0_\sigma + \pi^{\lambda \beta}\chihu^\sigma \Jcch^0_\sigma,
\end{equation}
and similarly,
\begin{equation} \label{Bhu-2}
    \Bhu^{\lambda\beta\alpha}\Jcch^\Gamma_\alpha = -\frac{\chihu^\sigma \Jcch^\Gamma_\sigma}{|\chihu|_{\ghu}^2}\chihu^\beta\chihu^\lambda - \chihu^\lambda \pi^{\beta\sigma}\Jcch^\Gamma_\sigma - \chihu^\beta \pi^{\lambda\sigma}\Jcch^\Gamma_\sigma + \pi^{\lambda \beta}\chihu^\sigma \Jcch^\Gamma_\sigma,
\end{equation}
where
\begin{equation*} \label{pi-def} 
    \pi^\gamma_\lambda = \delta^\gamma_\lambda - \frac{\chihu^\gamma \chihu_\lambda}{|\chihu|_{\ghu}^2}
\end{equation*}
defines the projection onto the $\ghu$-orthogonal subspace to $\chihu^\mu$ and we set $\pi^{\gamma\lambda}=\ghu^{\lambda\sigma}\pi_\sigma^\gamma$. For simplicity, define the quantity
\begin{equation} \label{J-al-def}
    J_\alpha := \frac{\rho_1}{t^{\ep_2}}\Jcch_\alpha^0 - \frac{x_{\Gamma}}{|x|}\Jcch_\alpha^\Gamma.
\end{equation}
Using \eqref{Bhu-1}, \eqref{Bhu-2} and \eqref{J-al-def}, it follows that
\begin{equation*}
    -\frac{\rho_1}{t^{\ep_2}}\Bhu^{\lambda\beta\alpha} \Jcch_\alpha^0 + \frac{x_{\Gamma}}{|x|}\Bhu^{\lambda\beta\alpha} \Jcch_\alpha^\Gamma = -\Biggl( -\frac{\chihu^\sigma J_\sigma}{|\chihu|_{\ghu}^2}\chihu^\beta\chihu^\lambda - \chihu^\lambda \pi^{\beta\sigma}J_\sigma - \chihu^\beta \pi^{\lambda\sigma}J_\sigma + \pi^{\lambda \beta}\chihu^\sigma J_\sigma \Biggr).
\end{equation*}
Since we assume that $\chih^\mu$ is timelike, i.e. $|\chihu|_{\ghu}^2<0$, it can be verified that as long as $J_\alpha$ satisfies the conditions
\begin{equation} \label{J-al-con1}
    \chihu^\sigma J_\sigma > 0 \AND \ghu^{\alpha\beta}J_\alpha J_\beta < 0,
\end{equation}
then
\begin{equation*}
    -\frac{\rho_1}{t^{\ep_2}}\Bhu^{\lambda\beta\alpha} \Jcch_\alpha^0 + \frac{x_{\Gamma}}{|x|}\Bhu^{\lambda\beta\alpha} \Jcch_\alpha^\Gamma < 0.
\end{equation*}
To consider the conditions in \eqref{J-al-con1}, first we note from \cite[Eq.~(5.36)]{BeyerOliynyk:2024b} that $\chihu^\mu=\del{0}l^\mu=\Jc_0^\mu$. Thus,
\begin{equation*}
    \chihu^\sigma J_\sigma = \Jc_0^\sigma \Biggl(\frac{\rho_1}{t^{\ep_2}}\Jcch_\sigma^0 - \frac{x_{\Gamma}}{|x|}\Jcch_\sigma^\Gamma \Biggr) = \frac{n\rho_1}{t^{\ep_2}} > 0,
\end{equation*}
where we have used that facts that $\Jc_\mu^\sigma\Jcch_\sigma^\nu=\delta_\mu^\nu$ and $\rho_1>0$. Furthermore,
\begin{equation} \label{J-al-con2}
    \ghu^{\alpha\beta}J_\alpha J_\beta = \ghu^{\alpha\beta} \Biggl(\frac{\rho_1}{t^{\ep_2}}\Jcch_\alpha^0 - \frac{x_{\Gamma}}{|x|}\Jcch_\alpha^\Gamma \Biggr) \Biggl(\frac{\rho_1}{t^{\ep_2}}\Jcch_\beta^0 - \frac{x_{\Gamma}}{|x|}\Jcch_\beta^\Gamma \Biggr) = \frac{\rho_1^2}{t^{2\ep_2}}g^{00} + \frac{x_\Gamma x_\Sigma}{|x|^2}g^{\Gamma\Sigma},
\end{equation}
which follows from the identity $\ghu^{\alpha\beta}\Jcch_\alpha^\mu \Jcch_\beta^\nu=g^{\mu\nu}$ and the fact that the shift is zero, which means that $g^{0\Gamma}$ is also zero. The metric components can be expressed using the variables from the frame formalism in Section~\ref{Tetrad-formalism}. Since the shift is zero, we have
\begin{equation} \label{g00}
    g^{00} = -\frac{1}{\alphat^2} \AND g^{\Gamma\Sigma} = \delta^{AB}\et_A^\Gamma \et_B^\Sigma,
\end{equation}
where $\et_A$ denotes the spatially orthogonal frame constructed in Section~\ref{Tetrad-formalism}. Substituting \eqref{g00} into the second condition in \eqref{J-al-con2} and applying the Cauchy-Schwarz inequality along with the inequality \eqref{CS-ineq}, we derive
\begin{align}
    \ghu^{\alpha\beta}J_\alpha J_\beta &= -\frac{\rho_1^2}{t^{2\ep_2}\alphat^2} + \frac{x_\Gamma x_\Sigma}{|x|^2} \delta^{AB}\et_A^\Gamma \et_B^\Sigma \leq -\frac{\rho_1^2}{t^{2\ep_2}\alphat^2} + \Biggl( \sum_{\Gamma,\Sigma=1}^{n-1} \bigl(\delta^{AB}\et_A^\Gamma \et_B^\Sigma\bigr)^2 \Biggr)^{\frac{1}{2}} \notag \\
        &\leq -\frac{\rho_1^2}{t^{2\ep_2}\alphat^2} + \Biggl( \sum_{\Gamma,\Sigma=1}^{n-1} (n-1)\delta^{AB}\bigl(\et_A^\Gamma \et_B^\Sigma\bigr)^2 \Biggr)^{\frac{1}{2}} \notag \\
        &\leq -\frac{\rho_1^2}{t^{2\ep_2}\alphat^2} + \Biggl( (n-1)\sum_{\Gamma,\Sigma=1}^{n-1} \sum_{A,B=1}^{n-1} \bigl(\et_A^\Gamma\bigr)^2 \bigl(\et_B^\Sigma\bigr)^2 \Biggr)^{\frac{1}{2}} \notag \\
        &= -\frac{\rho_1^2}{t^{2\ep_2}\alphat^2} + \sqrt{n-1}|\et|^2. \label{J-al-est}
\end{align}
Therefore, if the condition
\begin{equation} \label{Pb-con}
    \sup_{(t,x)\in\Omega_{\Icv}}t^{\ep_2}|\alphat(t,x)||\et(t,x)| < \frac{\rho_1}{(n-1)^{\frac{1}{4}}},
\end{equation}
is satisfied, then it follows from \eqref{Pb-fix}, \eqref{J-al-con1} and \eqref{J-al-est} that $\Gamma_{\Icv}$ is weakly spacelike with respect to both \eqref{Lag-IVP-1} and \eqref{Lag-IVP-3}. Observe that \eqref{Pb-con} is an open condition, and if we choose initial data that satisfies this inequality, then by continuity, it will continue to hold for a sufficiently short time. In summary, we have proved the following proposition:

\begin{prop} \label{local-local}
    Suppose $t_0>0$, $k_0\in\Zbb_{>\frac{n-1}{2}+1}$, $0<\rho_0<L$, $\ep_2$ satisfies \eqref{ep-con-1}, $\Wt_0\in H^{k_0}(\mathbb{B}_{\rho_0})$ and its components $\alphat(x)$ and $\et_A^\Omega(x)$\footnote{Here, the dependence of $\alphat$ and $\et_A^\Omega$ is only $x$, indicating they are initial data. Later in this theorem, we use $\alphat(t,x)$ and $\et_A^\Omega(t,x)$ to denote the solutions on $M_{t_1,t_0}$ and $\Omega_{\Icv}$.}, satisfy that $\inf_{x\in\mathbb{B}_{\rho_0}}\alphat(x)>0$, $\inf_{x\in\mathbb{B}_{\rho_0}}\det(\et_A^\Omega(x))>0$ and the inequality
    \begin{equation*}
        \sup_{x\in\mathbb{B}_{\rho_0}} t_0^{\ep_2}|\alphat(x)||\et(x)| < \frac{\rho_1}{(n-1)^{\frac{1}{4}}},
    \end{equation*}
    then there exists a $t_1\in(0,t_0)$ and a unique classical solution $\Wt\in C^1(\Omega_{\Icv})$ to the IVP \eqref{te-local-3} to \eqref{te-local-4}, with $\Icv=(t_0,t_1,\rho_0,\rho_1,\ep_2)$, that satisfies $\inf_{x\in\Omega_{\Icv}}\alphat(t,x)>0$ and $\inf_{x\in\Omega_{\Icv}}\det(\et_A^\Omega(t,x))>0$ and \, $\sup_{t_1<t\leq t_0}\norm{\del{t}^i\Wt(t)}_{H^{k_0-i}(\mathbb{B}_{\rho(t)})}<\infty$, with $i\in\{0, 1,...,k_0\}$ and $\rho(t)=\frac{\rho_1(t^{1-\ep_2}-t_0^{1-\ep_2})}{1-\ep_2}+\rho_0$. And $\Wt$ satisfies the continuation principle, i.e. if $\sup_{t_1<t<t_0}\norm{\Wt}_{W^{1,\infty}(\mathbb{B}_{\rho(t)})}<\infty$ and $\sup_{x\in\Omega_{\Icv}} t^{\ep_2}|\alphat(t,x)||\et(t,x)| < \frac{\rho_1}{(n-1)^{\frac{1}{4}}}$, then there exists a time $\tilde{t}_1\in[0,t_1)$, such that $\Wt$ can be uniquely continued to a classical solution on $\Omega_{\tilde{\Icv}}$ with $\tilde{\Icv}=(t_0,\tilde{t}_1,\rho_0,\rho_1,\ep_2)$.
    
    Moreover, if the initial data $\Wt_0$ is chosen so that the constraints \eqref{A-cnstr-2} to \eqref{H-cnstr-2} vanish on $\{t_0\}\times\mathbb{B}_{\rho_0}$, then the constraints vanish on $\Omega_{\Icv}$.
\end{prop}

\section{Localized Big Bang Stability}
We are now ready to establish the past global-in-time stability of the sub-critical Kasner metrics on a truncated cone domain, along with a description of the perturbed solution near $t=0$, e.g. asymptotic behavior and curvature invariants blow up.

\subsection{Physical quantities in terms of the rescaled variables} \label{physical-quant}
Prior to stating the main theorem, we express the physical curvature invariants $\Rb$ and $\Rb_{ab}\Rb^{ab}$, the second fundamental form associated with both the physical metric $\gb$ and the conformal metric $\gt$ on the $t=constant$ hypersurfaces, as well as certain components of the Weyl tensor in terms of the rescaled variables defined in \eqref{Hc-def} to \eqref{C-def}. We stress here that all the tensors are expressed relative to the conformal orthonormal frame $\{\et_a\}$.

We begin by calculating the physical Ricci tensor $\Rb_{ab}$. From \eqref{ESF.3} and \eqref{Kasner-metric}, we have
\begin{equation} \label{Rbab}
    \Rb_{ab} = \Tb_{ab} - \frac{1}{n-2}\gb^{cd}\Tb_{cd}\gb_{ab} = \Tb_{ab} - \frac{1}{n-2}\eta^{cd}\Tb_{cd}\eta_{ab},
\end{equation}
where we have used the fact that $\gb^{ab}=t^{-\frac{2}{n-2}}\eta^{ab}$. For use below, we list the components of $\Rb_{ab}$:
\begin{align}
    \Rb_{00} &= \frac{n-3}{n-2}\Tb_{00} + \frac{1}{n-2}\Tb_A{}^A, \label{Rb00} \\
    \Rb_{0A} &= \Tb_{0A}, \label{Rb0A} \\
    \Rb_{AB} &= \Tb_{AB} + \frac{1}{n-2}\Tb_{00}\delta_{AB} - \frac{1}{n-2}\Tb_C{}^C \delta_{AB}, \label{RbAB}
\end{align}
where $\Tb_A{}^A=\delta^{AB}\Tb_{AB}$. Using \eqref{Rbab}, the physical scalar curvature $\Rb$ is given by
\begin{equation*}
    \Rb = \gb^{ab}\Rb_{ab} = \gb^{ab}(\Tb_{ab} - \frac{1}{n-2}\eta^{cd}\Tb_{cd}\eta_{ab}) = -\frac{2}{n-2}t^{-\frac{2}{n-2}}(-\Tb_{00}+\Tb_A{}^A).
\end{equation*}
From \eqref{Tb-Upsilon-cmpts} and a short calculation, it follows that
\begin{equation} \label{physi-scalar}
    \Rb = -\frac{2}{n-2}t^{-\frac{2}{n-2}}(-\Tb_{00}+\Tb_A{}^A) = -\frac{n-1}{(n-2)\alpha^2 t^{\frac{2n-2}{n-2}-2\ep_1}}.
\end{equation}
Similarly, using \eqref{Tb-Upsilon-cmpts} and \eqref{Rbab} to \eqref{RbAB}, the invariant $\Rb_{ab}\Rb^{ab}$ is given by
\begin{align} \label{physi-Ricci}
    \Rb_{ab}\Rb^{ab} &= \gb^{ac}\gb^{bd}\Rb_{ab}\Rb_{cd} = t^{-\frac{4}{n-2}}\eta^{ac}\eta^{bd}\Rb_{ab}\Rb_{cd} \notag \\
        &= t^{-\frac{4}{n-2}}(\Rb_{00}\Rb_{00} - 2\delta^{AB}\Rb_{0A}\Rb_{0B} + \delta^{AC}\delta^{BD}\Rb_{AB}\Rb_{CD}) \notag \\
        &= \frac{(n-1)^2}{(n-2)^2 \alpha^4 t^{\frac{4n-4}{n-2}-4\ep_1}}.
\end{align}

Now let $\{\thetat^a\}$ be the dual basis of $\{\et_a\}$, that is we have $\thetat^a(\et_b)=\delta_b^a$. Then the conformal metric $\gt$ can be expressed as
\begin{equation*}
    \gt = -\thetat^0\otimes\thetat^0 + \gttt_{AB}\thetat^A\otimes\thetat^B,
\end{equation*}
where $\gttt_{AB}=\delta_{AB}$ and since we employ the zero shift gauge, $\gttt_{AB}$ correspond to the spatial metric. Observe that the lapse and the unit normal vector of $\gt$ on the $t=constant$ hypersurface are $\alphat$ and $\et_0$, respectively. The second fundamental form $\Kttt_{AB}$ associated with $\gt$ is
\begin{align} \label{conf-sec-fund}
    \Kttt_{AB} &= \frac{1}{2\alphat}(\Ltt_{\alphat\et_0}\gttt)(\et_A, \et_B) \notag \\
        &= \frac{1}{2\alphat}\Bigl( \alphat\et_0\bigl(\gttt(\et_A,\et_B)\bigr) - \gttt\bigl([\alphat\et_0,\et_A], \et_B\bigr) - \gttt\bigl(\et_A, [\alphat\et_0,\et_B]\bigr) \Bigr) \notag \\
        &= \frac{1}{2\alpha t^{1-\ep_1}}(r_{AB} + 2\Hc\delta_{AB} + 2\Sigma_{AB}),
\end{align}
where $\Ltt$ denotes the Lie derivative, and we have used the decompositions \eqref{comm-decomp-0} and \eqref{comm-decomp-A}, the constraint \eqref{D-cnstr-2} and the definition for the rescaled variables in \eqref{Hc-def} to \eqref{alpha-def}.

Relative to the dual frame $\{\thetat^a\}$ defined above, the physical metric $\gb$, by virtue of \eqref{Kasner-metric}, takes the form
\begin{equation} \label{physi-metric}
    \gb = -t^{\frac{2}{n-2}}\thetat^0\otimes\thetat^0 + \gttb_{AB}\thetat^A\otimes\thetat^B,
\end{equation}
with spatial metric $\gttb_{AB}=t^{\frac{2}{n-2}}\delta_{AB}$. Notice that the lapse and the unit normal vector of $\gb$ on the $t=constant$ hypersurface are $\alphat t^{\frac{1}{n-2}}$ and $t^{-\frac{1}{n-2}}\et_0$, respectively. Hence, by a similar calculation, the second fundamental form $\Kttb_{AB}$ of $\gb$ is
\begin{align} \label{physi-sec-funda}
    \Kttb_{AB} &= \frac{1}{2\alphat t^{\frac{1}{n-2}}}\bigl(\Ltt_{\alphat\et_0}\gttb\bigr)(\et_A, \et_B) \notag \\
        &= \frac{1}{2\alphat t^{\frac{1}{n-2}}}\Bigl( \alphat\et_0\bigl(\gttb(\et_A,\et_B)\bigr) - \gttb\bigl([\alphat\et_0,\et_A], \et_B\bigr) - \gttb\bigl(\et_A, [\alphat\et_0,\et_B]\bigr) \Bigr) \notag \\
        &= \frac{1}{2\alpha t^{\frac{n-3}{n-2}-\ep_1}}\Biggl(\frac{2}{n-2}\delta_{AB} + r_{AB} + 2\Hc\delta_{AB} + 2\Sigma_{AB}\Biggr).
\end{align}

We now turn our attention to the invariants associated with the Weyl tensor. According to \cite[Eq.~(3.2.28)]{Wald:1994}, the Weyl tensor of $\gt$ is defined as
\begin{align}
    \Ct_{abcd} &= \Rt_{abcd} - \frac{1}{n-2}(\Rt_{ac}\gt_{bd} + \Rt_{bd}\gt_{ac} - \Rt_{ad}\gt_{bc} - \Rt_{bc}\gt_{ad}) \notag \\
        &\quad + \frac{\Rt}{2(n-1)(n-2)}(\gt_{ac}\gt_{bd} + \gt_{bd}\gt_{ac} - \gt_{ad}\gt_{bc} - \gt_{bc}\gt_{ad}). \label{Weyl-def}
\end{align}
For 4-dimensional spacetimes, from \cite[Eq.~(1.34)]{Wainwright-Ellis:1997}, the electric part of the Weyl tensor is defined as $\Et_{ab}=\Ct_{acbd}u^cu^d$, where $u^c$ denotes a preferred timelike vector field. In \cite[Thm. 6.1(e)]{BOZ:2025}, the authors have proved that the blow up of the Weyl tensor invariant actually originates from the electric part contracting with itself. For higher dimensional spacetimes, we aim to analyze the behavior of the analogous components. Selecting $u$ as $\et_0$, which implies $(\et_0)^c=\delta_0^c$, we focus on the components interested in are
\begin{equation*}
    \Ct_{acbd}(\et_0)^c(\et_0)^d = \Ct_{a0b0} \AND \Cb_{acbd}(\et_0)^c(\et_0)^d = \Cb_{a0b0}.
\end{equation*}
Utilizing the fact that the Weyl tensor shares the same symmetry properties of the Riemann tensor, i.e. anti-symmetry in its first and last two indices, along with definition \eqref{Weyl-def}, we find that the nonzero components of $\Ct_{a0b0}$ are
\begin{align} \label{EtAB-def}
    \Ct_{A0B0} = \Rt_{A0B0} - \frac{1}{n-2}(\Rt_{00}\delta_{AB}-\Rt_{AB}) - \frac{\Rt}{(n-1)(n-2)}\delta_{AB},
\end{align}
in which
\begin{align}
    \Rt_{AB} &= \eta^{cd}\Rt_{AcBd} = -\Rt_{A0B0} + \delta^{CD}\Rt_{ACBD}, \label{RtAB} \\
    \Rt_{00} &= \eta^{ab}\Rt_{0a0b} = -\Rt_{0000} + \delta^{AB}\Rt_{A0B0}, \label{Rt00} \\
    \Rt &= \eta^{ab}\Rt_{ab} = \Rt_{0000} - 2\delta^{AB}\Rt_{A0B0} + \delta^{AB}\delta^{CD}\Rt_{ACBD}. \label{Rt}
\end{align}
Substituting expressions \eqref{RtAB} to \eqref{Rt} into \eqref{EtAB-def} yields the expression for $\Ct_{A0B0}$ in terms of the Riemann tensors:
\begin{align}
    \Ct_{A0B0} &= \frac{n-3}{n-2}\Rt_{A0B0} - \frac{n-3}{(n-1)(n-2)}\delta_{AB}\delta^{CD}\Rt_{C0D0} \notag \\
        &\quad + \frac{1}{n-2}\delta^{CD}\Rt_{ACBD} - \frac{1}{(n-1)(n-2)}\delta_{AB}\delta^{EF}\delta^{CD}\Rt_{ECFD}, \label{EtAB-1}
\end{align}
where we have used the fact that $\Rt_{0000}=0$. A direct calculation using \eqref{connect-form}, \eqref{Riemann}, \eqref{EEb.3} and \eqref{EEb.4} provides the following identities:
\begin{align*}
    \Rt_{A0B0} &= \frac{1}{n-1}\Sigmat_{CD}\Sigmat^{CD}\delta_{AB} - \frac{1}{\alphat t}\Hct\delta_{AB} + (n-3)\Hct\Sigmat_{AB} - \Sigmat_A{}^C\Sigmat_{BC} + \frac{1}{\alphat t}\Sigmat_{AB} \notag \\
        &\quad + \et*\del{}\Ct + \Ct*\Ct, \\
    \Rt_{ABCD} &= \Hct^2\delta_{AC}\delta_{BD} + \Hct\Sigmat_{AC}\delta_{BD} + \Hct\Sigmat_{BD}\delta_{AC} + \Sigmat_{AC}\Sigmat_{BD} \notag \\
        &\quad - \Hct^2\delta_{AD}\delta_{BC} - \Hct\Sigmat_{AD}\delta_{BC} - \Hct\Sigmat_{BC}\delta_{AD} - \Sigmat_{AD}\Sigmat_{BC} + \et*\del{}\Ct + \Ct*\Ct.
\end{align*}
In these expressions, $\del{}$ denotes the spatial derivatives of the tensors and see Section~\ref{Index} for the definition of the $*$ operator. Substituting these results into \eqref{EtAB-1} leads to
\begin{align*}
    \Ct_{A0B0} &= \frac{n-3}{(n-2)\alphat t}\Sigmat_{AB} + (n-3)\Hct\Sigmat_{AB} + \frac{1}{n-1}\Sigmat_{CD}\Sigmat^{CD}\delta_{AB} - \Sigmat_{AC}\Sigmat_B{}^C + \et*\del{}\Ct + \Ct*\Ct.
\end{align*}
According to \cite[Thm.~7.30]{JohnMLee-Rie:2018}, under conformal transformation \eqref{conf-metricA}, the Weyl tensor of the physical metric $\gb_{ab}$ is related to that of the conformal metric $\gt_{ab}$ via
\begin{equation} \label{conf-Weyl}
    \Cb_{abcd} = e^{2\Phi}\Ct_{abcd}.
\end{equation}
Defining $\Cb^{abcd}=\gb^{am}\gb^{bn}\gb^{cp}\gb^{dq}\Cb_{mnpq}$, and using \eqref{Kasner-rels-B}, \eqref{conf-metricA}, \eqref{f-metric}, \eqref{Phi-fix}, \eqref{D-cnstr-2}, the definitions for the rescaled variables \eqref{Hc-def} to \eqref{C-def} and the relation in \eqref{conf-Weyl}, a brief calculation shows that within the Weyl tensor invariant $\Cb_{acbd}\Cb^{acbd}$, the part $\Cb_{A0B0}\Cb^{A0B0}$ is given by
\begin{align} \label{Weyl-invar-def}
    \Cb_{A0B0}\Cb^{A0B0} = e^{-4\Phi}\delta^{AC}\delta^{BD}\Ct_{A0B0}\Ct_{C0D0} = \frac{1}{\alpha^4 t^{4+\frac{4}{n-2}-4\ep_1}}\delta^{AC}\delta^{BD}(t^2\alphat^2\Ct_{A0B0}\cdot t^2\alphat^2\Ct_{C0D0})
\end{align}
with
\begin{align} \label{CtA0B0}
    t^2\alphat^2\Ct_{A0B0} &= \frac{n-3}{2(n-2)}r_{AB} + \frac{r_0}{2(n-2)}\delta_{AB} + \frac{r_0}{4}r_{AB} - \frac{1}{4}r_{AC}r_B^C \notag \\
        &\quad + \frac{n-3}{n-2}\Sigma_{AB} + (n-3)\Hc\Sigma_{AB} + \Biggl(\frac{n-3}{2}r_{AB} - \frac{(n-3)r_0}{2(n-1)}\delta_{AB}\Biggr)\Hc \notag \\
        &\quad + \frac{(n-3)r_0}{2(n-1)}\Sigma_{AB} + \frac{1}{n-1}\Sigma_{CD}\Sigma^{CD}\delta_{AB} + \frac{1}{n-1}r_{CD}\Sigma^{CD}\delta_{AB} \notag \\
        &\quad - \Sigma_{AC}\Sigma_B{}^C - \Sigma_{C(A}r_{B)}^C + \frac{r_0}{n-1}\Sigma_{AB} + t^{1-\ep_2}e*\del{}C + (U+C)*C.
\end{align}

\begin{rem}
    The Riemann tensor can be decomposed into three parts, see \eqref{Weyl-def}, consisting of the Weyl curvature, the Ricci curvature and the scalar curvature respectively. By \eqref{Rb00} to \eqref{RbAB}, the Ricci and scalar curvature can be determined via the stress-energy tensor. In contrast, we may interpret the Weyl curvature as a purely gravitational element. We will see in theorem~\ref{main-thm} below that the big bang singularity not only arises from the matter field, but also from gravity itself. This is consistent with the findings in \cite[Thm.~6.1(e), 6.5(f)]{BOZ:2025}.
\end{rem}

\subsection{Past stability of the Kasner-scalar field metrics on $\Omega_{\Icv}$.}
\begin{thm} \label{main-thm}
    Suppose $T_0>0$, $k_0\in\Zbb_{>\frac{n-1}{2}+1}$, the conformal Kasner exponents $r_0$ and $r_A$, where $A=1,...,n-1$, satisfy the Kasner relations \eqref{Kasner-rels-B} and the subcritical condition \eqref{sub-cond-2}, $\ep_1,\ep_2\in\Rbb$ satisfy \eqref{ep-con-1}, $\nu\in\Rbb_{>0}$ satisfies \eqref{nu-con}, $k_1\in\Zbb_{\geq0}$ is chosen large enough such that for $\Bsc$ in \eqref{Bcvt-def}, $\frac{1}{2}(\Bsc^{\tr}+\Bsc)$ is positive definite and $0<\rho_0<L$, then there exists a $\delta>0$ such that for every $t_0\in(0,T_0]$ and $\delta_0\in(0,\delta]$, and \footnote{Here, the fields depending only on $x$ are initial data, while those depending on both $t$ and $x$ or without specifying dependence are solutions.}
    \begin{equation} \label{m-initial}
        W_0 = \bigl(e_A^\Omega(x), \alpha(x), C_{ABC}(x), U_A(x), \Hc(x), \Sigma_{AB}(x)\bigr)^{\tr} \in H^k(\mathbb{B}_{\rho_0}) \quad \text{with} \quad k=k_0+k_1,
    \end{equation}
    satisfying $\inf_{x\in\mathbb{B}_{\rho_0}}\alpha(x)>0$, $\inf_{x\in\mathbb{B}_{\rho_0}}\det(e_A^\Omega(x))>0$, the constraints \eqref{A-cnstr-3} to \eqref{H-cnstr-3} on $\{t_0\}\times\mathbb{B}_{\rho_0}$ and $\norm{W_0}_{H^k(\mathbb{B}_{\rho_0})}<\delta_0$, there exists a classical solution
    \begin{equation*}
        W = (e_A^\Omega, \alpha, C_{ABC}, U_A, \Hc, \Sigma_{AB})^{\tr} \in C^1(\Omega_{\Icv})
    \end{equation*}
    with $\Icv=(t_0,0,\rho_0,\rho_1,\ep_2)$ and $\rho_1$ satisfying \eqref{rho1-range}, that solves the evolution equations \eqref{EEc.1} to \eqref{EEc.6} on $\Omega_{\Icv}$, satisfies the initial condition $W|_{t=t_0}=W_0$ and the constraints \eqref{A-cnstr-3} to \eqref{H-cnstr-3} on $\Omega_{\Icv}$, $\alpha>0$ and $\det(e_A^\Omega)>0$. Furthermore,
    \begin{equation*}
        W(t) \in H^k(\mathbb{B}_{\rho(t)}) \AND \del{t}W(t) \in H^{k-1}(\mathbb{B}_{\rho(t)})
    \end{equation*}
    with $\rho(t)=\frac{\rho_1(t^{1-\ep_2}-t_0^{1-\ep_2})}{1-\ep_2}+\rho_0$ for all $t\in(0,t_0]$. Moreover, the following hold:
    \begin{enumerate}[(a)]
        \item The solution $W$ is uniformly bounded in the sense that
        \begin{equation*}
            \norm{W(t)}_{H^k(\mathbb{B}_{\rho(t)})} \lesssim \delta_0
        \end{equation*}
        for all $t\in(0,t_0]$ and there exist constants $\zeta>0$, $\delta_1>0$ and functions $\alphah, \Hchat, \Sigmah_{AB}\in H^{k-1}(\rhot_0)$ satisfying
        \begin{gather*}
            0 < \inf_{x\in\mathbb{B}_{\rhot_0}}\alphah(x) \leq \sup_{x\in\mathbb{B}_{\rhot_0}}\alphah(x) \lesssim 1, \quad \Sigmah_{AB} = \Sigmah_{BA}, \quad \delta^{AB}\Sigmah_{AB} = 0, \\
            \max\Bigl\{ \norm{\Hchat}_{L^\infty(\mathbb{B}_{\rhot_0})}, \norm{\Sigmah_{AB}}_{L^\infty(\mathbb{B}_{\rhot_0})} \Bigr\} \leq \delta_1 \AND \max\Bigl\{ \norm{\Hchat}_{H^{k-1}(\mathbb{B}_{\rhot_0})}, \norm{\Sigmah_{AB}}_{{H^{k-1}(\mathbb{B}_{\rhot_0})}} \Bigr\} \lesssim \delta_0,
        \end{gather*}
        with $\tilde{\rho}_0=\rho_0-\frac{\rho_1 t_0^{1-\ep_2}}{1-\ep_2}$, such that the components of $W$ satisfy the following estimates \footnote{See Section~\ref{order-nota} for the definition of order notation.}
        \begin{gather*}
            e_A^\Omega = \Ord_{H^{k-1}(\mathbb{B}_{\rhot_0})}(t^\zeta), \quad \alpha = t^{\ep_1+\frac{r_0}{2}+(n-1)\Hchat}\alphah \bigl(1+\Ord_{H^{k-1}(\mathbb{B}_{\rhot_0})}(t^{\zeta})\bigr), \quad C_{ABC} = \Ord_{H^{k-1}(\mathbb{B}_{\rhot_0})}(t^\zeta), \notag \\
            U_A = \Ord_{H^{k-1}(\mathbb{B}_{\rhot_0})}(t^\zeta), \quad \Hc = \Hchat + \Ord_{H^{k-1}(\mathbb{B}_{\rhot_0})}(t^\zeta), \quad \Sigma_{AB} = \Sigmah_{AB} + \Ord_{H^{k-1}(\mathbb{B}_{\rhot_0})}(t^\zeta),
        \end{gather*}
        for all $t\in(0,t_0]$. Furthermore, the explicit and implicit constants are independent of the choices of $\delta_0\in(0,\delta]$ and $t_0\in(0,T_0]$.

        \item The pair
        \begin{equation*}
            \Biggl\{ \gb = t^{\frac{2}{n-2}}\bigl(-\alphat^2dt\otimes dt + \gt_{\Sigma\Omega}dx^\Sigma \otimes dx^\Omega\bigr), \, \varphi = \sqrt{\frac{n-1}{2(n-2)}}\ln(t) \Biggr\},
        \end{equation*}
        where
        \begin{equation*}
            \alphat = t^{-\ep_1}\alpha, \quad (\gt_{\Sigma\Omega}) = (\delta^{AB}\et_A^\Sigma \et_B^\Omega)^{-1} \AND \et_A^\Omega = t^{-\ep_2}\alphat^{-1}e_A^\Omega
        \end{equation*}
        determines a classical solution to the Einstein-scalar field equations \eqref{ESF.1} to \eqref{ESF.2} on $\Omega_{\Icv}$. In addition, the trace of the second fundamental form $\Kttb_{AB}$ with respect to the physical metric $\gb_{ab}$ on the $t=constant$ hypersurface satisfies the following estimate 
        \begin{equation*}
            \Kttb_A{}^A(t,x) = \gttb^{AB}\Kttb_{AB}(t,x) \gtrsim \frac{1}{t^{\frac{n-1}{n-2}+\frac{r_0}{2}-(n-1)\delta_1}} 
        \end{equation*}
        for $(t,x)\in\Omega_{\Icv}$ when $t\in(0,t_0]$ is small enough. Here, when expressed relative to the frame $\{\et_0,\et_A\}$, $\gttb_{AB}=t^{\frac{2}{n-2}}\delta_{AB}$ denotes the spatial metric of $\gb_{ab}$. For an appropriate choice of $\delta_1$, we ensure that $\frac{n-1}{n-2}+\frac{r_0}{2}-(n-1)\delta_1>0$, which implies $\Kttb_A{}^A$ blows up uniformly as $t\searrow0$ and the hypersurface $t=0$ is a crushing singularity in the sense of the definition in \cite{Eardley:1979}.

        \item The physical solution $\{\gb,\varphi\}$ of the Einstein-scalar field equations on $\Omega_{\Icv}$ is past $C^2$-inextendible at $t=0$ and past timelike geodesically incomplete. The physical scalar curvature $\Rb$ and the invariant $\Rb_{ab}\Rb^{ab}$ satisfy the following estimates:
        \begin{equation*}
            \Rb(t,x) \lesssim -\frac{1}{t^{\frac{2(n-1)}{n-2}+r_0-2(n-1)\delta_1}} \AND \Rb_{ab}(t,x)\Rb^{ab}(t,x) \gtrsim -\frac{1}{t^{\frac{2(n-1)}{n-2}+r_0-2(n-1)\delta_1}},
        \end{equation*}
        for $(t,x)\in\Omega_{\Icv}$ when $t\in(0,t_0]$ is small enough. With the same choice of $\delta_1$ as in (b), we have $\frac{n-1}{n-2}+\frac{r_0}{2}-(n-1)\delta_1>0$. Consequently, $\Rb$ and $\Rb_{ab}\Rb^{ab}$ blow up uniformly as $t\searrow0$, thereby establishing the $C^2$-inextendibility of the physical metric $\gb$.

        \item The pair
        \begin{equation*}
            \Bigl\{ \gt = -\alphat^2dt\otimes dt + \gt_{\Sigma\Omega}dx^\Sigma \otimes dx^\Omega, \,\, \tau = t \Bigr\}
        \end{equation*}
        determines a classical solution to the conformal Einstein-scalar field equations \eqref{conf-ESF.1} to \eqref{conf-ESF.2} on $\Omega_{\Icv}$. The solution $\{\gt,\tau\}$ satisfies the AVTD property in the sense of \cite[\S 1.3]{BeyerOliynyk:2024b}. Moreover, the second fundamental form $\Kttt_{AB}$ of the conformal metric $\gt$ on the $t=constant$ hypersurfaces, when expressed relative to the frame $\{\et_0,\et_A\}$, satisfies that
        \begin{equation*}
            2t\alphat\Kttt_{AB} = \kf_{AB} + \Ord_{H^{k-1}(\mathbb{B}_{\rhot_0})}(t^\zeta),
        \end{equation*}
        where $\kf_{AB}=r_{AB}+2\Hchat\delta_{AB}+2\Sigmah_{AB}$. Furthermore, $\kf_{AB}$ satisfies
        \begin{equation*}
            \kf_A{}^A \geq 0 \AND (\kf_A{}^A)^2 - \kf_A{}^B\kf_B{}^A + 4\kf_A{}^A = 0,
        \end{equation*}
        implying that the spacetime $(\Omega_{\Icv},\gt,\tau)$ is asymptotically pointwise Kasner according to \cite[Def. 1.1]{BeyerOliynyk:2024a}.

        \item The geometric invariant, resulting from the Weyl tensor of the physical metric $\gb$, $\Cb_{A0B0}\Cb^{A0B0}$ satisfies that
        \begin{equation*}
            \inf_{x\in\mathbb{B}_{\rhot_0}}\Cb_{A0B0}\Cb^{A0B0} \gtrsim \frac{1}{t^{\frac{4(n-1)}{n-2}+2r_0-4\delta_1}}
        \end{equation*}
        provided there is at least one nonzero conformal Kasner exponent, when $t\in(0,t_0]\cap(0,1]$ and for the same choice of $\delta_1$ as in (b). Consequently, the invariant $\Cb_{A0B0}\Cb^{A0B0}$, as a component of the invariant $\Cb_{abcd}\Cb^{abcd}$, blows up uniformly as $t\searrow0$ if we are not perturbing around the FLRW metric defined in \eqref{FLRW-def}.
    \end{enumerate}
\end{thm}

\subsection{Proof of Theorem~\ref{main-thm}}
We make a few statements before we start the proof so that the big picture is rather clear. First, in the following proof, $k_0$ denotes the order of regularity of the Fuchsian system, while $k_1$ represents the number of spatial derivatives that we need. Concerning the smallness condition on the initial data, at the beginning $\delta>0$ is arbitrary, but but will be successively restricted to smaller values throughout the argument. And the core idea of the proof is to manipulate Proposition~\ref{Fuch-local} and ~\ref{local-local} to get the global-in-time also local-in-space solution to the Einstein-scalar field equations. Subsequently, using the corresponding energy and decay estimates to derive estimates on the solution, which ultimately enables an analysis of the behavior of the physical quantities near $t=0$.

\bigskip

\noindent \underline{\textit{Global-in-time existence.}}
Given initial data $W_0$ as specified in \eqref{m-initial} that satisfies the constraint equations \eqref{A-cnstr-3} to \eqref{H-cnstr-3} on $\{t_0\}\times\mathbb{B}_{\rho_0}$ and $\norm{W_0}_{H^k(\mathbb{B}_{\rho_0})}<\delta_0$, we construct a new initial dataset by taking spatial derivatives of $W_0$. This yields
\begin{equation} \label{m-Wv0}
    \Wv_0 = \bigl( (W_{0\bc})_{|\bc|=0}, (W_{0\bc})_{|\bc|=1},..., (W_{0\bc})_{|\bc|=k_1-1},(\Wt_{0\bc})_{|\bc|=k_1} \bigr)^{\tr},
\end{equation}
where $W_{0\bc}=(t_0)^{|\bc|\nu}\partial^{\bc}W_0$ for $0\leq|\bc|<k_1$, $\Wt_{0\bc}=V^{-1}(t_0)^{|\bc|\nu}\partial^{\bc}W_0$ for $|\bc|=k_1$ and $V^{-1}$ is defined in \eqref{Vinvs-def}. It is clear then that $\Wv_0$ can be used as the initial data of the Fuchsian GIVP \eqref{Fuchsian-3} to \eqref{Fuchsian-4}, and there exists a constant $C_0>0$ that is independent of the choice of $\delta_0$, such that
\begin{equation*}
    \norm{\Wv_0}_{H^{k_0}(\mathbb{B}_{\rho_0})} < C_0\delta_0.
\end{equation*}
By Proposition~\ref{Fuch-local} and the smallness condition \eqref{smallness-2}, there exists a $\delta>0$ that is small enough, such that for every $\delta_0\in(0,\delta]$, there exists a unique solution $\Wv$ such that
\begin{equation*} \label{m-Wv-regu}
    \Wv(t) \in H^{k_0}(\mathbb{B}_{\rho(t)}) \AND \del{t}\Wv(t) \in H^{k_0-1}(\mathbb{B}_{\rho(t)}),
\end{equation*}
where $\rho(t)=\frac{\rho_1(t^{1-\ep_2}-t_0^{1-\ep_2})}{1-\ep_2}+\rho_0$, for all $t\in(0,t_0]$. Moreover, the limit $\lim_{t\searrow0}\Pv^{\perp}\Wv(t)$, denoted by $\Pv^{\perp}\Wv(0)$, exists in $H^{k_0-1}(\mathbb{B}_{\tilde{\rho}_0})$, with $\tilde{\rho}_0=\rho_0-\frac{\rho_1 t_0^{1-\ep_2}}{1-\ep_2}$, and there exists a $\zeta>0$ such that the solution $\Wv$ satisfies the energy estimate
\begin{equation} \label{m-energy}
    \norm{\Wv(t)}_{H^{k_0}(\mathbb{B}_{\rho(t)})}^2 + \int_t^{t_0} \frac{1}{s}\norm{\Pv\Wv(s)}_{H^{k_0}(\mathbb{B}_{\rho(s)})}^2 \, ds \lesssim \norm{\Wv_0}_{H^{k_0}(\mathbb{B}_{\rho_0})}^2
\end{equation}
and the decay estimate
\begin{equation} \label{m-decay}
    \norm{\Pv\Wv(t)}_{H^{k_0-1}(\mathbb{B}_{\rhot_0})} + \norm{\Pv^{\perp}\Wv(t)-\Pv^{\perp}\Wv(0)}_{H^{k_0-1}(\mathbb{B}_{\rhot_0})} \lesssim t^{\zeta}
\end{equation}
for all $t\in(0,t_0]$. Moreover, the very first component of $\Wv$, i.e. $e=(e_A^\Omega)$, satisfies the following estimate
\begin{equation} \label{m-e-bound}
    \sup_{(t,x)\in \Omega_{\Icv}} |e(t,x)| \leq \frac{\rho_1}{6n^3}.
\end{equation}
From \eqref{e-def} and \eqref{m-e-bound}, we see that in $W_0$, the component $e_A^\Omega(x)$ satisfies that
\begin{equation*}
    \sup_{x\in\mathbb{B}_{\rho_0}} |e(x)| = \sup_{x\in\mathbb{B}_{\rho_0}} t_0^{\ep_2}|\alphat(x)||\et(x)| < \frac{\rho_1}{(n-1)^{\frac{1}{4}}},
\end{equation*}
along with $\inf_{x\in\mathbb{B}_{\rho_0}}\alpha(x)>0$, $\inf_{x\in\mathbb{B}_{\rho_0}}\det(e_A^\Omega(x))>0$. Hence, by Proposition~\ref{local-local}, there exists a $t_1\in(0,t_0]$ and a unique classical solution
\begin{equation} \label{m-solution}
    W = (e_A^\Omega, \alpha, C_{ABC}, U_A, \Hc, \Sigma_{AB})^{\tr} \in C^1(\Omega_{\Icv'})
\end{equation}
with $\Icv'=(t_0,t_1,\rho_0,\rho_1,\ep_2)$, that satisfies the initial condition $W|_{t=t_0}=W_0$. This solution solves both the evolution equations \eqref{EEc.1} to \eqref{EEc.6} and \eqref{EEd.1} to \eqref{EEd.6}\footnote{The reason that we stress here the system \eqref{EEd.1} to \eqref{EEd.6} is that we construct the highest order term $(\Wt_{\bc})_{|\bc|=k_1}$ in $\Wv$ from it. Also, when the constraints vanish, the systems \eqref{EEc.1} to \eqref{EEc.6} and \eqref{EEd.1} to \eqref{EEd.6} are equivalent.}, guarantees that the constraints \eqref{A-cnstr-3} to \eqref{H-cnstr-3} vanish on $\Omega_{\Icv'}$ and exhibits the regularity $\sup_{t_1<t\leq t_0}\norm{\del{t}^i W(t)}_{H^{k-i}(\mathbb{B}_{\rho(t)})}<\infty$ for $i\in\{0, 1,...,k\}$. Also the solution $W$ can be extended to a larger time interval as long as
\begin{equation*}
    \sup_{t_1<t<t_0}\norm{W}_{W^{1,\infty}(\mathbb{B}_{\rho(t)})} < \infty \AND \sup_{x\in\Omega_{\Icv'}} |e(t,x)| < \frac{\rho_1}{(n-1)^{\frac{1}{4}}}.
\end{equation*}

Now, taking the spatial derivatives of $W$, we define
\begin{equation*}
    \check{\Wv} = \bigl( (W_{\bc})_{|\bc|=0}, (W_{\bc})_{|\bc|=1},..., (W_{\bc})_{|\bc|=k_1-1}, (\Wt_{\bc})_{|\bc|=k_1} \bigr)^{tr},
\end{equation*}
where $W_{\bc}=t^{|\bc|\nu}\partial^{\bc}W$ for $0\leq|\bc|<k_1$, $\Wt_{\bc}=V^{-1}t^{|\bc|\nu}\partial^{\bc}W$ for $|\bc|=k_1$ and $V^{-1}$ is defined in \eqref{Vinvs-def}. Repeating the procedure in Section~\ref{Fuch-form}, we find that $\check{\Wv}$ satisfies the Fuchsian system \eqref{Fuchsian-3}. Moreover, it is not difficult to verify that $\check{\Wv}|_{t=t_0}=\Wv_0$, see \eqref{m-Wv0}, then by uniqueness of the solution, we have that
\begin{equation*}
    \check{\Wv} = \Wv \quad \text{in} \quad \Omega_{\Icv'}.
\end{equation*}
This fact together with \eqref{m-energy}, \eqref{m-e-bound} and the Sobolev inequality \eqref{Sobolev}, further imply
\begin{equation*}
    \sup_{t_1<t<t_0}\norm{W}_{W^{1,\infty}(\mathbb{B}_{\rho(t)})} \leq C_1\delta_0 \AND \sup_{x\in\Omega_{\Icv'}} |e(t,x)| \leq \frac{\rho_1}{6n^3},
\end{equation*}
for some constant $C_1>0$. Consequently, the solution $W$ can be continued past $t_1$ and in fact, exists on the whole interval $(0,t_0]$, and so does $\check{\Wv}$. By uniqueness again, we have that $\check{\Wv}=\Wv$ in $\Omega_{\Icv}$, which in turn, by the energy estimate \eqref{m-energy}, implies that
\begin{equation} \label{m-W-bound}
    \norm{W(t)}_{H^k(\mathbb{B}_{\rho(t)})} \leq C_2 \delta_0
\end{equation}
for some constant $C_2>0$ and all $t\in(0,t_0]$.

\bigskip

\noindent \underline{\textit{Estimates for components in $W$.}}
Based on the decay estimate in \eqref{m-decay} and the uniform bound in \eqref{m-W-bound}, we obtain the following order estimates:
\begin{gather} \label{m-fields-decay}
    e_A^\Omega = \Ord_{H^{k-1}(\mathbb{B}_{\rhot_0})}(t^\zeta), \quad \alpha = \Ord_{H^{k-1}(\mathbb{B}_{\rhot_0})}(t^\zeta), \quad C_{ABC} = \Ord_{H^{k-1}(\mathbb{B}_{\rhot_0})}(t^\zeta), \notag \\
    U_A = \Ord_{H^{k-1}(\mathbb{B}_{\rhot_0})}(t^\zeta), \quad \Hc = \Hchat + \Ord_{H^{k-1}(\mathbb{B}_{\rhot_0})}(t^\zeta), \quad \Sigma_{AB} = \Sigmah_{AB} + \Ord_{H^{k-1}(\mathbb{B}_{\rhot_0})}(t^\zeta),
\end{gather}
where $\Hchat,\Sigmah_{AB}\in H^{k-1}(\mathbb{B}_{\rhot_0})$ and are bounded above by
\begin{equation} \label{m-Hchat-bd0}
    \max\Bigl\{ \norm{\Hchat}_{H^{k-1}(\mathbb{B}_{\rhot_0})}, \norm{\Sigmah_{AB}}_{{H^{k-1}(\mathbb{B}_{\rhot_0})}} \Bigr\} \lesssim \delta_0,
\end{equation}
for all $t\in(0,t_0]$. Furthermore, from \eqref{m-W-bound} and the Sobolev inequality \eqref{Sobolev}, it follows that
\begin{align*}
    \max\Bigl\{ \norm{\Hchat}_{L^\infty(\mathbb{B}_{\rhot_0})}, \norm{\Sigmah_{AB}}_{L^\infty(\mathbb{B}_{\rhot_0})}, \sup_{0<t\leq t_0}\norm{\Hc(t)}_{L^\infty(\mathbb{B}_{\rho(t)})}, \sup_{0<t\leq t_0}\norm{\Sigma_{AB}(t)}_{L^\infty(\mathbb{B}_{\rho(t)})} \Bigr\} \leq C_3 \delta_0
\end{align*}
for some constant $C_3>0$. Restricting $\delta_0\leq\frac{\delta_1}{C_3}$, for some $\delta_1\in(0,C_3\delta]$, yields estimate
\begin{align} \label{m-Hchat-bd}
    \max\Bigl\{ \norm{\Hchat}_{L^\infty(\mathbb{B}_{\rhot_0})}, \norm{\Sigmah_{AB}}_{L^\infty(\mathbb{B}_{\rhot_0})}, \sup_{0<t\leq t_0}\norm{\Hc(t)}_{L^\infty(\mathbb{B}_{\rho(t)})}, \sup_{0<t\leq t_0}\norm{\Sigma_{AB}(t)}_{L^\infty(\mathbb{B}_{\rho(t)})} \Bigr\} \leq \delta_1
\end{align}

To analyze the behavior of the physical curvature invariants, a more refined estimate for the lapse $\alpha$ is required. First we observe from \eqref{EEc.5} that the evolution equations for $\alpha$ is actually a homogeneous ODE. By uniqueness of the solution and the assumption that $\inf_{x\in\mathbb{B}_{\rho_0}}\alpha(x)>0$, we conclude that $\alpha>0$ on $\Omega_{\Icv}$. A direct calculation using \eqref{EEc.5} shows that
\begin{equation*}
    \del{t}\ln(\alpha t^{-\ep_1-\frac{r_0}{2}}) = \frac{n-1}{t}\Hc.
\end{equation*}
Applying \eqref{m-fields-decay}, we further compute
\begin{equation*}
    \del{t}\ln(\alpha t^{-\ep_1-\frac{r_0}{2}}\cdot t^{-(n-1)\Hchat}) = \frac{n-1}{t}(\Hc-\Hchat) = \Ord_{H^{k-1}(\mathbb{B}_{\rhot_0})}(t^{\zeta-1}).
\end{equation*}
Hence, by Lemma~\ref{lem:asymptotic}, there exists a $\check{\alpha}\in H^{k-1}(\mathbb{B}_{\rhot_0})$ such that
\begin{equation} \label{m-alpha-1}
    \ln(\alpha t^{-\ep_1-\frac{r_0}{2}-(n-1)\Hchat}) = \check{\alpha} + \Ord_{H^{k-1}(\mathbb{B}_{\rhot_0})}(t^{\zeta}).
\end{equation}
Setting $\alphah=\exp{(\check{\alpha})}$, it follows from the Sobolev and Moser inequality (equations \eqref{Sobolev} and \eqref{Moser}) that
\begin{equation} \label{m-alphah-bd}
    \alphah \in H^{k-1}(\mathbb{B}_{\rhot_0}) \AND 0 < \inf_{x\in\mathbb{B}_{\rhot_0}}\alphah(x) \leq \sup_{x\in\mathbb{B}_{\rhot_0}}\alphah(x) \lesssim 1.
\end{equation}
Now taking the exponential on both sides of \eqref{m-alpha-1} and applying the Sobolev and Moser inequality again (where we set $f(u)=e^u-1$), we derive
\begin{equation} \label{m-alpha-2}
    \alpha = t^{\ep_1+\frac{r_0}{2}+(n-1)\Hchat}\alphah \betah,
\end{equation}
where
\begin{equation*}
    \betah = 1+\Ord_{H^{k-1}(\mathbb{B}_{\rhot_0})}(t^{\zeta})
\end{equation*}
and
\begin{equation} \label{m-betah-bd}
    1 \lesssim \betah(t,x) \lesssim 1, \quad \forall (t,x)\in\Omega_{\Icv},
\end{equation}
with the implicit constants independent of the choices of $\delta_0$ and $t_0$.

For use below, we impose a further restriction on $\delta_1$ via
\begin{equation*}
    \delta_1 < \min\Biggl\{ C_3\delta,\, \frac{1}{n-2}+\frac{r_0}{2(n-1)} \Biggr\}.
\end{equation*}
By \eqref{m-Hchat-bd}, the new bound for $\delta_1$ guarantees that
\begin{align} \label{m-Hchat-bd-1}
    \inf_{x\in\mathbb{B}_{\rhot_0}}\frac{n-1}{n-2}+\frac{r_0}{2}+(n-1)\Hchat(x) &> \frac{n-1}{n-2}+\frac{r_0}{2}-(n-1)\delta_1 > 0 \notag \\
    \inf_{(t,x)\in\Omega_{\Icv}}\frac{n-1}{n-2}+\frac{r_0}{2}+(n-1)\Hc(t,x) &> \frac{n-1}{n-2}+\frac{r_0}{2}-(n-1)\delta_1 > 0.
\end{align}

\bigskip

\noindent \underline{\textit{Crushing singularity.}}
Observe from \eqref{EEc.1} that
\begin{equation*}
    \del{t}e_A^\Omega = \frac{1}{t}\Biggl( \biggl(\ep_2+\frac{r_0}{2}\biggr)\delta_A^B - \frac{1}{2}r_A^B + (n-2)\Hc\delta_A^B - \Sigma_A{}^B \Biggr) e_B^\Omega.
\end{equation*}
Using the fact that $\delta^{AB}\Sigma_{AB}=0$ and the Lemma A.2 in \cite{BeyerOliynyk:2024b}, we obtain
\begin{equation*}
    \det\bigl(e_A^\Omega(t,x)\bigr) = e^{-\int_t^{t_0}\frac{1}{s}((n-1)(\ep_2+\frac{r_0}{2})-\frac{r_0}{2}+(n-1)(n-2)\Hc(s))ds}\det\bigl(e_A^\Omega(x)\bigr)
\end{equation*}
for all $t\in(0,t_0]$. Then due to the assumption that $\inf_{x\in\mathbb{B}_{\rho_0}}\det(e_A^\Omega(x))>0$, it is clear that $\det(e_A^\Omega)>0$ on $\Omega_{\Icv}$. Moreover, from the previous argument, we know that $\alpha>0$ on $\Omega_{\Icv}$. Consequently, the solution \eqref{m-solution} defines a classical solution to the Einstein-scalar field equations \eqref{ESF.1} to \eqref{ESF.2} as
\begin{equation*}
    \Biggl\{ \gb = t^{\frac{2}{n-2}}\bigl(-\alphat^2dt\otimes dt + \gt_{\Sigma\Omega}dx^\Sigma \otimes dx^\Omega\bigr), \,\, \varphi = \sqrt{\frac{n-1}{2(n-2)}}\ln(t) \Biggr\},
\end{equation*}
where
\begin{equation*}
    \alphat = t^{-\ep_1}\alpha, \quad (\gt_{\Sigma\Omega}) = (\delta^{AB}\et_A^\Sigma \et_B^\Omega)^{-1} \AND \et_A^\Omega = t^{-\ep_2}\alphat^{-1}e_A^\Omega.
\end{equation*}
Now we see from \eqref{Kasner-rels-B}, \eqref{physi-metric} and \eqref{physi-sec-funda} that the trace of the physical second fundamental form $\Kttb_{AB}$ on the $t=constant$ hypersurfaces can be calculated by
\begin{equation*}
    \Kttb_A{}^A = \gttb^{AB}\Kttb_{AB} = \frac{1}{\alpha t^{\frac{n-1}{n-2}-\ep_1}}\Biggl(\frac{n-1}{n-2} + \frac{r_0}{2} + (n-1)\Hc\Biggr).
\end{equation*}
Substituting \eqref{m-alpha-2} into this expression yields that
\begin{equation*}
    \Kttb_A{}^A = \frac{1}{t^{\frac{n-1}{n-2}+\frac{r_0}{2}+(n-1)\Hchat}\alphah\betah}  \Biggl(\frac{n-1}{n-2} + \frac{r_0}{2} + (n-1)\Hc\Biggr).
\end{equation*}
Thus, from \eqref{m-alphah-bd}, \eqref{m-betah-bd} and \eqref{m-Hchat-bd-1}, for $(t,x)\in\Omega_{\Icv}$ with $t\in(0,t_0]\cap(0,1]$, we have
\begin{equation} \label{m-meancur-bd}
    \Kttb_A{}^A(t,x) \gtrsim \frac{1}{t^{\frac{n-1}{n-2}+\frac{r_0}{2}-(n-1)\delta_1}},
\end{equation}
which implies that $\Kttb_A{}^A$ blows up uniformly as $t\searrow0$. By definition, see \cite{Eardley:1979}, the hypersurface $t=0$ is a crushing singularity.

\bigskip

\noindent \underline{\textit{Past timelike geodesic incompleteness.}}
For the Einstein-scalar field equations, we can see from \eqref{ESF.1} that the physical scalar curvature $\Rb_{ab}$ satisfies the strong energy condition, namely
\begin{equation*}
    \Rb_{ab}\xi^a\xi^b = 2(\nablab_a\varphi)\xi^a (\nablab_b\varphi)\xi^b \geq 0,
\end{equation*}
for any timelike vector field $\xi^a$. Moreover, on the time interval $t\in(0,t_0]\cap(0,1]$, the hypersurface $\{t\}\times\mathbb{B}_{\rho(t)}$ is a spacelike Cauchy hypersurface in $\Omega_{\Icv}$. In view of the bound in \eqref{m-meancur-bd}, the trace of the second fundamental form $\Kttb_A{}^A$ is bounded below by some constant $\Cb>0$. Hence, by \cite[Thm. 9.5.1]{Wald:1994}, no past directed timelike curve originating from $\{t\}\times\mathbb{B}_{\rho(t)}$ can have a length greater than $3/\Cb$. In particular, all past directed timelike geodesics are incomplete.

\bigskip

\noindent \underline{\textit{Ricci and scalar curvature blowup.}}
From \eqref{physi-scalar}, \eqref{physi-Ricci} and \eqref{m-alpha-2}, the physical scalar curvature and the physical Ricci curvature invariants can be expressed as
\begin{equation*}
    \Rb = -\frac{n-1}{(n-2)\alphah^2\betah^2 t^{\frac{2(n-1)}{n-2}+r_0+2(n-1)\Hchat}} \AND \Rb_{ab}\Rb^{ab} = \frac{(n-1)^2}{(n-2)^2\alphah^4\betah^4 t^{\frac{4(n-1)}{n-2}+2r_0+4(n-1)\Hchat}}.
\end{equation*}
By \eqref{m-alphah-bd}, \eqref{m-betah-bd} and \eqref{m-Hchat-bd-1}, $\alphah$ and $\betah$ are positive and bounded and the power of $t$ in the above expressions remains positive. Hence, for $(t,x)\in\Omega_{\Icv}$ with $t\in(0,t_0]\cap(0,1]$,
\begin{equation*}
    \Rb(t,x) \lesssim -\frac{1}{t^{\frac{2(n-1)}{n-2}+r_0-2(n-1)\delta_1}} \AND \Rb_{ab}(t,x)\Rb^{ab}(t,x) \gtrsim -\frac{1}{t^{\frac{2(n-1)}{n-2}+r_0-2(n-1)\delta_1}}.
\end{equation*}
This implies that the physical curvature invariants $\Rb$ and $\Rb_{ab}\Rb^{ab}$ blow up uniformly as $t\searrow0$. The $C^2$-inextendibility of the physical metric $\gb$ then follows directly.

\bigskip

\noindent \underline{\textit{AVTD and asymptotic pointwise Kasner behavior.}}
It is clear from the frame formalism in Section~\ref{Tetrad-formalism} and Proposition~\ref{Fuch-local}, the pair
\begin{equation*}
    \Bigl\{ \gt = -\alphat^2dt\otimes dt + \gt_{\Sigma\Omega}dx^\Sigma \otimes dx^\Omega, \,\, \tau = t \Bigr\}
\end{equation*}
determines a classical solution to the conformal Einstein-scalar field equations \eqref{conf-ESF.1} to \eqref{conf-ESF.2} on $\Omega_{\Icv}$.

From the energy estimate \eqref{m-energy}, the assumption on $\ep_2$ in \eqref{ep-con-1} and the Fuchsian equation \eqref{Fuchsian-3}, we obtain
\begin{equation*}
    \int_0^{t_0} \norm{s^{-\ep_2}\Bv^\Lambda(\Wv)\del{\Lambda}\Wv}_{H^{k_0-1}(\rho(s))} \, ds < \infty.
\end{equation*}
This implies that the spatial derivative terms are negligible near $t=0$, thereby establishing the AVTD property of the solution $\{\gt,\tau\}$ to the conformal Einstein equations, as defined in \cite[\S 1.3]{BeyerOliynyk:2024b}.

Now from \eqref{conf-sec-fund} and \eqref{m-fields-decay}, the second fundamental form $\Kttt_{AB}$ of the metric $\gt$ on the $t=constant$ hypersurfaces, expressed relative to the frame $\{\et_0,\et_A\}$, satisfies that
\begin{equation*}
    2t\alphat\Kttt_{AB} = r_{AB} + 2\Hchat\delta_{AB} + 2\Sigmah_{AB} + \Ord_{H^{k-1}(\mathbb{B}_{\rhot_0})}(t^\zeta).
\end{equation*}
Set
\begin{equation} \label{kf-def}
    \kf_{AB} = r_{AB} + 2\Hchat\delta_{AB} + 2\Sigmah_{AB}.
\end{equation}
Then, from the Sobolev inequality \eqref{Sobolev}, it follows that
\begin{equation*}
    \sup_{x\in\mathbb{B}_{\rhot_0}}\lim_{t\searrow0} |2t\alphat\Kttt_{AB} - \kf_{AB}| = 0.
\end{equation*}
Using \eqref{Kasner-rels-B} and the fact that $\delta^{AB}\Sigmah_{AB}=0$, we see
\begin{equation*}
    (\kf_A{}^A)^2 - \kf_A{}^B\kf_B{}^A + 4\kf_A{}^A = 4(n-1)(n-2)\Hchat^2 + 4(n-2)r_0\Hchat + 8(n-1)\Hchat - 4r_{AB}\Sigmah^{AB} - 4\Sigmah_{AB}\Sigmah^{AB}.
\end{equation*}
Moreover, from the constraint \eqref{H-cnstr-3}, we observe
\begin{align*}
     &\quad (n-1)(n-2)\Hc^2 + (n-2)r_0\Hc + 2(n-1)\Hc - r_{AB}\Sigma^{AB} - \Sigma_{AB}\Sigma^{AB} \notag \\
        &= - 2t^{1-\ep_2}e_A(C^A{}_B{}^B) + 2U_A C^A{}_B{}^B + C_{AB}{}^B C^A{}_C{}^C + \frac{1}{4}C_{ABC}(C^{ABC}+C^{BAC}+C^{ACB}).
\end{align*}
Applying \eqref{m-W-bound}, \eqref{m-fields-decay} and the calculus inequalities \eqref{Sobolev} to \eqref{PC.3}, it is straightforward to verify that
\begin{equation*}
    (\kf_A{}^A)^2 - \kf_A{}^B\kf_B{}^A + 4\kf_A{}^A = \Ord_{H^{k-1}(\mathbb{B}_{\rhot_0})}(t^{\zeta}),
\end{equation*}
for $t\in(0,t_0]\cap(0,1]$. Letting $t\searrow0$, we conclude that
\begin{equation} \label{Asym-Kasner}
    (\kf_A{}^A)^2 - \kf_A{}^B\kf_B{}^A + 4\kf_A{}^A = 0.
\end{equation}
Solving for $\kf_A{}^A$ in \eqref{Asym-Kasner} yields that
\begin{equation*}
    \kf_A{}^A = -2 \pm \sqrt{4+\kf_{AB}\kf^{AB}}.
\end{equation*}
Note that from \eqref{m-Hchat-bd-1} and \eqref{kf-def}, we deduce
\begin{equation*}
    \kf_A{}^A = r_0 + 2(n-1)\Hchat > -\frac{2(n-1)}{n-2},
\end{equation*}
and hence,
\begin{equation*}
    \kf_A{}^A = -2 + \sqrt{4+\kf_{AB}\kf^{AB}} \geq 0.
\end{equation*}
In summary, we prove that the spacetime $(\Omega_{\Icv},\gt,\tau)$ is asymptotically pointwise Kasner, see definition in \cite[Def. 1.1]{BeyerOliynyk:2024a}.

\bigskip

\noindent \underline{\textit{Blow up of the components of the Weyl tensor.}}
According to \eqref{CtA0B0}, \eqref{m-fields-decay}, \eqref{m-Hchat-bd} and the calculus inequalities \eqref{Sobolev} to \eqref{PC.3}, we see that
\begin{equation*}
    t^2\alphat^2\Ct_{A0B0} =  \Cc_{AB} + \Csc_{AB} + \Ord_{H^{k-1}(\mathbb{B}(\rhot_0))}(t^{\zeta}),
\end{equation*}
where
\begin{align}
    \Cc_{AB} &= \frac{n-3}{2(n-2)}r_{AB} + \frac{r_0}{2(n-2)}\delta_{AB} + \frac{r_0}{4}r_{AB} - \frac{1}{4}r_{AC}r_B^C, \label{CcAB-def} \\
    \Csc_{AB} &= \frac{n-3}{n-2}\Sigmah_{AB} + (n-3)\Hchat\Sigmah_{AB} + \Biggl(\frac{n-3}{2}r_{AB} - \frac{(n-3)r_0}{2(n-1)}\delta_{AB}\Biggr)\Hchat + \frac{(n-3)r_0}{2(n-1)}\Sigmah_{AB} \notag \\
        &\quad + \frac{1}{n-1}\Sigmah_{CD}\Sigmah^{CD}\delta_{AB} + \frac{1}{n-1}r_{CD}\Sigmah^{CD}\delta_{AB} - \Sigmah_{AC}\Sigmah_B{}^C - \Sigmah_{C(A}r_{B)}^C + \frac{r_0}{n-1}\Sigmah_{AB}. \notag
\end{align}
Since the exponents $r_0$, $r_A$ and $\Hc$, $\Hchat$, $\Sigma_{AB}$ and $\Sigmah_{AB}$ are bounded due to \eqref{m-W-bound} and \eqref{m-Hchat-bd}, it follows from \eqref{Weyl-invar-def} and \eqref{m-alpha-2} that
\begin{equation} \label{Weyl-fix}
    \Cb_{A0B0}\Cb^{A0B0} = \frac{1}{\alphah^4\betah^4 t^{\frac{4(n-1)}{n-2}+2r_0+4(n-1)\Hchat}}\Bigl((\Cc_{AB}+\Csc_{AB})(\Cc^{AB}+\Csc^{AB}) + \Ord_{H^{k-1}(\mathbb{B}(\rhot_0))}(t^{\zeta})\Bigr)
\end{equation}
By \eqref{CcAB-def} and the definition of $r_{AB}$ in \eqref{rAB-def-1} and \eqref{rAB-def-2}, we derive that
\begin{equation*}
    \Cc_{AB}\Cc^{AB} = \sum_{A=1}^{n-1}\Biggl( \frac{n-3}{2(n-2)}r_A + \frac{r_0}{2(n-2)} + \frac{r_0}{4}r_A - \frac{1}{4}r_A^2 \Biggr)^2.
\end{equation*}
It is clear that $\Cc_{AB}\Cc^{AB}\geq0$. Moreover, in the case of an FLRW metric, that is, when all conformal exponents vanish (see \eqref{FLRW-def}), we have $\Cc_{AB}\Cc^{AB}=0$.

Now we claim that $\Cc_{AB}\Cc^{AB}=0$ if and only if $r_A=0$ for $A=1,2,...,n-1$ and we prove this statement by contradiction. Suppose there is at least one nonzero conformal Kasner exponent (so $r_0>0$), and yet we have $\Cc_{AB}\Cc^{AB}=0$, then two cases follow: either there is at least one zero conformal Kasner exponent $r_C$, or all the exponents are nonzero. In the first case, for any index $C$ with $r_C=0$, we have
\begin{equation*}
    \frac{n-3}{2(n-2)}r_C + \frac{r_0}{2(n-2)} + \frac{r_0}{4}r_C - \frac{1}{4}r_C^2 = \frac{r_0}{2(n-2)} \neq 0,
\end{equation*}
due to $r_0>0$. This implies $\Cc_{AB}\Cc^{AB}>0$, a contradiction. In the second case, we must have
\begin{equation} \label{rA-contradiction}
    \frac{n-3}{2(n-2)}r_A + \frac{r_0}{2(n-2)} + \frac{r_0}{4}r_A - \frac{1}{4}r_A^2 = 0,
\end{equation}
for $A=1,2,...,n-1$. Solving for $r_A$ in \eqref{rA-contradiction} yields
\begin{equation*}
    r_A = \frac{r_0}{2} + \frac{n-3}{n-2} \pm \sqrt{\frac{1}{4}r_0^2 + \frac{n-1}{n-2}r_0 + \frac{(n-3)^2}{(n-2)^2}}.
\end{equation*}
Consequently, summing over $A$ and using the relation $\sum_{A=1}^{n-1}=r_0$, we derive
\begin{align*}
    &\, \sum_{A=1}^{n-1}r_A = \frac{n-1}{2}r_0 + \frac{(n-1)(n-3)}{n-2} \pm (n-1)\sqrt{\frac{1}{4}r_0^2 + \frac{n-1}{n-2}r_0 + \frac{(n-3)^2}{(n-2)^2}} = r_0, \\
    \Longleftrightarrow& \, \frac{n-3}{2}r_0 + \frac{(n-1)(n-3)}{n-2} = \mp (n-1)\sqrt{\frac{1}{4}r_0^2 + \frac{n-1}{n-2}r_0 + \frac{(n-3)^2}{(n-2)^2}}, \\
    \Longleftrightarrow& \, \frac{(n-3)^2}{4}r_0^2 + \frac{(n-1)(n-3)^2}{n-2}r_0 = \frac{(n-1)^2}{4}r_0^2 + \frac{(n-1)^3}{n-2}r_0.
\end{align*}
The only solution for the resulting equation is $r_0=0$, again a contradiction. Hence, the claim is proved.

Assume now that we have at least on nonzero conformal Kasner exponent, so that $\Cc_{AB}\Cc^{AB}>0$. By further restricting $\delta_0$ and $\delta_1$, we can ensure from \eqref{m-W-bound}, \eqref{m-Hchat-bd0} and \eqref{m-Hchat-bd} that
\begin{equation*}
    \inf_{x\in\mathbb{B}_{\rhot_0}}(\Cc_{AB}+\Csc_{AB})(\Cc^{AB}+\Csc^{AB}) > 0,
\end{equation*}
where we have used the calculus inequalities \eqref{Sobolev} to \eqref{PC.3}. As a result, from \eqref{m-alphah-bd}, \eqref{m-betah-bd}, \eqref{m-Hchat-bd-1} and \eqref{Weyl-fix}, we have
\begin{equation*}
    \inf_{x\in\mathbb{B}_{\rhot_0}}\Cb_{A0B0}\Cb^{A0B0} \gtrsim \frac{1}{t^{\frac{4(n-1)}{n-2}+2r_0-4\delta_1}}
\end{equation*}
for $t\in(0,t_0]\cap(0,1]$. Since $\frac{4(n-1)}{n-2}+2r_0-4\delta_1>0$, the invariant $\Cb_{A0B0}\Cb^{A0B0}$, as a component of the invariant $\Cb_{abcd}\Cb^{abcd}$, blows up uniformly as $t\searrow0$ as long as there is at least one nonzero conformal Kasner exponent.

\begin{rem} \label{main-thm-v2}
    Theorem~\ref{main-thm} can be modified to establish the past stability of the Kasner-scalar field metrics on $M_{0,t_0}$, i.e. a global-in-space result; see also \cite[Thm. 6.1]{BOZ:2025}. The modification involves replacing the sets $\mathbb{B}_{\rho_0}$, $\mathbb{B}_{\rho(t)}$ and $\mathbb{B}_{\rhot_0}$ with $\Tbb^{n-1}$, and substituting $\Omega_{\Icv}$ by $M_{0,t_0}$. The proof follows essentially the same argument, except that we use Proposition~\ref{global-local} instead of Proposition~\ref{local-local} to obtain the local-in-time existence on $M_{t_1,t_0}$, and Proposition~\ref{Fuch-global} instead of Proposition~\ref{Fuch-local} to derive the global-in-time existence and corresponding estimates.
\end{rem}

\appendix

\section{Matrix Adjoint and Matrix Identities} \label{Madj-Miden}
In this appendix, we derive the adjoint of $\delta_{[A}^D\delta_{C]}^{\langle P}\delta_B^{Q\rangle}$ and $\delta^{D\langle P}\delta_{[A}^{Q\rangle}\delta_{C]B}$. Additionally, several identities useful for the symmetrization procedure in Section~\ref{Symmetrization} are listed. These identities are stated without proof, as they follow from straightforward calculations.

\subsection{Adjoint of $\delta_{[A}^D\delta_{C]}^{\langle P}\delta_B^{Q\rangle}$ and $\delta^{D\langle P}\delta_{[A}^{Q\rangle}\delta_{C]B}$}
For a fixed index $D$, we define the matrices $M_3$ and $M_4$ to be
\begin{equation*}
    (M_3)^D{}_{ABC}{}^{PQ} = \delta_{[A}^D\delta_{C]}^{\langle P}\delta_B^{Q\rangle} \AND (M_4)^D{}_{ABC}{}^{PQ} = \delta^{D\langle P}\delta_{[A}^{Q\rangle}\delta_{C]B}.
\end{equation*}
Applying the index conventions established in Section~\ref{Index}, the expansion of $M_3$ yields
\begin{align*}
    (M_3)^{DABCPQ} &= \frac{1}{4}\delta^{AD}\delta^{CP}\delta^{BQ} + \frac{1}{4}\delta^{AD}\delta^{CQ}\delta^{BP} - \frac{1}{2(n-1)}\delta^{PQ}\delta^{AD}\delta^{CE}\delta_E^B \notag \\
        &\quad - \frac{1}{4}\delta^{CD}\delta^{AP}\delta^{BQ} - \frac{1}{4}\delta^{CD}\delta^{AQ}\delta^{BP} + \frac{1}{2(n-1)}\delta^{PQ}\delta^{CD}\delta^{AE}\delta_E^B.
\end{align*}
Renaming the indices gives
\begin{align*}
    (M_3)^{DPQRAB} &= \frac{1}{4}\delta^{PD}\delta^{RA}\delta^{QB} + \frac{1}{4}\delta^{PD}\delta^{RB}\delta^{QA} - \frac{1}{2(n-1)}\delta^{AB}\delta^{PD}\delta^{RE}\delta_E^Q \notag \\
        &\quad - \frac{1}{4}\delta^{RD}\delta^{PA}\delta^{QB} - \frac{1}{4}\delta^{RD}\delta^{PB}\delta^{QA} + \frac{1}{2(n-1)}\delta^{AB}\delta^{RD}\delta^{PE}\delta_E^Q,
\end{align*}
from which it follows that the adjoint of $(M_3)^D{}_{ABC}{}^{PQ}$ is
\begin{align} \label{M3-adj}
    (M_3)^{DPQR}{}_{AB} &= \delta^{D[P}\delta_{\langle A}^{R]}\delta_{B\rangle}^Q.
\end{align}
A similar calculation shows that the adjoint of $(M_4)^D{}_{ABC}{}^{PQ}$ is
\begin{equation} \label{M4-adj}
    (M_4)^{DPQR}{}_{AB} = \delta_{\langle A}^D\delta_{B\rangle}^{[P}\delta^{R]Q}.
\end{equation}

\subsection{Matrix identities.}
The following matrix identities are provided for reference:
\begin{align}
    \delta_A^{[E}\delta^{G]F}\delta_{[E}^P\delta_{G]F} &= \frac{n-2}{2}\delta_A^P, \label{dd-1} \\
    \delta_{[A}^E\delta_{C]B}\delta_E^{[P}\delta^{R]Q} &= \frac{1}{4}(\delta_A^P\delta_{BC}\delta^{QR} - \delta_A^R\delta_{BC}\delta^{PQ} - \delta_C^P\delta_{AB}\delta^{QR} + \delta_C^R\delta_{AB}\delta^{PQ}), \label{dd-2} \\
    \delta_A^{[I}\delta^{K]J}(M_3)^D{}_{IJK}{}^{PQ} &= -\frac{1}{2}\delta^{D\langle P}\delta_A^{Q\rangle}, \label{d-pi1} \\
    \delta_A^{[I}\delta^{K]J}(M_4)^D{}_{IJK}{}^{PQ} &= \frac{n-2}{2}\delta^{D\langle P}\delta_A^{Q\rangle}, \label{d-pi2} \\
    \delta_{[A}^E\delta_{C]B}\delta_E^{[I}\delta^{K]J}(M_3)^D{}_{IJK}{}^{PQ} &= -\frac{1}{2}(M_4)^D{}_{ABC}{}^{PQ}, \label{dd-pi1} \\
    \delta_{[A}^E\delta_{C]B}\delta_E^{[I}\delta^{K]J}(M_4)^D{}_{IJK}{}^{PQ} &= \frac{n-2}{2}(M_4)^D{}_{ABC}{}^{PQ}. \label{dd-pi2}
\end{align}

\section{Positive Definiteness of Matrices}
\begin{lem} \label{Mc-PD}
    Consider a symmetric matrix $(\Mc_{ABC}^{PQR})$ of the form $a\delta_A^P\delta_B^Q\delta_C^R+b\delta_{[A}^E\delta_{C]B}\delta_E^{[P}\delta^{R]Q}$, where $a,b\in\Rbb$, with indices ranging from $1$ to $n-1$ and adopting the convention in \eqref{index-op} for the square brackets. Then, if $a$ and $b$ satisfy the inequality $a>\frac{|b|}{2}$\footnote{Here we only care about the sufficient condition for this matrix to be positive definite.}, the matrix $(\Mc_{ABC}^{PQR})$ is positive definite.
\end{lem}

\begin{proof}
    First, observe that for any vector $v_{PQR}$, a direct calculation yields
    \begin{equation} \label{PDM-1}
        v_{PQR}\Bigl(\delta_{[A}^E\delta_{C]B}\delta_E^{[P}\delta^{R]Q}\Bigr)v^{ABC} = \frac{1}{4}\sum_{A=1}^{n-1}\Biggl( \Bigl(\sum_{B=1}^{n-1}v_{ABB}\Bigr) - \Bigl(\sum_{B=1}^{n-1}v_{BBA}\Bigr) \Biggr)^2 \geq 0,
    \end{equation}
    where we use the identity \eqref{dd-2}. By expanding the bracket on the right-hand side of \eqref{PDM-1}, we obtain
    \begin{align}
        \Biggl( \Bigl(\sum_{B=1}^{n-1}v_{ABB}\Bigr) - \Bigl(\sum_{B=1}^{n-1}v_{BBA}\Bigr) \Biggr)^2 &= \sum\limits_{\substack{B=1 \\ B\neq A}}^{n-1}(v_{ABB})^2 + \sum\limits_{\substack{B=1 \\ B\neq A}}^{n-1}(v_{BBA})^2 + \sum_{\text{\scalebox{0.5}{$A\neq B\neq C\neq A$}}}2v_{ABB}v_{ACC} \notag \\
            &\quad + \sum_{\text{\scalebox{0.5}{$A\neq B\neq C\neq A$}}}2v_{BBA}v_{CCA} - 2\sum_{\text{\scalebox{0.5}{$B,C\neq A$}}}2v_{ABB}v_{CCA} \geq 0. \label{PDM-2}
    \end{align}
    A straightforward deduction from \eqref{PDM-2}, which will be utilized subsequently, is that
    \begin{equation} \label{PDM-3}
        \sum\limits_{\substack{B=1 \\ B\neq A}}^{n-1}(v_{ABB})^2 + \sum\limits_{\substack{B=1 \\ B\neq A}}^{n-1}(v_{BBA})^2 \geq \Biggl| \sum_{\text{\scalebox{0.5}{$A\neq B\neq C\neq A$}}}2v_{ABB}v_{ACC} + \sum_{\text{\scalebox{0.5}{$A\neq B\neq C\neq A$}}}2v_{BBA}v_{CCA} - 2\sum_{\text{\scalebox{0.5}{$B,C\neq A$}}}2v_{ABB}v_{CCA} \Biggr|.
    \end{equation}
    Consequently, if $a>\frac{|b|}{2}$, then for any nonzero vector $v_{ABC}$, it follows that
    \begin{align*}
        v_{PQR}\,\Mc_{ABC}^{PQR}\,v^{ABC} &> \frac{|b|}{2}\sum_{A,B,C=1}^{n-1}(v_{ABC})^2 + \frac{b}{4}\sum_{A=1}^{n-1}\Biggl( \sum\limits_{\substack{B=1 \\ B\neq A}}^{n-1}(v_{ABB})^2 + \sum\limits_{\substack{B=1 \\ B\neq A}}^{n-1}(v_{BBA})^2 \notag \\
        &\quad + \sum_{\text{\scalebox{0.5}{$A\neq B\neq C\neq A$}}}2v_{ABB}v_{ACC} + \sum_{\text{\scalebox{0.5}{$A\neq B\neq C\neq A$}}}2v_{BBA}v_{CCA} - 2\sum_{\text{\scalebox{0.5}{$B,C\neq A$}}}2v_{ABB}v_{CCA} \Biggr) \notag \\
        &\geq \frac{|b|}{4}\sum_{A=1}^{n-1}\Biggl( \sum\limits_{\substack{B=1 \\ B\neq A}}^{n-1}(v_{ABB})^2 + \sum\limits_{\substack{B=1 \\ B\neq A}}^{n-1}(v_{BBA})^2 \Biggr) \notag \\
        &\quad + \frac{b}{4}\sum_{A=1}^{n-1}\Biggl( \sum_{\text{\scalebox{0.5}{$A\neq B\neq C\neq A$}}}2v_{ABB}v_{ACC} + \sum_{\text{\scalebox{0.5}{$A\neq B\neq C\neq A$}}}2v_{BBA}v_{CCA} - 2\sum_{\text{\scalebox{0.5}{$B,C\neq A$}}}2v_{ABB}v_{CCA} \Biggr) \notag \\
        &\geq 0,
    \end{align*}
    where the second inequality employs the fact that
    \begin{align*}
        &\quad \frac{|b|}{2}\sum_{A,B,C=1}^{n-1}(v_{ABC})^2 + \frac{b}{4}\sum_{A=1}^{n-1}\Biggl( \sum\limits_{\substack{B=1 \\ B\neq A}}^{n-1}(v_{ABB})^2 + \sum\limits_{\substack{B=1 \\ B\neq A}}^{n-1}(v_{BBA})^2 \Biggr) \\
        &\geq \Biggl(\frac{|b|}{2}+\frac{b}{4}\Biggr)\sum_{A=1}^{n-1} \Biggl( \sum\limits_{\substack{B=1 \\ B\neq A}}^{n-1}(v_{ABB})^2 + \sum\limits_{\substack{B=1 \\ B\neq A}}^{n-1}(v_{BBA})^2 \Biggr) \geq \frac{|b|}{4}\sum_{A=1}^{n-1} \Biggl( \sum\limits_{\substack{B=1 \\ B\neq A}}^{n-1}(v_{ABB})^2 + \sum\limits_{\substack{B=1 \\ B\neq A}}^{n-1}(v_{BBA})^2 \Biggr),
    \end{align*}
    and in the final step, we use \eqref{PDM-3}, which completes the proof.
\end{proof}

\section{Algebraic and Calculus Inequalities}
This appendix compiles several inequalities that are utilized throughout this article. Their proofs are either straightforward or can be found in \cite[Ch.~4]{AdamsFournier:2003}, \cite[Ch.~VI, \S 3]{Choquet_et_al:2000} and \cite[Ch.~13, \S 1-3]{TaylorIII:1996}.

\begin{thm}{\emph{[Cauchy-Schwarz inequality]}}
    For any real numbers $a_1,a_2,...,a_n$, we have
    \begin{equation} \label{CS-ineq}
        (a_1+a_2+\cdot\cdot\cdot+a_n)^2 \leq n(a_1^2 + a_2^2 + \cdot\cdot\cdot + a_n^2).
    \end{equation}
\end{thm}

\begin{thm}{\emph{[Sobolev's inequality]}}
    Suppose $k\in\Zbb_{\geq0}$ and $0<\alpha<k-n/2\leq1$, then we have $H^k(\Tbb^{n})\subset C^{0,\alpha}(\Tbb^{n})$ and
    \begin{equation} \label{Sobolev}
        \norm{u}_{L^{\infty}} \lesssim \norm{u}_{C^{0,\alpha}} \lesssim \norm{u}_{H^k}
    \end{equation}
    for all $u\in H^k(\Tbb^n)$.
\end{thm}

\begin{thm}{\emph{[Product and commutator inequalities]}}
    \begin{enumerate}[(a)]
    \item
    Suppose $k\in \Zbb_{\geq 1}$ and $|\alpha|=k$, then
    \begin{align}
        \norm{D^\alpha (uv)}_{L^2} &\lesssim \norm{u}_{H^{k}}\norm{v}_{L^{\infty}} + \norm{u}_{L^{\infty}}\norm{v}_{H^{k}}, \label{PC.1} \\
        \norm{[D^\alpha,u]v}_{L^2} &\lesssim \norm{D u}_{L^{\infty}}\norm{v}_{H^{k-1}} + \norm{D u}_{H^{k-1}}\norm{v}_{L^{\infty}}, \label{PC.2}
    \end{align}
    for all $u,v \in C^\infty(U)$.
    \item Suppose $k_1,k_2,k_3\in \Zbb_{\geq 0}$, $\;k_1,k_2\geq k_3$ and $k_1+k_2-k_3 > n/2$, then
    \begin{equation}
        \norm{uv}_{H^{k_3}} \lesssim \norm{u}_{H^{k_1}}\norm{v}_{H^{k_2}}, \label{PC.3}
    \end{equation}
    for all $u\in H^{k_1}(U)$ and $v\in H^{k_2}(U)$.
    \end{enumerate}
\end{thm}

\begin{thm}{\emph{[Moser's inequality]}}
    Suppose $k\in\Zbb_{\geq1}$, $0\leq s\leq k$, $|\alpha|=s$, $f\in C^k(V)$, where $V$ is open and bounded in $\Rbb^N$ and contains 0, and $f(0)=0$, then
    \begin{equation} \label{Moser}
        \norm{D^{\alpha}f(u)}_{L^2(U)} \leq C\bigl(\norm{f}_{C^k(\overline{V})}\bigr) (1+\norm{u}_{L^{\infty}(U)}^{k-1})\norm{u}_{H^k(U)}
    \end{equation}
    for all $u\in C^0(U)\cap L^{\infty}(U)\cap H^{k}(U)$ with $u(x)\in V$ for all $x\in U$.
\end{thm}

\bibliographystyle{amsplain}
\bibliography{references}

@BOOK{AdamsFournier:2003,
   author =     {R.A.~Adams and J.~Fournier},
   title=       {Sobolev Spaces},
   publisher =  {Academic Press},
   edition =    {$2^{\text{nd}}$},
   year =       {2003}
   }

@article{ABIO:2022,
  author =       {F.~Beyer and E.~Ames and J.~Isenberg and T.A.~Oliynyk},
  title =        {Stability of {AVTD} behavior within the polarized $\mathbb{T}^ 2$-symmetric vacuum spacetimes},
  journal = {Ann. Inst. Henri Poincar\'{e}},
  volume = {23},
  pages = {2299-2343},
  year = {2022},
  }

@article{ABIO:2022_Royal_Soc,
      title={Stability of Asymptotic Behavior Within Polarised {$\mathbb{T}^2$}-Symmetric Vacuum Solutions with Cosmological Constant}, 
      author={E.~Ames and F.~Beyer and J.~Isenberg and T.A.~Oliynyk},
      journal = {Phil.~Trans.~R.~Soc.~A},
      year={2022},
      volume=380,
      issue=2222,
}

@Article{AthanasiouFournodavlos:2024,
      title={A localized construction of {K}asner-like singularities}, 
      author={N.~Athanasiou and G.~Fournodavlos},
      journal={Commun. Math. Phys.},
      year={2025},
      volume={406},
      page={252},
}

@article{Barrow:1978,
  author =      {J.~Barrow},
  title =       {Quiescent cosmology},
  journal =     {Nature},
  year =        {1978},
  volume =      {272},
  pages =       {211–215},
  doi ={10.1038/272211a0}
  }

@article{belinskii1970,
  author = {V.A.~Belinskii and I.M.~Khalatnikov and E.M.~Lifshitz},
  year = {1970},
  volume = {19},
  pages = {525--573},
  doi = {10.1080/00018737000101171},
  journal = {Adv. Phys.},
  number = {80},
  title = {Oscillatory Approach to a Singular Point in the Relativistic Cosmology},
}

@article{belinskiiKhalatnikov:1972,
  author = {V.A.~Belinskii and I.M.~Khalatnikov},
  year = {1972},
  volume = {63},
  pages = {1121-1134},
  journal = {\u{Z}.~\`{E}ksper.~Teoret. Fiz.},
  title = { Effect of scalar and vector fields on the nature of the cosmological
singularity},
}

@BOOK{BenzoniSerre:2007,
   author =     {S.~Benzoni-Gavage and D.~Serre},
   title=       {Multi-dimensional hyperbolic partial differential equations: first-order systems and applications},
   publisher =  {Oxford University Press},
   address = {Oxford},
   year =       {2007}
   }

@Article{BeyerOliynyk:2024a,
 author =       {F.~Beyer and T.A.~Oliynyk},
  title =        {Relativistic perfect fluids near {K}asner singularities},
  journal =     {Comm. Anal. Geom.},
  volume = {32},
  pages = {1701-1794},
  year  = {2024},
  }

@Article{BeyerOliynyk:2024b,
  author =       {F.~Beyer and T.A.~Oliynyk},
  title =        {Localized big bang stability for the {E}instein-scalar field equations},
  year =         {2024},
  volume =    {248},
  journal = 	{Arch. Rat. Mech. Anal.},
  pages = 	{3}
  }

@Article{BeyerOliynyk:2024c,
  author =       {F.~Beyer and T.A.~Oliynyk},
  title =        {Past stability of {FLRW} solutions to the {E}instein-{E}uler-scalar field equations and their big bang singularities},
  journal = 	{Beijing J. of Pure and Appl. Math.},
  year =      {2024},
  volume = 	{1},
  pages = 	{515-637}
  }

@Article{BOOS:2021,
 author =       {F.~Beyer and T.A.~Oliynyk and J.A~Olvera-SantaMar\'{i}a},
  title =        {The {F}uchsian approach to global existence for hyperbolic equations},
  journal = {Comm. Part. Diff. Eqn.},
  volume = {46},
  pages = {864-934},
  year  = {2021}
  }

@BOOK{Choquet_et_al:2000,
   author = {Y.~Choquet-Bruhat and C.~De Witt-Morette},
   title= {Analysis, {M}anifolds and {P}hysics {P}art {II}},
   publisher = {North-Holland, Amsterdam},
   edition = {revised and enlarged},
   year = {2000}
   }

@article{Eardley:1979,
  title = {Time Functions in Numerical Relativity: {{Marginally}} Bound Dust Collapse},
  author = {D.M.~Eardley and L.~Smarr},
  year = {1979},
  journal = {Phys. Rev. D},
  volume = {19},
  pages = {2239},
  doi = {10.1103/PhysRevD.19.2239}
}

@article{FajmanUrban:2025,
    author = "D.~Fajman and L.~Urban",
    title = "{Cosmic censorship near FLRW spacetimes with negative spatial curvature}",
    eprint = "2211.08052",
    archivePrefix = "arXiv",
    primaryClass = "gr-qc",
    doi = "10.2140/apde.2025.18.1615",
    journal = "Anal. Part. Diff. Eq.",
    volume = "18",
    number = "7",
    pages = "1615--1713",
    year = "2025"
}

@Unpublished{FajmanUrban:2024,
      title={On the past maximal development of near-{FLRW} data for the {E}instein scalar-field {V}lasov system}, 
      author={D.~Fajman and L.~Urban},
      year={2024},
      note ={Preprint. \href{https://arxiv.org/abs/2402.08544}{arXiv:2402.08544}}
}

@Article{FournodavlosLuk:2023,
author = {G.~Fournodavlos and J.~Luk},
title = {Asymptotically {K}asner-like singularities},
year = {2023},
journal = {Amer.~J.~Math},
volume = {145},
pages = {1183-1272},
doi = {10.1353/ajm.2023.a902957}
}

@article{Fournodavlos_et_al:2023,
	title = {Stable {Big} {Bang} formation for {Einstein}’s equations: {The} complete sub-critical regime},
	volume = {36},
	doi = {10.1090/jams/1015},
	journal = {J. Amer. Math. Soc.},
	author = {G.~Fournodavlos and I.~Rodnianski and J.~Speck},
	year = {2023},
	pages = {827--916}
}

@article{GarfinkleGundlach:2005,
	title = {Well-posedness of the scale-invariant tetrad formulation of the vacuum {Einstein} equations},
	volume = {22},
	journal = {Class. Quantum Grav.},
	author = {D.~Garfinkle and C.~Gundlach},
	year = {2005},
	pages = {2679},
}

@unpublished{Groeniger_et_al:2023,
      title={Formation of quiescent big bang singularities}, 
      author={H.O.~Groeniger and O.~Petersen and H.~Ringstr\"{o}m},
      year={2023},
      note =  {preprint [arXiv:2309.11370]},
      doi = {10.48550/arXiv.2309.11370}
}

@article{IsenbergMoncreif:1990,
author = {J.~Isenberg and V.~Moncrief},
title = {{Asymptotic behavior of the gravitational field and the nature of singularities in Gowdy spacetimes}},
journal = {Ann.\ Phys.},
year = {1990},
volume = {199},
number = {1},
pages = {84--122},
publisher = {Elsevier}
}

@article{IsenbergKichenassamy:1999,
author = {J.~Isenberg  and S.~Kichenassamy},
title = {Asymptotic behavior in polarized ${T}^2$-symmetric vacuum space–times},
journal = {J. Math. Phys.},
volume = {40},
number = {1},
pages = {340-352},
year = {1999},
doi = {10.1063/1.532775},
}

@Book{Lax:2006,
  author =       {P.D.~Lax},
  title =        {Hyperbolic Partial Differential Equations},
  publisher =    {AMS/CIMS},
  year =         {2006}
  }

@unpublished{Li:2024,
      title={Scattering towards the singularity for the wave equation and the linearized {E}instein-scalar field system in {K}asner spacetimes}, 
      author={W.~Li},
      year={2024},
      note={Preprint. \href{http://arxiv.org/abs/2401.08437}{arXiv:2401.08437}}
}

@article{lifshitz1963,
  author = {E.M.~Lifshitz and I.M.~Khalatnikov},
  year = {1963},
  volume = {12},
  pages = {185--249},
  doi = {10.1080/00018736300101283},
  journal = {Adv. Phys.},
  number = {46},
  title = {Investigations in Relativistic Cosmology},
}

@article{RodnianskiSpeck:2018b,
  author = {I.~Rodnianski and J.~Speck},
  year = {2018},
  volume = {187},
  pages = {65--156},
  issn = {0003-486X},
  doi = {10.4007/annals.2018.187.1.2},
  journal = {Ann. Math.},
  number = {1},
  title = {A Regime of Linear Stability for the {{Einstein}}-Scalar Field System with Applications to Nonlinear {{Big Bang}} Formation},
}

@article{RodnianskiSpeck:2018c,
  author = {I.~Rodnianski and J.~Speck},
  year = {2018},
  volume = {24},
  pages = {4293--4459},
  issn = {1022-1824, 1420-9020},
  doi = {10.1007/s00029-018-0437-8},
  journal = {Sel. Math. New Ser.},
  number = {5},
  title = {Stable {{Big Bang}} Formation in Near-{{FLRW}} Solutions to the {{Einstein}}-Scalar Field and {{Einstein}}-Stiff Fluid Systems},
}

@article{RodnianskiSpeck:2022,
	title = {On the nature of {Hawking}’s incompleteness for the {Einstein}-vacuum equations: {The} regime of moderately spatially anisotropic initial data},
	volume = {24},
	doi = {10.4171/jems/1092},
	urldate = {2023-08-03},
	author = {I.~Rodnianski and J.~Speck},
  journal = {J. Eur. Math. Soc.},
	year = {2021},
	pages = {167--263}
}

@article{ringstrom2009a,
	title = {Strong cosmic censorship in ${T}^3$-{Gowdy} spacetimes},
	volume = 170,
	url = {http://annals.math.princeton.edu/2009/170-3/p04},
	doi = {10.4007/annals.2009.170.1181},
	number = 3,
	journal = {Ann. Math.},
	author = {H.~Ringstr\"{o}m},
	month = nov,
	year = 2009,
	pages = {1181--1240}
}

@article{Speck:2018,
	title = {The {Maximal} {Development} of {Near}-{FLRW} {Data} for the {Einstein}-{Scalar} {Field} {System} with {Spatial} {Topology} $\mathbb{S}^3$},
	volume = {364},
	doi = {10.1007/s00220-018-3272-z},
	journal = {Comm. Math. Phys.},
	author = {J.~Speck},
	year = {2018},
	pages = {879--979}
}

@Book{TaylorIII:1996,
  author =       {M.E.~Taylor},
  title =        {Partial differential equations {III}: {N}onlinear equations},
  publisher =    {Springer},
  year =         {1996}
  }

@article{Uggla_et_al:2003,
	title = {Past attractor in inhomogeneous cosmology},
	volume = {68},
	journal = {Phys. Rev. D},
	author = {C.~Uggla and H.~van Elst and J.~Wainwright and G.F.R~Ellis},
	year = {2003},
	pages = {103502}
}

@unpublished{Urban:2024,
      title={Quiescent {Big} {Bang} formation in $2+1$ dimensions}, 
      author={L.~Urban},
      year={2024},
      note ={Preprint. \href{http://arxiv.org/abs/2412.03396}{arXiv:2412.03396}}
}

@article{vanElstUggla:1997,
	title = {General relativistic 1+3 orthonormal frame approach},
	volume = {14},
	journal = {Class. Quantum Grav.},
	author = {H.~van Elst and C.~Uggla},
	year = {1997},
	pages = {2673}
}

@Book{Wald:1994,
  author =       {R.M.~Wald},
  title =        {General Relativity},
  publisher =    {University of Chicago Press},
  year =         {1984},
  }

@article{beyer2017,
  volume = {42},
  number = {8},
  journal = {Commun. Part. Diff. Eq.},
  doi = {10.1080/03605302.2017.1345938},
  author = {F.~Beyer and P.G.~LeFloch},
  year = {2017},
  pages = {1199-1248},
  title = {Self\textendash{}Gravitating Fluid Flows with {{Gowdy}} Symmetry near Cosmological Singularities},
}

@article{stahl2002,
  volume = {19},
  number = {17},
  journal = {Class. Quantum Grav.},
  doi = {10.1088/0264-9381/19/17/301},
  author = {F.~St{\aa}hl},
  year = {2002},
  pages = {4483-4504},
  title = {Fuchsian Analysis of {$S^2\times S^1$} and {$S^3$} {{Gowdy}} Spacetimes},
}

@article{choquet-bruhat2004,
  volume = {119},
  number = {7-9},
  journal = {Nuovo Cim. B},
  doi = {10.1393/ncb/i2004-10174-x},
  author = {Y.~{Choquet-Bruhat} and J.~Isenberg and V.~Moncrief},
  year = {2004},
  pages = {625-638},
  title = {Topologically General {$U(1)$} Symmetric Vacuum Space-Times with {{AVTD}} Behavior},
}

@article{choquet-bruhat2006,
  volume = {56},
  number = {8},
  journal = {J. Geom. Phys.},
  doi = {10.1016/j.geomphys.2005.06.011},
  author = {Y.~{Choquet-Bruhat} and J.~Isenberg},
  year = {2006},
  pages = {1199-1214},
  title = {Half Polarized {$U(1)$}-Symmetric Vacuum Spacetimes with {{AVTD}} Behavior},
}

@article{ames2013a,
  volume = {14},
  number = {6},
  journal = {Ann. Henri Poincar{\'e}},
  doi = {10.1007/s00023-012-0228-2},
  author = {E.~Ames and F.~Beyer and J.~Isenberg and P.G.~LeFloch},
  year = {2013},
  pages = {1445-1523},
  title = {Quasilinear {{Hyperbolic Fuchsian Systems}} and {{AVTD Behavior}} in {$T^2$}-{{Symmetric Vacuum Spacetimes}}},
}

@article{andersson2001,
  volume = {218},
  number = {3},
  journal = {Commun. Math. Phys.},
  doi = {10.1007/s002200100406},
  author = {L.~Andersson and A.D.~Rendall},
  year = {2001},
  pages = {479-511},
  title = {Quiescent {{Cosmological Singularities}}},
}

@article{rendall2000,
  volume = {17},
  number = {16},
  journal = {Class. Quantum Grav.},
  doi = {10.1088/0264-9381/17/16/313},
  author = {A.D.~Rendall},
  year = {2000},
  pages = {3305-3316},
  title = {Fuchsian Analysis of Singularities in {{Gowdy}} Spacetimes beyond Analyticity},
}

@article{kichenassamy1998,
  volume = {15},
  number = {5},
  journal = {Class. Quantum Grav.},
  doi = {10.1088/0264-9381/15/5/016},
  author = {S.~Kichenassamy and A.D.~Rendall},
  year = {1998},
  pages = {1339-1355},
  title = {Analytic Description of Singularities in {{Gowdy}} Spacetimes},
}

@article{ames2017,
  volume = {121},
  journal = {J. Geom. Phys.},
  doi = {10.1016/j.geomphys.2017.06.005},
  author = {E.~Ames and F.~Beyer and J.~Isenberg and P.G.~LeFloch},
  year = {2017},
  pages = {42-71},
  title = {A Class of Solutions to the {{Einstein}} Equations with {{AVTD}} Behavior in Generalized Wave Gauges},
}

@article{damour2002,
  volume = {3},
  issn = {1424-0637, 1424-0661},
  number = {6},
  journal = {Ann. Henri Poincar{\'e}},
  doi = {10.1007/s000230200000},
  author = {T.~Damour and M.~Henneaux and A.D.~Rendall and M.~Weaver},
  year = {2002},
  pages = {1049-1111},
  title = {Kasner-{{Like Behaviour}} for {{Subcritical Einstein}}-{{Matter Systems}}},
}

@Book{JohnMLee-Rie:2018,
  author =       {J.M.~Lee},
  title =        {Introduction to {R}iemannian {M}anifolds},
  publisher =    {Springer},
  edition =      {$2^{\text{nd}}$},
  year =         {2018}
  }

@Book{Wainwright-Ellis:1997,
    author = {J.~Wainwright and G.F.R.~Ellis},
    title = {Dynamical {S}ystems in {C}osmology},
    publisher = {Cambridge University Press},
    year = {1997},
    }

@article{klinger2015,
  title = {A new class of asymptotically non-chaotic vacuum singularities},
  author = {P.~Klinger},
  year = {2015},
  journal = {Ann. Phys.},
  volume = {363},
  pages = {1--35}
}

@article{ringstrom2017,
  title = {Linear systems of wave equations on cosmological backgrounds with convergent asymptotics},
  author = {H.~Ringstr\"{o}m},
  year = {2020},
  journal = {Ast{\'e}risque},
  volume = {420},
  pages = {1--526},
  langid = {english}
}

@article{ringstrom2021,
  title = {Wave equations on silent big bang backgrounds},
  author = {H.~Ringstr\"{o}m},
  year = {2021},
  eprint = {2101.04939},
  primaryclass = {gr-qc},
  urldate = {2021-01-17},
  archiveprefix = {arXiv},
  note = {Preprint. \href{http://arxiv.org/abs/2101.04939}{arXiv:2101.04939}},
  journal = {},
  volume = {},
  number = {}
}

@article{ringstrom2021a,
  title = {On the geometry of silent and anisotropic big bang singularities},
  author = {H.~Ringstr\"{o}m},
  year = {2021},
  eprint = {2101.04955},
  primaryclass = {gr-qc},
  urldate = {2021-01-17},
  archiveprefix = {arXiv},
  note = {accepted for publication in Journal of Differential Geometry},
  journal = {},
  volume = {},
  number = {}
}

@article{ringstrom2022,
    author = {H.~Ringstr\"{o}m},
    title = "{Initial data on big bang singularities in symmetric settings}",
    eprint = "2202.11458",
    archivePrefix = "arXiv",
    primaryClass = "gr-qc",
    doi = "10.4310/pamq.2024.v20.n4.a2",
    journal = "Pure Appl. Math. Quart.",
    volume = "20",
    number = "4",
    pages = "1505--1539",
    year = "2024"
}

@article{ringstrom2022a,
  title = {Initial data on big bang singularities},
  author = {H.~Ringstr\"{o}m},
  year = {2022},
  publisher = {arXiv},
  doi = {10.48550/ARXIV.2202.04919},
  urldate = {2022-02-25},
  copyright = {Creative Commons Attribution 4.0 International},
  note = {accepted in JEMS}
}

@misc{li2024a,
  title = {{{BKL}} bounces outside homogeneity: {{Gowdy}} symmetric spacetimes},
  shorttitle = {{{BKL}} bounces outside homogeneity},
  author = {W.~Li},
  year = {2024},
  eprint = {2408.12427},
  primaryclass = {gr-qc},
  publisher = {arXiv},
  urldate = {2024-08-23},
  archiveprefix = {arXiv},
  note = {Preprint. \href{http://arxiv.org/abs/2408.12427}{arXiv:2408.12427}},
  journal = {},
  volume = {},
  number = {}
}

@misc{li2024b,
  title = {{{BKL}} bounces outside homogeneity: {{Einstein-Maxwell-scalar}} field in surface symmetry},
  shorttitle = {{{BKL}} bounces outside homogeneity},
  author = {W.~Li},
  year = {2024},
  eprint = {2408.12434},
  primaryclass = {gr-qc},
  publisher = {arXiv},
  urldate = {2024-08-23},
  archiveprefix = {arXiv},
  note = {Preprint. \href{http://arxiv.org/abs/2408.12434}{arXiv:2408.12434}},
  journal = {},
  volume = {},
  number = {}
}

@article{ringstrom2008,
  author = {H.~Ringstr\"{o}m},
  year = {2008},
  journal = {Class. Quantum Grav.},
  volume = {25},
  number = {11},
  pages = {114010},
  title = {Strong cosmic censorship in the case of {$T^3$}-{{Gowdy}} vacuum spacetimes}
}

@article{chruscielStrongCosmicCensorship1990,
  title = {Strong cosmic censorship in polarised {{Gowdy}} spacetimes},
  author = {P.T.~Chru{\'s}ciel and J.~Isenberg. and V.~Moncrief},
  year = {1990},
  journal = {Classical and Quantum Gravity},
  volume = {7},
  number = {10},
  pages = {1671--1680}
}

@Article{BOZ:2025,
      title = {Localized past stability of the subcritical {K}asner-scalar field spacetimes}, 
      author = {F.~Beyer and T.A.~Oliynyk and W.~Zheng},
      year = {2025},
      note = {preprint [arXiv:2502.09210]},
}

@BOOK{HornJohnson:2013,
   author = {R.A.~Horn and C.R.~Johnson},
   title = {Matrix Analysis},
   publisher = {Cambridge University Press},
   edition = {$2^{\text{nd}}$},
   year = {2013}
   }

@BOOK{Ringstrom-CauchyGR:2009,
    author = {H.~Ringstr\"{o}m},
    title = {The {C}auchy Problem in General Relativity},
    publisher = {European Mathematical Society Publishing House},
    year = {2009},
    }

@article{Demaret:1985,
    title = {Non-oscillatory behaviour in vacuum {K}aluza-{K}lein cosmologies},
    author = {J.~Demaret and M.~Henneaux and P.~Spindel},
    journal = {Physics Letters B},
    volume = {164},
    number = {1-3},
    year = {1985},
    pages = {27-30},
    issn = {0370-2693}
    }

@Article{Belinskii:1982,
    title = {A general solution of the {E}instein equations with a time singularity},
    author = {V.A.~Belinskii and I.M.~Khalatnikov and E.M.~Lifshitz},
    journal = {Adv. Phys.},
    volume = {31},
    number = {6},
    year = {1982},
    pages = {639-667},
    doi = {10.1080/00018738200101428}
    }

@article{Misner:1969,
  title = {Mixmaster Universe},
  author = {C.W.~Misner},
  journal = {Phys. Rev. Lett.},
  volume = {22},
  issue = {20},
  pages = {1071--1074},
  numpages = {0},
  year = {1969},
  month = {May},
  publisher = {American Physical Society},
  doi = {10.1103/PhysRevLett.22.1071},
  url = {https://link.aps.org/doi/10.1103/PhysRevLett.22.1071}
}

@Article{AnHeShen:2025,
      title = {Stability of Big Bang singularity for the {E}instein-{M}axwell-scalar field-{V}lasov system in the full strong sub-critical regime}, 
      author = {X.~An and T.~He and D.~Shen},
      year = {2025},
      note = {preprint [arXiv:2507.18585v2]},
}

@article{Hawking:1967,
 ISSN = {00804630},
 URL = {http://www.jstor.org/stable/2415769},
 author = {S.W.~Hawking},
 journal = {Proceedings of the Royal Society of London. Series A, Mathematical and Physical Sciences},
 number = {1461},
 pages = {187--201},
 publisher = {The Royal Society},
 title = {The Occurrence of Singularities in Cosmology. III. Causality and Singularities},
 urldate = {2025-09-01},
 volume = {300},
 year = {1967}
}

@article{Hawking:1970,
author = {S.W.~Hawking and R.~Penrose},
title = {The singularities of gravitational collapse and cosmology},
journal = {Proceedings of the Royal Society of London. A. Mathematical and Physical Sciences},
volume = {314},
number = {1519},
pages = {529-548},
year = {1970},
doi = {10.1098/rspa.1970.0021},

URL = {https://royalsocietypublishing.org/doi/abs/10.1098/rspa.1970.0021},
eprint = {https://royalsocietypublishing.org/doi/pdf/10.1098/rspa.1970.0021}
}

@article{Chrusciel:2003jj,
    author = "P.T.~Chrusciel and K.~Lake",
    title = "{Cauchy horizons in Gowdy space-times}",
    eprint = "gr-qc/0307088",
    archivePrefix = "arXiv",
    doi = "10.1088/0264-9381/21/3/010",
    journal = "Class. Quant. Grav.",
    volume = "21",
    pages = "S153--S170",
    year = "2004"
}

@article{ringstrom2005,
author = {H.~Ringstr\"{o}m},
title = {{Curvature blow up on a dense subset of the singularity in $\Tbb^3$-Gowdy}},
journal = {Journal of Hyperbolic Differential Equations},
volume = {02},
number = {02},
pages = {547-564},
year = {2005},
doi = {10.1142/S021989160500052X},

URL = { 
    https://doi.org/10.1142/S021989160500052X
    },
}

@article{Ringstrom2006,
author = {H.~Ringstr\"{o}m},
title = {Existence of an asymptotic velocity and implications for the asymptotic behavior in the direction of the singularity in {$T^3$-Gowdy}},
journal = {Communications on Pure and Applied Mathematics},
volume = {59},
number = {7},
pages = {977-1041},
doi = {https://doi.org/10.1002/cpa.20105},
url = {https://onlinelibrary.wiley.com/doi/abs/10.1002/cpa.20105},
eprint = {https://onlinelibrary.wiley.com/doi/pdf/10.1002/cpa.20105},
year = {2006}
}

@article{Weaver2000,
    author = {M.~Weaver},
    title = {{Dynamics of magnetic Bianchi VI 0-cosmologies}},
    journal = {Class. Quantum Grav.},
    year = {2000}
}

@article{Ringstrom2000,
    author = {H.~Ringstr\"{o}m},
    title = {{Curvature blow up in Bianchi VIII and IX vacuum spacetimes}},
    journal = {Class. Quantum Grav.},
    year = {2000}
}

@article{Ringstrom2001,
    author = {H.~Ringstr\"{o}m},
    title = {{The Bianchi IX attractor}},
    journal = {Ann. Henri Poincar\'e},
    year = {2001}
}

@article{Beguin2010,
    author = {F.~B\'{e}guin},
    title = {{Aperiodic oscillatory asymptotic behavior for some Bianchi spacetimes}},
    journal = {Class. Quantum Grav.},
    year = {2010}
}

@article{Liebscher2013,
    author = {S.~Liebscher and A.D.~Rendall and S.B.~Tchapnda},
    title = {{Oscillatory singularities in Bianchi models with magnetic fields}},
    journal = {Ann. Henri Poincar\'e},
    year = {2013}
}

@article{Beguin2023,
    author = {F.~B\'{e}guin and T.~Dutilleul},
    title = {{Chaotic dynamics of spatially homogeneous spacetimes}},
    journal = {Comm. Math. Phys.},
    year = {2023}
}

@article{Yau1978,
    author = {S.T.~Yau},
    title = {{On the Ricci Curvature of a Compact Kähler Manifold and the Complex Monge-Ampére Equation, I*}},
    journal = {Communications on Pure and Applied Mathematics},
    year = {1978},
    volume = {31},
    issue = {3},
    pages = {339-411}
}

@article{GREEN1984,
title = {{Anomaly cancellations in supersymmetric D = 10 gauge theory and superstring theory}},
journal = {Physics Letters B},
volume = {149},
number = {1},
pages = {117-122},
year = {1984},
issn = {0370-2693},
doi = {https://doi.org/10.1016/0370-2693(84)91565-X},
url = {https://www.sciencedirect.com/science/article/pii/037026938491565X},
author = {M.B.~Green and J.H.~Schwarz}
}

@article{CANDELAS1985,
title = {{Vacuum configurations for superstrings}},
journal = {Nuclear Physics B},
volume = {258},
pages = {46-74},
year = {1985},
issn = {0550-3213},
doi = {https://doi.org/10.1016/0550-3213(85)90602-9},
url = {https://www.sciencedirect.com/science/article/pii/0550321385906029},
author = {P.~Candelas and G.T.~Horowitz and A.~Strominger and E.~Witten}
}

@article{WITTEN1995,
title = {String theory dynamics in various dimensions},
journal = {Nuclear Physics B},
volume = {443},
number = {1},
pages = {85-126},
year = {1995},
issn = {0550-3213},
doi = {https://doi.org/10.1016/0550-3213(95)00158-O},
url = {https://www.sciencedirect.com/science/article/pii/055032139500158O},
author = {E.~Witten}
}

@Article{Ringstrom2025,
      title = {{Local existence theory for a class of CMC gauges for the Einstein-non-linear scalar field equations}}, 
      author = {H.~Ringstr\"{o}m},
      year = {2025},
      note = {preprint [arXiv:2509.14110]},
}

@article{Marshall:2025,
    author = {E.~Marshall},
    title = "{Mixmaster fluids near the big bang}",
    eprint = "2512.11375",
    archivePrefix = "arXiv",
    primaryClass = "gr-qc",
    doi = "10.1103/86cf-6yct",
    journal = "Phys. Rev. D",
    volume = "113",
    number = "2",
    pages = "024053",
    year = "2026"
}

@Article{Dong2026,
      title = {{The past stability of Kasner singularities for the (3+1)-dimensional Einstein vacuum spacetime under polarized $U$(1)-symmetry}}, 
      author = {K.~Dong},
      year = {2026},
      note = {preprint [arXiv:2601.06957]},
}

@Article{Franco-Grisales2025,
      title = {{Developments of initial data on big bang singularities for the Einstein-nonlinear scalar field equations}}, 
      author = {A.~Franco-Grisales},
      year = {2025},
      note = {preprint [arXiv:2409.17065v2]},
}

@article{Isenberg_2002,
doi = {10.1088/0264-9381/19/21/305},
url = {https://doi.org/10.1088/0264-9381/19/21/305},
year = {2002},
month = {oct},
publisher = {},
volume = {19},
number = {21},
pages = {5361},
author = {J.~Isenberg and V.~Moncrief},
title = {{Asymptotic behaviour in polarized and half-polarized $U(1)$ symmetric vacuum spacetimes}},
journal = {Classical and Quantum Gravity}
}

@Article{ringstrom2026,
      title = {{Complete asymptotics in the formation of quiescent big bang singularities}}, 
      author = {A.~Franco-Grisales and H.~Ringstr\"{o}m},
      year = {2026},
      note = {preprint [arXiv:2602.02373]},
}

\end{document}